\documentclass{article}
\usepackage{xcolor}
\usepackage{hyperref}
\usepackage{bbm}
\usepackage{graphicx}
\usepackage[utf8]{inputenc}
\usepackage[T1]{fontenc}
\usepackage{amsmath,amsthm,amssymb}
\usepackage{todonotes}
\usepackage[margin=1in]{geometry}

\newtheorem{mthm}{Main Theorem}
\newtheorem{lemma}{Lemma}
\newtheorem{theorem}{Theorem}
\newtheorem{remark}{Remark}
\newtheorem{proposition}{Proposition}

\newtheorem{definition}{Definition}
\newtheorem{corollary}{Corollary}
\newtheorem{question}{Question}
\newtheorem{problem}{Problem}

\def\<#1>{\left\langle #1 \right\rangle}

\newcommand{\Apart}[1]{\ensuremath{\mathcal{A}^{#1}}}
\newcommand{\Bpart}[1]{\ensuremath{\mathcal{B}^{#1}}}
\newcommand{\w}{\ensuremath{\mathrm{w}}}

\newcommand{\FP}{\ensuremath{\Pi}}
\newcommand{\axisDir}{\ensuremath{\mathrm{A}}}
\newcommand{\angleFP}{\ensuremath{\theta}}

\renewcommand{\a}{\ensuremath{\mathrm{a}}}

\newcommand{\V}{\ensuremath{\mathcal{V}}}

\newcommand{\ci}{\ensuremath{\mathrm{i}\,}}
\newcommand{\qi}{\ensuremath{\mathbbm{i}}}
\newcommand{\qj}{\ensuremath{\mathbbm{j}}}
\newcommand{\qk}{\ensuremath{\mathbbm{k}}}

\newcommand{\N}{\ensuremath{\mathbbm{N}}}
\newcommand{\Z}{\ensuremath{\mathbbm{Z}}}
\newcommand{\R}{\ensuremath{\mathbbm{R}}}
\newcommand{\C}{\ensuremath{\mathbbm{C}}}
\newcommand{\Q}{\ensuremath{\mathbbm{Q}}}
\renewcommand{\H}{\ensuremath{\mathbbm{H}}}
\renewcommand{\S}{\ensuremath{\mathbbm{S}}}

\renewcommand{\Re}{\ensuremath{\mathrm{Re}}}
\newcommand{\Imc}{\ensuremath{\mathrm{Im}_{\C}}}
\newcommand{\Imq}{\ensuremath{\mathrm{Im}_{\H}}}

\graphicspath{{figures/}}
\author{\normalsize Alexander I. Bobenko, Tim Hoffmann, Andrew O. Sageman-Furnas}
\date{\normalsize \today}
\title{Compact Bonnet Pairs: isometric tori with the same curvatures}

\AtEndDocument{\bigskip{\footnotesize%
		\noindent Alexander I. Bobenko: \texttt{bobenko@math.tu-berlin.de} \\
		\textsc{Institute of Mathematics, Secretary MA 8-4, \\ Technical University of Berlin, 10623 Berlin, Germany} \par 
		\addvspace{\medskipamount}
		\noindent Tim Hoffmann: \texttt{tim.hoffmann@ma.tum.de} \\
		\textsc{Department of Mathematics, \\ Technical University of Munich, 85748 Garching, Germany} \par
		\addvspace{\medskipamount}
		\noindent Andrew O. Sageman-Furnas: \texttt{asagema@ncsu.edu} \\
		\textsc{Department of Mathematics, \\ North Carolina State University, Raleigh, NC 27695, USA} \par
}}
\hypersetup{
 pdfauthor={},
 pdftitle={},
 pdfkeywords={},
 pdfsubject={},
 pdfcreator={Emacs 25.1.1 (Org mode 8.3.4)}, 
 pdflang={English},
 final
}

\begin{document}
\maketitle

\begin{abstract}
We explicitly construct a pair of immersed tori in three dimensional Euclidean space that are related by a mean curvature preserving isometry. These Bonnet pair tori are the first examples of compact Bonnet pairs.

This resolves a longstanding open problem on whether the metric and mean curvature function determine a unique smooth compact immersion.

Moreover, we prove these isometric tori are real analytic. This resolves a second longstanding open problem on whether real analyticity of the metric already determines a unique compact immersion.

Our construction uses the relationship between Bonnet pairs and isothermic surfaces. The Bonnet pair tori arise as conformal transformations of an isothermic torus with one family of planar curvature lines. We classify such isothermic tori in our companion paper~\cite{short-isothermic-planar}.

The above approach stems from computational investigations of a $5\times7$ quad decomposition of a torus using a discrete differential geometric analog of isothermic surfaces and Bonnet pairs.

\end{abstract}

\tableofcontents


\section{Introduction}
\label{sec:intro}
\begin{figure}[tbh!]
\includegraphics[width= \textwidth]{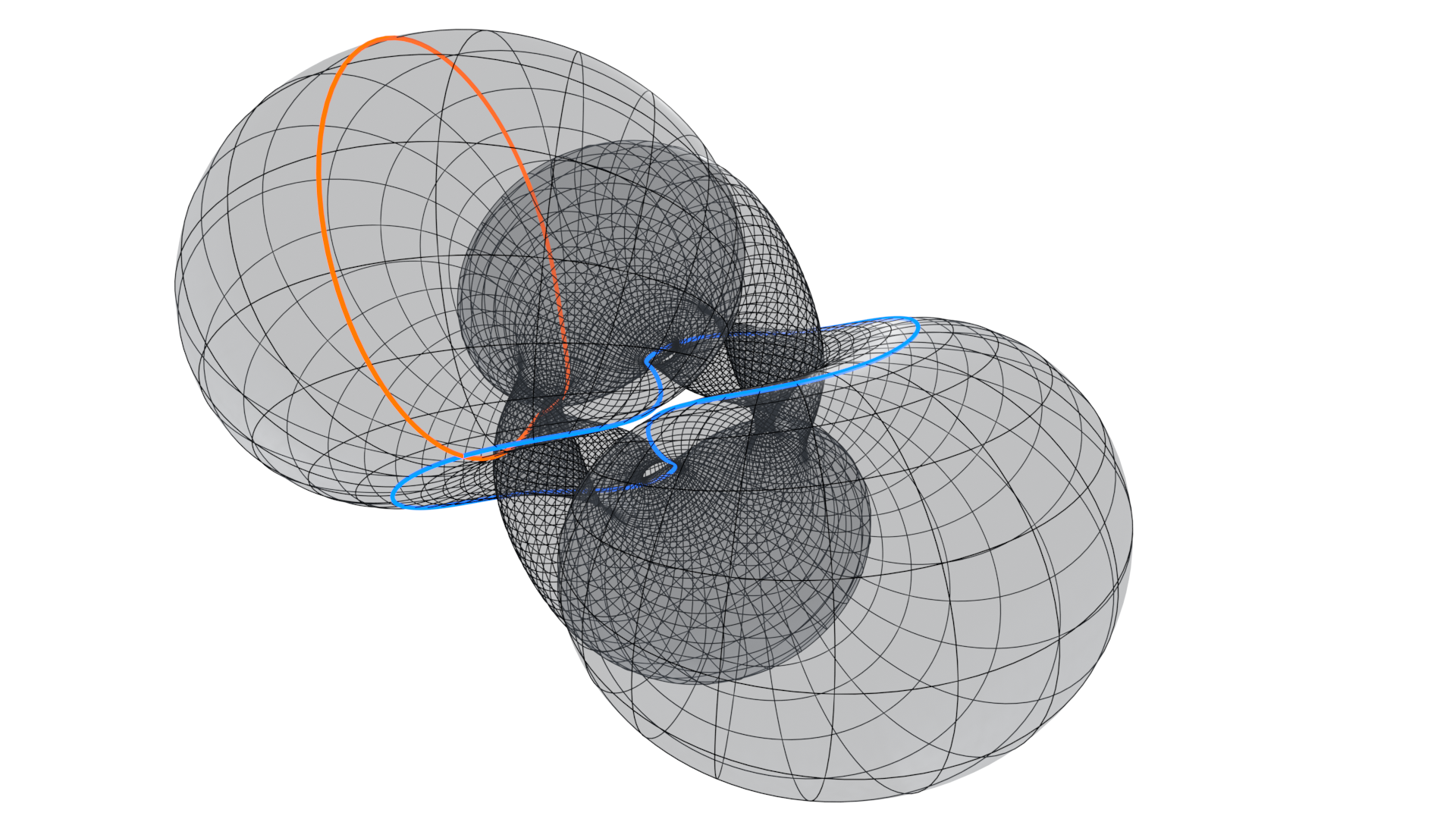}\\
\includegraphics[width= \textwidth]{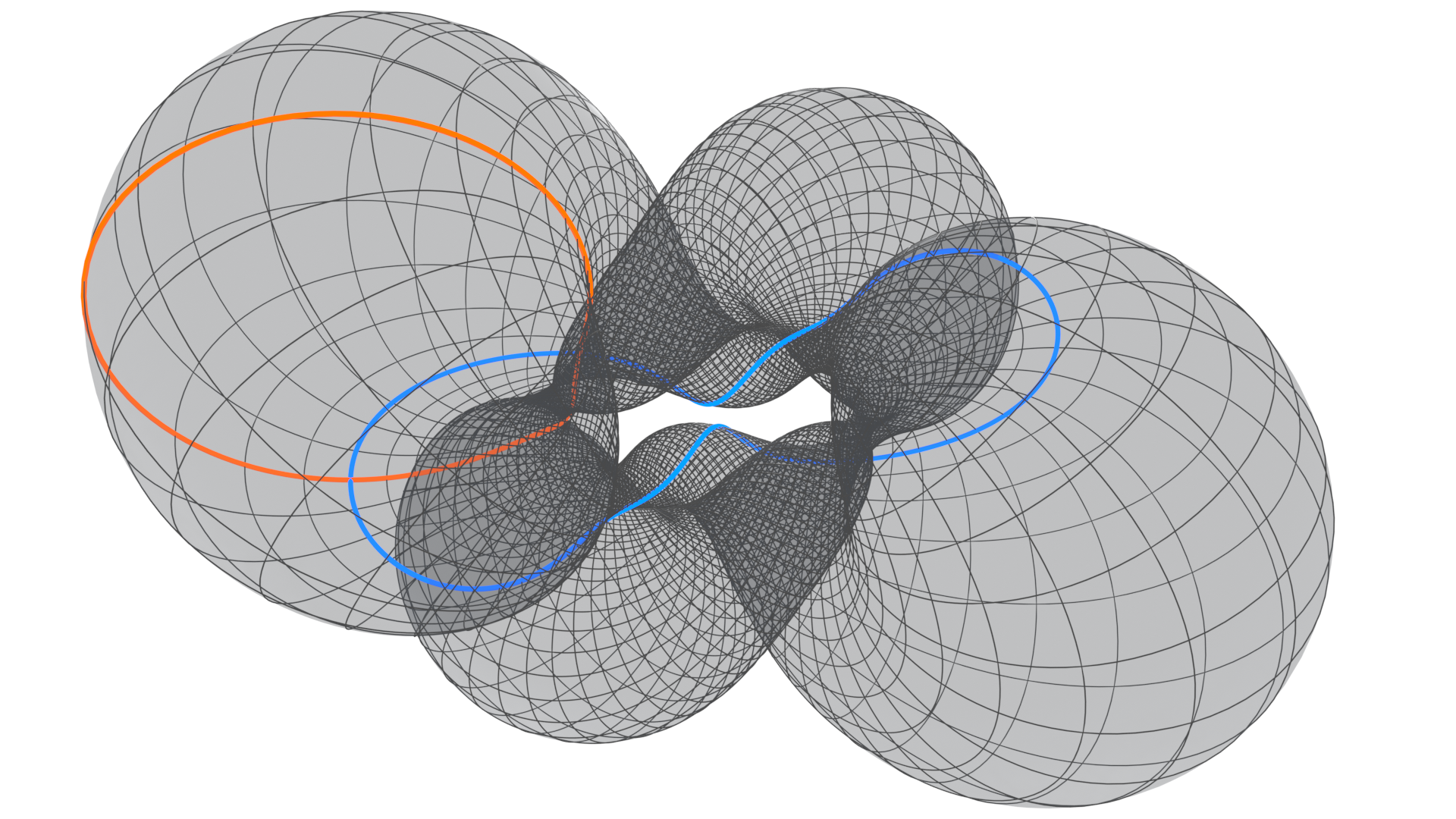}
	\caption{A compact Bonnet pair: two non-congruent immersed real analytic tori that are related by a mean curvature preserving isometry. Corresponding generators are shown on each surface in orange and blue. Note that the two large bubbles on the left are closer together than the corresponding bubbles on the right, and both surfaces have $180^\circ$ rotational symmetry. These Bonnet pair tori arise as conformal transformations from an isothermic torus with one family of planar curvature lines. The orange generators come from the planar curves and are therefore congruent, while the blue generators are not congruent.}
	\label{fig:teaserNonsphericalPlusMinus}
\end{figure}

A smooth orientied surface immersed in three-dimensional Euclidean space is analytically described by a metric and second fundamental form. The latter is a symmetric bilinear form with respect to the metric, and its determinant and trace are the Gauss curvature and twice the mean curvature, respectively. Two immersions are \emph{congruent} if they are related by an orientation-preserving ambient isometry, i.e., a rigid motion. The classical Bonnet theorem is that a metric and second fundamental form satisfying the Gauss--Codazzi compatibilty equations determine an immersion that, up to congruence, is unique.

When is a reduced set of geometric data sufficient for uniqueness?

The metric determines the Gauss curvature, so in 1867 Bonnet asked if a surface can instead be characterized by a metric and mean curvature function~\cite{Bonnet}. Generically, the answer is yes, but there are important exceptions. These include constant mean curvature surfaces, like the textbook example (see \cite{Hoffman-Karcher}) of the isometry between the helicoid and catenoid minimal surfaces, both of which have vanishing mean curvature.

In 1981, Lawson and Tribuzy proved that for each smooth metric and non-constant mean curvature function there exist at most two compact smooth immersions~\cite{LawTri}. Moreover, they showed there is at most one immersion of a compact surface with genus zero (see Corollary \ref{Bonnet_spheres} below). They emphasized the following remained unanswered:

\begin{problem}[Global Bonnet Problem]
Do there exist two non-congruent compact smooth ($C^\infty$) immersions in three-dimensional Euclidean space that are related by an isometry with the same mean curvature at corresponding points?
\end{problem}

On the other hand, sometimes a compact immersion is uniquely determined by the metric alone. In 1927 Cohn-Vossen proved that two isometric compact analytic ($C^\omega$) surfaces that are convex
must be congruent~\cite{CohnVossen1927}. A similar statement can also hold for non-convex analytic surfaces. For example, an analytic surface isometric to a circular torus of revolution must be congruent to it (a special case of A.D. Alexandrov's uniqueness result on tight analytic immersions~\cite{zbMATH03032753}). 

In 1929, Cohn-Vossen constructed two isometric compact 
surfaces that are nowhere locally congruent (compare to  Remark~\ref{rem:IsometricFromFlatExamples}) but had to drastically reduce the regularity from analytic to class $C^2$~\cite{CohnVossen1929}. In 2010, Marcel Berger highlighted that Cohn-Vossen's analytic non-convex metric uniqueness question remains open. It is the first unsolved problem Berger states in the section ``What we don't entirely know how to do for surfaces'' of his beautiful book Geometry Revealed~\cite[Section VI.9, pp. 386--387]{BergerBook2010}.

\begin{problem}[Cohn-Vossen--Berger Problem]
	\label{problem:cvb}
	Do there exist two isometric compact immersions in Euclidean three-space that are analytic ($C^\omega$) but not related by an ambient isometry?
\end{problem}

Note the Cohn-Vossen--Berger problem asks for uniqueness up to ambient isometry, i.e., rigid motions and reflections.

\begin{remark}
	\label{rem:IsometricFromFlatExamples}
	The $C^\infty$ version of Problem~\ref{problem:cvb} is also mostly unexplored. As far as we know, the only known examples of a pair of isometric compact $C^\infty$ surfaces not related by ambient isometry are constructed by locally altering a smooth compact surface with a flat (zero Gaussian curvature) region. One smoothly attaches a bump either outward or inward, respectively, in place of the flat region, see~\cite[Sec. 5-2, Fig. 5-1]{doCarmoBook} and ~\cite[Chapter 12, pp. 209--211]{SpivakVol5}). The isometry is thus a congruence away from the locally altered regions.
\end{remark}
We explicitly construct genus one, i.e., tori, examples to the Global Bonnet Problem. A numerical example is shown in Figure~\ref{fig:teaserNonsphericalPlusMinus}.

\begin{mthm}
	\label{thm:main1}
	There exist two non-congruent smooth tori in $\R^3$ that are related by a mean curvature preserving isometry.
\end{mthm}

We prove there are uncountably many such pairs, as their construction has a functional parameter. Moreover, our methods lead to immersions that are real analytic and generically lead to pairs that are not related by an ambient isometry. We therefore simultaneously resolve both the Global Bonnet Problem and the Cohn-Vossen--Berger Problem.

\begin{mthm}
	\label{thm:main2}
	There exist two isometric analytic tori in $\R^3$ not related by an ambient isometry.
\end{mthm}

Both Main Theorems follow directly from Theorem~\ref{thm:Bonnet_pair_two} in Section~\ref{sec:Bonnet_two_surfaces}.

\begin{remark}
Since our main focus is the Global Bonnet Problem, the tori we construct correspond via a mean curvature preserving isometry that is nowhere locally a congruence. In other words, every pair of corresponding neighborhoods are non-congruent. This is in stark contrast to the only previously known $C^\infty$ examples of isometric compact immersions not related by an ambient isometry, as in Remark~\ref{rem:IsometricFromFlatExamples}.
\end{remark}

\subsection{Background}
There are various problems about the existence and uniqueness of immersions with some prescribed data drawn from the metric and second fundamental form. In \cite{Cartan4} Cartan gives a good overview of such local problems including immersions with prescribed metric and either second fundamental form~\cite{Cartan3} or Weingarten operator ~\cite{Bryant2001}. The existence of an isometric immersion, i.e., with arbitrary prescribed metric, is a challenging question with vast, complicated literature~\cite{HanHongBook, Gromov2017}. Our focus is on uniqueness up to congruence of surfaces already in $\R^3$, with an emphasis on Bonnet's problem.

A generic surface is locally determined by its metric and mean curvature function (see \cite{Car1942}). Bonnet knew that there are three exceptional cases~\cite{Bonnet}: constant mean curvature surfaces, Bonnet families, and Bonnet pairs.
\begin{enumerate}
	\item {\bf Constant mean curvature surfaces.} Every constant mean curvature surface is part of a one-parameter associated family of isometric surfaces with the same constant mean curvature.
	
	Global theories for constant mean curvature surfaces are an active field of research. Techniques for constructing complete and embedded examples span from integrable systems~\cite{PS1989,Hitchin1990,Bob_CMC1991} and (generalized) Weierstrass representations \cite{HoffmanTraizetWhite2016, DPW} to geometric analysis~\cite{Kapouleas1990,KorevaarKusner93} . Recently, these approaches are starting to be blended together \cite{Traizet2020}.
	
	As far as we know, it remains an open question if the associated family of a compact constant mean curvature surface contains a second, non-congruent compact immersion.
	
	\item {\bf Bonnet families.} There exists a finite dimensional space of non-constant mean curvature surfaces that exhibit a one-parameter family of isometric deformations preserving principal curvatures. The local classification of such families was obtained in \cite{Haz1887, Car1942, Che1985}. The global classification of Bonnet families was obtained in \cite{BE_Painleve2000} using techniques from the theory of Painlev\'e equations and isomonodromic deformations. In particular it was shown that surfaces in Bonnet families cannot be compact.
	
	\item {\bf Bonnet pairs.} A Bonnet pair is two non-congruent immersions $f^+$ and $f^-$ with the same metric and mean curvature function. If a third immersion exists that is isometric to the other two and has the same mean curvature function, then an entire one-parameter family must exist. These families are further classified depending on whether the mean curvature is constant or not (see above).
	
	Global results for compact Bonnet pairs have focused on uniqueness, i.e., non-existence of a pair. For example, a compact surface of revolution is uniquely determined by its metric and mean curvature function~\cite{RH1990}. In 2010, Sabitov published a paper claiming that compact Bonnet pairs cannot exist for every genus \cite{Sab2010}. In 2012, however, he retracted his claims and published a second paper with sufficient conditions for uniqueness ~\cite{Sab2012}. The geometry of these sufficient conditions has been further clarified in \cite{JMN2018}.	
\end{enumerate}

We study the Global Bonnet Problem by investigating Bonnet pairs. We build from Kamberov, Pedit, and Pinkall's local classification of Bonnet pairs, using a quaternionic function theory, in terms of isothermic surfaces~\cite{KPP}. Iso\-thermic surfaces are characterized by exhibiting conformal, curvature line coordinates away from umbilic points. Isothermic surfaces have a Christoffel dual surface $f^*$ with parallel tangent planes and inverse metric. The differentials of the Bonnet pair surfaces $f^\pm$ are written in terms of the isothermic surface $f$, its dual surface $f^*$, and a real parameter $\epsilon$ as follows.
\begin{align}
	\label{eq:dfpm}
	df^\pm = (\pm \epsilon - f)df^*(\pm \epsilon + f)
\end{align}
The action of $f \pm \epsilon$ on $df^*$ is a rotation and scaling, implying that $f^+$ and $f^-$ are conformally equivalent to $f^*$ and therefore also $f$.

Here, we construct compact Bonnet pairs that are tori.

\subsection{Outline of the construction}
We explicitly construct examples of compact Bonnet pairs of genus one, i.e., both surfaces $f^+$ and $f^-$ are tori. Moreover, they are analytic. The construction uses the above relationship to isothermic surfaces. For an isothermic surface $f(u,v)$ with conformal curvature line coordinates $u,v$, \eqref{eq:dfpm} allows us to study the period problem of $f^+(u,v)$ and $f^-(u,v)$ directly. An immediate necessary condition is that the isothermic surface $f$ must be a torus. We therefore find an appropriate isothermic torus $f$ that integrates using \eqref{eq:dfpm} to the Bonnet pair tori $f^\pm$.

\paragraph{Isothermic tori with one family of planar curvature lines.}
Our essential geometric observation is that the periodicity conditions drastically simplify when the isothermic torus has one family of curvature lines that are planar. We make this planar assumption for the curvature $u$-lines, i.e., $f(u,v_0)$ lies in a plane for each constant $v_0$.

In 1883, Darboux used complex analytic methods to locally classify isothermic surfaces with one family of planar curvature lines~\cite{Darboux1883,Darboux1896}. His choice of real reduction does not include tori. In our companion paper~\cite{short-isothermic-planar}, we classify the tori found in the second real reduction. We restate the key results in Sections~\ref{sec:periodicty-conditions-planar-u} and ~\ref{sec:periodicty-conditions-planar-u-spherical-v}. The geometry of these isothermic tori are key to constructing Bonnet pair tori.

In particular, an isothermic surface $f(u,v)$ with closed planar $u$-lines has a functional freedom in its construction. Given $f$, there exists a mapping of all planar $u$-lines into a common plane such that the family of planar curves is holomorphic with respect to $u + \ci \w$ for a \emph{reparametrization function} $\w(v)$. Conversely, given the holomorphic family of closed planar curves, choosing a reparametrization function $\w(v)$ gives $f$. The surface depends on the choice of reparametrization function $\w(v)$. Some choices close $f$ into a torus.

\paragraph{Constructing Bonnet tori.}
The Kamberov--Pedit--Pinkall construction \eqref{eq:dfpm} is a formula for 1-forms. It allows us to analyze when the resulting Bonnet pair surfaces $f^+$ and $f^-$ are closed.
\begin{itemize}
	\item The isothermic surface $f$ must be a torus for $f^\pm$ to be tori.
\end{itemize}
	Expanding \eqref{eq:dfpm} gives
	\begin{align}
		\label{eq:dfpmEpsExpansion}
		df^\pm = -f df^* f + \epsilon^2 df^* \pm \epsilon (df^* f - f df^*).
	\end{align}
In Section~\ref{sec:periodicty-conditions-planar-u} we show that if $f$ has planar curvature $u$-lines the periodicity conditions are drastically simplified.
	
	If $f$ is a torus with planar curvature $u$-lines then:
\begin{itemize}

	\item $f^*$ is a torus. Thus, the $\epsilon^2$ term of \eqref{eq:dfpmEpsExpansion} is closed.
	\item $(f^{-1})^*$ is a torus. Thus, the $\epsilon^0$ term of \eqref{eq:dfpmEpsExpansion} is closed.
	
	The $\epsilon^1$ term is not automaticlly closed, but simplifies as follows:
	
	\item The $u$-period of $df^* f - f df^*$ vanishes.
	\item The $v$-period of $df^* f - f df^*$ is independent of $u$.
\end{itemize}
In short, when $f$ is an isothermic surface with one family of planar curvature lines the Bonnet surfaces $f^\pm$ are tori if and only if
\begin{enumerate}
	\item[i.] $f$ is a torus and
	\item[ii.] the $\R^3$-valued $v$-period vanishes.
\end{enumerate}
In Theorem~\ref{thm:isothermicCylinderBonnetPeriodicityConditions} we show these two conditions reduce even more, see Figure~\ref{fig:periodicity-conditions-outline}.
\begin{enumerate}
	\item[i.] (Rationality condition) For the isothermic surface $f$ to be a torus the reparametrization function $\w(v)$ must be periodic. In this case, $f$ is generated by the rotation of a fundamental piece around an axis. Thus, $f$ is a torus when the rotation angle is a rational multiple of $\pi$. 
	\item[ii.] (Vanishing axial part) The $\R^3$-valued integral for the $v$-period reduces to an $\R$-valued period that must vanish.
\end{enumerate}

To construct Bonnet tori, we show how to choose a reparametrization function $\w(v)$ so these two conditions are simultaneously satisfied. The challenge is to make both conditions analytically tractable.

In Section~\ref{sec:periodicty-conditions-planar-u-spherical-v} we consider the special case of isothermic surfaces with one family of closed planar and one family of spherical curvature lines. The spherical curvature lines are governed by a second elliptic curve. Both the rationality condition and the vanishing of the real period (conditions i. and ii. above) are expressed in terms of elliptic integrals, see Theorem~\ref{thm:bonnetPeriodicityAsAbelianIntegrals}. We show that both conditions can be satisfied, proving the existence of real analytic Bonnet pair tori in Theorem~\ref{thm:compactBonnetPairsSphericalExistence}.

We remark that isothermic tori with two families of spherical curvature lines were studied by Bernstein in \cite{Bernstein2001} in relationship to the Bonnet problem, but compact Bonnet pairs were not constructed.

In Section~\ref{sec:compact-bonnet-pairs} we prove the existence of more general examples of real analytic Bonnet pair tori. We retain the planar curvature $u$-lines of the isothermic torus $f$ but analytically perturb the reparametrization function $\w(v)$ so that the $v$-lines are no longer spherical, see Theorem~\ref{thm:Bonnet_pair_two}. This perturbation retains a real analytic functional freedom. Figure~\ref{fig:teaserNonsphericalPlusMinus} shows a numerical example of these more general analytic Bonnet pair tori.

Both Main Theorems are immediate corollaries of Theorem~\ref{thm:Bonnet_pair_two} in Section~\ref{sec:Bonnet_two_surfaces}.

\subsection{Discovery using Discrete Differential Geometry}
At last, we would like to mention the role of Discrete Differential Geometry (DDG) in the discovery of compact Bonnet tori. DDG aims at the development of structure preserving discrete equivalents of notions and methods of classical differential geometry. 
 Discrete isothermic surfaces introduced in \cite{DiscreteIsothermic96} are a well-studied example, highlighting the link between geometry and integrable systems. It was recently observed that a discrete analog of the Kamberov--Pedit--Pinkall construction \eqref{eq:dfpm} allows to define discrete Bonnet pairs \cite{HSFW17}. This led to numerical experiments to search for Bonnet tori on an extremely coarse torus, see Section~\ref{sec:discrete-theory}.

A careful study of the discrete isothermic torus, which led to a discrete compact Bonnet pair, showed, in particular, that one of its family of curvature lines was planar. This observation initiated our work on the present paper. It is remarkable that a very coarse $5\times7$ torus has the essential features of the corresponding smooth object. This exemplifies the importance of a structure preserving discrete theory.

Data from the figures are available in the Discretization for Geometry and Dynamics Gallery \url{https://www.discretization.de/gallery/}  .

\paragraph{Acknowledgements}
We thank Max Wardetzky for inspiring discussions that led to the computational discovery of discrete compact Bonnet pairs. This research was supported by the DFG Collaborative Research Center TRR 109 ``Discretization in Geometry and Dynamics''.


\section{Differential equations of surfaces}
\label{sec:diff-eq-surfaces}
\subsection{Conformally parametrized surfaces}
\label{sec:conf-parametrized}

Let ${\cal F}$ be a smooth orientable surface in 3-dimensional Euclidean space.
The Euclidean metric induces a metric $\Omega$ on this surface, which
in turn generates the complex structure of a Riemann surface $\cal R$.
 Under
such a parametrization, which is called  conformal, the surface
$\cal F$ is given by an immersion
$$
f=(f_1, f_2, f_3) : {\cal R} \rightarrow \R^3,
$$
and the metric is conformal: $\Omega=e^{2h}\, dzd\bar{z}$, where $z$ is
a local coordinate on $\cal R$. Denote by $u$ and $v$ its real and imaginary parts: $ z=u+\ci v$.

The tangent vectors $f_u, f_v$  together with the unit normal $n:{\cal R}\to \mathbb{S}^2$ define a conformal moving frame on the surface:
\begin{align*}
\<f_u, f_u>=\<f_v, f_v>&=e^{2h},\ \<f_u,f_v>=0,& \\
\<f_u, n>=\<f_v,n>&=0,\ \<n,n>=1.&                         
\end{align*}

Let us use the complex operator   $\partial_z=\frac{1}{2}(\partial_u-\ci \partial_v)$ and the complexified inner product to introduce the quadratic Hopf differential $Qdz^2$ by
$$
Q=\<f_{zz},n>.
$$
The first and second fundamental forms of the surface are given by
\begin{eqnarray}
\label{eq:fund.forms}
\< df, df >&=&e^{2h} dz d\bar{z},\\
-\< df, dn >&=&He^{2h} dz d\bar{z}+Qdz^2+\bar{Q}d\bar{z}^2 \nonumber
\end{eqnarray}
where $H=\frac{1}{2}(k_1+k_2)$ is the mean curvature (average of the principal curvatures) of the surface
\begin{equation}			\label{eq:mean_curvature}
\<f_{z\bar{z}},n>=\frac{1}{2}He^{2h}.
\end{equation}
We will also consider the third fundamental form $\<dn,dn>$.

The Gaussian curvature is given by 
$$
K=k_1k_2=H^2-4Q\bar{Q}e^{-4h},
$$
and is known to be determined by the metric only.

A point $P$ on the surface $\cal F$ is called {\em umbilic} if the principal curvatures coincide $k_1(P)=k_2(P)$. The Hopf differential vanishes $Q(P)=0$ exactly at umbilic points.

The conformal frame satisfies the following complex frame equations:
\begin{eqnarray}
\left(\begin{array}{c}
f_z\\
f_{\bar{z}}\\
n
\end{array}\right)_{z}=
\left(\begin{array}{ccc}
2h_z & 0 & Q\\
0 & 0 & \frac{1}{2}He^{2h}\\
-H & -2e^{-2h}Q & 0
\end{array}\right)
\left(\begin{array}{c}
f_z\\
f_{\bar{z}}\\
n
\end{array}\right),  \label{eq:conformal_frame_z}
\\
\left(\begin{array}{c}
f_z\\
f_{\bar{z}}\\
n
\end{array}\right)_{\bar{z}}=
\left(\begin{array}{ccc}
0 & 0 &  \frac{1}{2}He^{2h}\\
0 & 2h_{\bar{z}} &  \bar{Q}\\
-2e^{-2h}\bar{Q} & -H & 0
\end{array}\right)
\left(\begin{array}{c}
f_z\\
f_{\bar{z}}\\
n
\end{array}\right).   \label{eq:conformal_frame_zbar}
\end{eqnarray}                                    
Their compatibility conditions, known as the Gauss--Codazzi equations, have the following form:
\begin{equation}        \label{complex version of Gauss-Codazzi}
\begin{array}{lrcl}
 {\rm Gauss\  equation}\quad &
    h_{z \bar z} + \frac{1}{4}\,
    H^2\,e^{2h} -  \vert Q\vert^2\,e^{-2h}
    & = & 0,
\\
 {\rm Codazzi\  equation}\quad &
    Q_{\bar z} & = & \frac{1}{2}\,H_z\,e^{2h}.
\end{array}
\end{equation}
These equations are necessary and sufficient for the existence of the corresponding
surface. The classical Bonnet theorem characterizes surfaces via the coefficients $e^{2h},Q,H$ of their fundamental forms.

\begin{theorem} {\bf (Bonnet theorem)}.                \label{Bonnet.t}
Given a metric $e^{2h}\, dzd\bar{z}$, a quadratic differential $Q\, dz^2$, and a mean curvature 
function $H$ on $\cal R$ satisfying the Gauss--Codazzi equations, there exists an
immersion
$$
f:\tilde{\cal R} \to \R^3
$$
with the fundamental forms (\ref{eq:fund.forms}). Here $\tilde{\cal R}$ is the
universal covering of $\cal R$. The immersion $f$ is unique up to Euclidean
motions in $\R^3$.
\end{theorem}

Generic surfaces are determined uniquely by the metric and the mean curvature function. This paper is devoted to the investigation of the exceptions, i.e., to surfaces which possess non-congruent isometric ``relatives'' with the same curvatures.

\subsection{Isothermic surfaces}
\label{sec:isothermic}
Isothermic surfaces play a crucial role in this paper.

A parametrization that is simultaneously conformal and curvature line is called {\em isothermic}.
In this case the preimages of the curvature lines are the lines
$u={\rm const}$ and $v={\rm const}$ on the parameter domain, where
$z=u+\ci v$ is a conformal coordinate. Equivalently a parametrization is isothermic if it is conformal and $f_{uv}$ lies in the tangent plane, i.e.,
\begin{equation}   \label{eq:f_uv-isothermic}
f_{uv}\in {\rm span}\{ f_u, f_v \}.
\end{equation}
A surface  is called {\em isothermic}
if it admits an isothermic parametrization. Isothermic surfaces are
divided by their curvature lines into infinitesimal squares. 

Written in terms of an isothermic coordinate $z$, the Hopf differential of an isothermic surface is real,
i.e., $Q(z, \bar{z})\in \R$.

The differential equations describing isothermic surfaces simplify in isothermic coordinates. The frame equations are as follows: 
\begin{eqnarray}
\left(\begin{array}{c}
f_u\\
f_v\\
n
\end{array}\right)_u=
\left(\begin{array}{ccc}
h_u & -h_v & k_1e^{2h}\\
h_v & h_u & 0\\
-k_1 & 0 & 0
\end{array}\right)
\left(\begin{array}{c}
f_u\\
f_v\\
n
\end{array}\right),
\label{eq:isothermicFrameEquationsDU}
\\
\left(\begin{array}{c}
f_u\\
f_v\\
n
\end{array}\right)_v=
\left(\begin{array}{ccc}
h_v & h_u &  0\\
-h_u & h_v & k_2e^{2h}\\
0 & -k_2 & 0
\end{array}\right)
\left(\begin{array}{c}
f_u\\
f_v\\
n
\end{array}\right).
\label{eq:isothermicFrameEquationsDV}
\end{eqnarray}
Here $k_1, k_2$ are the principal curvatures along the $u$ and $v$ curvature lines. The Hopf differential is
$$
Q=\frac{1}{4}e^{2h}(k_1-k_2).
$$ 
The first, second, and the third fundamental forms are given by
\begin{eqnarray}
\label{eq:fund.forms_isothermic}
\<df,df> & =& e^{2h} (du^2+dv^2), \nonumber\\
- \< df,dn >&=&e^{2h}(k_1du^2+k_2dv^2), \\
\<dn,dn>&=&e^{2h}(k_1^2 du^2+k_2^2 dv^2).\nonumber
\end{eqnarray}
The Gauss--Codazzi equations become
\begin{eqnarray}
& h_{uu}+h_{vv}+k_1k_2e^{2h}=0,  \label{eq:Gauss_isothermic}\\
& {k_2}_u=h_u(k_1-k_2),\quad {k_1}_v=h_v(k_2-k_1).\label{eq:Codazzi_isothermic}
\end{eqnarray}

Let $D\subset \C$ be a simply connected domain and $f:D\to \R^3$ be an isothermic immersion without umbilic points. Its differential is $df=f_u du+f_v dv$. An important property of an isothermic immersion is that the following form is closed.
\begin{equation}
\label{eq:isothermic_dual}
df^*:=e^{-2h}(f_u du -f_v dv).
\end{equation}
The corresponding immersion $f^*:D\to \R$, which is determined up to a translation, is also isothermic and is called the {\em (Christoffel) dual isothermic surface}. The relation (\ref{eq:isothermic_dual}) is an involution. Note that the dual isothermic surface is defined through one forms and the periodicity properties of $f:{\cal R}\to \R^3$ are not respected.


\section{The Bonnet problem}
\label{sec:the-problem}

The Bonnet Theorem \ref{Bonnet.t} characterizes surfaces via the coefficients
$e^{2h},Q, H$ of their fundamental forms. These coefficients are not independent and
are subject to the Gauss--Codazzi equations (\ref{complex version of Gauss-Codazzi}).
A natural question is whether some of these data are superfluous.

A Bonnet pair is two non-congruent isometric surfaces ${\cal F}^+$ and ${\cal F}^-$ with the same mean curvature at corresponding points.

In 1981, Lawson and Tribuzy proved there are at most two non-congruent isometric immersions of a compact surface with the same non-constant mean curvature function \cite{LawTri}. This led them to ask:

\vspace{0.2cm}
(Global Bonnet Problem)\emph{ Do there exist compact Bonnet pairs?}
\vspace{0.2cm}

Let ${\cal F}^+, {\cal F}^-\subset\R^3$ be a smooth Bonnet pair. As conformal immersions of the same Riemann surface,
$$
f^+:{\cal R}\to \R^3 \quad \text{and} \quad f^-:{\cal R}\to \R^3
$$
are described by the corresponding Hopf differentials $Q^+, Q^-$, their common metric
$e^{2h}\, dzd\bar{z}$ and their common mean curvature function $H$. Since the surfaces are not congruent, the Gauss--Codazzi equations immediately imply that their Hopf differentials are not equal $Q^+\not\equiv Q^-$.

\begin{proposition}                            \label{conditions_forQ_12}
Let $Q^+$ and $Q^-$ be the Hopf differentials of a Bonnet pair
$f^\pm:{\cal R}\to\R^3$. Then
\begin{equation}                        \label{eq:holQ}
Q_h=Q^+ -Q^-
\end{equation}
is a holomorphic quadratic differential  $Q_h\, dz^2$  on $\cal R$ and
\begin{equation}                        \label{eq:modQ}
\mid Q^+ \mid=\mid Q^- \mid.
\end{equation}
\end{proposition}

Due to (\ref{eq:modQ}) the umbilic points of ${\cal F}^+$ and
${\cal F}^-$ correspond. 

A holomorphic quadratic differential on a sphere vanishes identically $Q_h\equiv
0$. This implies $Q^+=Q^-$, and the non-existence of Bonnet spheres. This result and the following corollary were proven by Lawson and Tribuzy \cite{LawTri}.

\begin{corollary}                            \label{Bonnet_spheres}
There exist no Bonnet pairs of genus $g=0$.
\end{corollary}

For tori, the Riemann surface can be represented as a factor ${\cal R}=\C/{\cal L}$ with respect to a lattice ${\cal L}$. A non-vanishing holomorphic quadratic differential $Q_h dz^2$ is represented by a doubly-periodic holomorphic function $Q_h$, which must be a non-vanishing constant. Scaling the complex coordinate $z$ one can normalize so that  $Q_h=i$, in which case
\begin{equation}                        \label{eq:normalizationQ}
Q^+={1\over 2}(\alpha +i), \quad Q^-={1\over 2}(\alpha -i),
\end{equation}
where $\alpha:\C/{\cal L}\to\R$ is a smooth function.

\begin{proposition}                            \label{prop:pairs_GC}
Bonnet pairs of genus $g=1$ have no umbilic points. 
The Gauss--Codazzi equations of Bonnet tori with properly normalized (\ref{eq:normalizationQ}) Hopf differentials are
\begin{eqnarray}
4h_{z\bar{z}}+H^2 e^{2h}-(1+\alpha^2)e^{-2h}&=&0,     \label{eq:Gauss_Bonnet}\\
\alpha_{\bar{z}}&=&e^{2h} H_z.               \label{eq:Codazzi_Bonnet}
\end{eqnarray}
\end{proposition}

\begin{remark}\label{re:umbilic_divisor}
Generally, for any Bonnet pair ${\cal F}^\pm$ the set of umbilic points is isolated and coincides with the zero divisor of the quadratic differential $Q_h$ on ${\cal R}$, see \cite{Bob_Bonnet2008} for the smooth case and \cite{Sab2012} for the case of finite smoothness. 
\end{remark}

A direct analytic way to construct compact Bonnet pairs that are tori would be to find doubly-periodic solutions $h, H, \alpha:\C/{\cal L}\to\R$ of
(\ref{eq:Gauss_Bonnet},\ref{eq:Codazzi_Bonnet}) that integrate by (\ref{eq:conformal_frame_z},\ref{eq:conformal_frame_zbar}) to doubly-periodic frames, and finally to doubly periodic immersions $f^\pm:\C/{\cal L}\to\R^3$. We will not take this approach, since it seems hardly realizable.

Instead, we will find genus $g=1$ Bonnet pairs using a quaternionic description of surfaces and a deep relationship to isothermic surfaces.


\section{Characterization of Bonnet pairs via isothermic surfaces}
\label{sec:isothermic-characterization}

\subsection{Quaternionic description of surfaces}
\label{sec:quaternionic-desc}

We construct and investigate surfaces in $\R^3$ by analytic methods.
For this purpose it is convenient to rewrite the conformal frame equations
(\ref{eq:conformal_frame_z},\ref{eq:conformal_frame_zbar},\ref{eq:isothermicFrameEquationsDU},\ref{eq:isothermicFrameEquationsDV}) in terms of quaternions.
This quaternionic description is
useful for studying general curves and surfaces in 3- and
4-space, and particular special classes of surfaces
\cite{Bob_CMC1991, DPW, KPP, BFLPP}.

Let us denote the algebra of quaternions by $\H$, the multiplicative quaternion
group by ${\H_*}={\H}\setminus\{0\}$, and their standard basis by $\{{\bf1}
,{\qi},{\qj},{\qk}\}$, where
\begin{eqnarray}
{\qi}{\qj}={\qk},\
{\qj}{\qk}={\qi},\
{\qk}{\qi}={\qj},\
\qi^2 = \qj^2 = \qk^2 = -1.                                      \label{eq:quaternion_algebra}
\end{eqnarray}
This basis can be represented for example by the matrices 
\begin{eqnarray*}
{\bf1}=\left( \begin {array}{cc} 1&0\\0&1 \end{array}\right), \
\qi=\left( \begin {array}{cc} 0&-\ci \\-\ci &0 \end{array}\right),\ 
\qj=\left( \begin {array}{cc} 0&{-1}\\1&0 \end{array}\right),\
\qk=\left( \begin {array}{cc} -\ci &0\\0& \ci \end{array}\right).
 \end{eqnarray*}
We identify
$\H$ with 4-dimensional Euclidean space
$$
q=q_0 {\bf 1}+ q_1 \qi+q_2 \qj+q_3 \qk
\longleftrightarrow \ q=(q_0, q_1, q_2, q_3)\in\R^4.
$$
The length of a quaternion is $|q|^2 = q \overline q$, where $\overline q =q_0 {\bf 1}- q_1 \qi-q_2 \qj-q_3 \qk$ is the conjugate of $q$. The inverse of $q \neq 0$ is $q^{-1} = \frac{\overline q}{|q|^2}.$ The sphere $\S^3\subset\R^4$ is naturally identified with the group of unitary
quaternions $\H_1=\mathrm{SU}(2)$.

Three dimensional Euclidean space is identified with the space of imaginary quaternions
${\rm Im}\ {\H}$
\begin{eqnarray}
X =X_1\qi +X_2\qj +X_3\qk \in{\rm Im}\ {\H}\
\longleftrightarrow \ X=(X_1, X_2, X_3)\in\R^3.                             \label{eq:R^3=ImH}
\end{eqnarray}
The scalar and the cross products of vectors in terms of quaternions are given by
\begin{eqnarray} 
XY=-\<X,Y>+X \times Y, 										 \label{eq:products}
\end{eqnarray}
in particular
$$
[X,Y]=XY-YX=2 X\times Y.
$$
Throughout this article we will not distinguish quaternions, their matrix representation, and their vectors in $\R^3$. For example vectors $f$ and $n$ are also identified with imaginary quaternions, and we can also write the Christoffel dual one-form \eqref{eq:isothermic_dual} as $df^*=-f_u^{-1} du + f_v^{-1} dv.$

Moreover, we identify the space of complex numbers $\C$ with the span of ${\bf1}$ and $\qi$.
$$
z = a + b \ci \longleftrightarrow z = a {\bf1} + b \qi.
$$
In particular, we will often use a copy of the complex plane in the span of $\qj, \qk$ written as $\C \qj$. For $z \in \C$
$$
z \qj = (a + b \ci ) \qj = a \qj + b \qk.
$$
Note that $(a + b \ci) \qj = z \qj = \qj \bar z = \qj (a - b \ci)$.

For clarity we use $\Imc : \C \to \R$ and $\Imq: \H \to {\rm Im}\ {\H}$ to distinguish between the complex and quaternionic imaginary part. Note, in particular, that under these identifications, for $z = (a + b \ci) \in \C$, $\Imc z = b$ while $\Imq  = b \ \qi$. There is no ambiguity for the real part $\Re z = a$.

We will extensively use the actions of quaternions on $\R^3$. For $q \in \H$ and $X \in \R^3$, the action
$X \mapsto q^{-1} X q$ 
rotates $X$ about an axis parallel to $\Imq q$,
while $X \mapsto \overline{q} X q$ rotates about $\Imq q$ and scales by $|q|^2$.

\subsection{Local description of Bonnet pairs}
\label{subsec:Bonnet_isothermic_local}

If $f:D\to {\rm Im}\ {\H}$ is an isothermic surface with the differential $df=f_u du+ f_v dv$ then the dual isothermic immersion (\ref{eq:isothermic_dual}) is given by the closed form
\begin{equation}
\label{eq:isothermic_qual_quaternionic}
df^*=-(f_u)^{-1}du + (f_v)^{-1}dv,
\end{equation}
where $f_u$ and $f_v$ are imaginary quaternions. The closedness of (\ref{eq:isothermic_qual_quaternionic}) is equivalent to $-\frac{\partial}{\partial v}(f_u)^{-1}=\frac{\partial}{\partial u}(f_v)^{-1}$. The conformality 
\begin{equation}		\label{eq:conformality_quaternionic}
f_u^2=f_v^2=-e^{2h}, \quad f_u f_v=-f_v f_u
\end{equation} 
implies $(f_u)^{-1}=-e^{2h}f_u$, so we obtain the condition 
\begin{equation}   \label{eq:f_uv_technical}
f_u f_{uv} f_u=f_v f_{uv} f_v.
\end{equation}
This identity is satisfied for isothermic surfaces (\ref{eq:f_uv-isothermic}). 

Bianchi found out that Bonnet pairs can be described in terms of isothermic surfaces in $S^3$ \cite{Bianchi1903}. The following description of simply connected Bonnet pairs in quaternionic terms via isothermic surfaces was obtained by Kamberov, Pedit, and Pinkall in \cite{KPP}.

\begin{theorem}	\label{t:KPP}
The immersions $f^\pm :D\to {\rm Im}\ {\H} =\R^3$ build a Bonnet pair if and only if there exists an isothermic surface $f:D\to {\rm Im}\ {\H}$ and a real number $\epsilon\in \R$ such that
\begin{equation}\label{eq:KPP}
df^\pm=(\pm \epsilon - f)df^*(\pm \epsilon + f),
\end{equation}
where $f^*$ is the dual isothermic surface (\ref{eq:isothermic_qual_quaternionic}). 
\end{theorem}
Since this theorem plays a crucial role in our construction, we give its proof in one direction, showing that the formulas (\ref{eq:KPP}) indeed give a Bonnet pair. For the proof in other direction see \cite{KPP}, and also \cite{Bob_Bonnet2008}, where a description of Bonnet pairs and isothermic surfaces in terms of loop groups is also presented.
\begin{proof}
To show are:  the forms $df^\pm$ are closed and give conformally parametrized surfaces, the surfaces $f^\pm$ are isometric and non-congruent, and that they have equal mean curvatures.

Let us decompose (\ref{eq:KPP}) as
$$
df^\pm=-f df^* f \pm \epsilon [df^*, f] +\epsilon^2 df^*
$$
and check that all three terms are closed.

The closedness of the first term $fdf^* f=f(-(f_u)^{-1}du + (f_v)^{-1}dv)f$ is equivalent to $-(f (f_u)^{-1}f)_v= (f (f_v)^{-1}f)_u$. The latter identity follows from the conformality (\ref{eq:f_uv_technical}) and (\ref{eq:conformality_quaternionic}). The closedness condition for the second term $[df^*, f] $ is $-[f,(f_u)^{-1}]_v=[f,(f_v)^{-1}]_u$. It also follows from (\ref{eq:conformality_quaternionic}, \ref{eq:f_uv_technical}). The closedness of the last term was already shown (\ref{eq:isothermic_qual_quaternionic}).

Since (\ref{eq:KPP}) is a scaled rotation, the conformality of the frames $df^\pm$ follows from the conformality of the frame $df^*$.

The surfaces $f^\pm$ are isometric because their frames are related by the rotation $df^+=T^{-1}df^-T$ where $T=(\epsilon +f)/(\epsilon -f)\in \H_1$. As $T$ is non-constant $f^\pm$ are non-congruent.

The surfaces $f^\pm$ are isometric, so equality of their mean curvatures (\ref{eq:mean_curvature}) is equivalent to $$\<(f^+)_{uu}+(f^+)_{vv},n^+> =\<(f^-)_{uu}+(f^-)_{vv},n^->,$$ where $n^\pm$ are the respective Gauss maps for $f^\pm$. Setting $S=\epsilon +f$, with $S^{-1}=\bar{S}/|S|^2$, we have
$$
df^+=\bar{S}df^* S, \quad df^-=S df^* \bar{S}, \quad n^+=S^{-1} n S,\quad n^-= S n S^{-1},
$$
where $n$ is the Gauss map of $f^*$ (and $f$). We obtain 
\begin{eqnarray*}
\<(f^+)_u,(n^+)_u>&=&\<f^*_u, n_u |S|^2+[n,S_u \bar{S}]>\\&=&\<f^*_u, n_u |S|^2>-\<n, [f^*_u,S_u \bar{S}]> \text{ and }\\
\<(f^-)_u,(n^-)_u>&=&\<f^*_u, n_u |S|^2-[n,\bar{S}S_u ]>\\&=&\<f^*_u, n_u |S|^2>+\<n, [f^*_u,\bar{S}S_u]>.
\end{eqnarray*}
These two expressions are equal 
 if 
$
[f_u,\bar{S}S_u+S_u \bar{S}]=0.
$
The last identity is easy to check directly.

Analogously we obtain $\<(f^+)_v,(n^+)_v>=\<(f^-)_v,(n^-)_v>$, and thus equality of the mean curvatures $H^+=H^-$.
\end{proof}

\subsection{Periodicity conditions}

Now we pass to the global theory of isothermic surfaces and Bonnet pairs. We are mostly interested in the case of tori ${\cal R}=\C/{\cal L}$. If $\gamma$ is a homologically nontrivial cycle on $\cal R$ that is closed on the isothermic surface $f:{\cal R}\to \R^3$, then the corresponding curves are closed on the Bonnet pair if $\int_\gamma df^\pm=0$, i.e.,
\begin{eqnarray}
&{\rm (}\Apart{}  {\rm - periodicity \ condition)} \ &\int_\gamma -f df^* f +\epsilon^2 df^*=0 \text { and }  \label{eq:A-condition}\\
&{\rm (}\Bpart{}  {\rm - periodicity \ condition)} \ &\int_\gamma [df^*, f] =2 \int_\gamma \Imq ( df^* f )=0. \label{eq:B-condition}
\end{eqnarray}

Let us start with some simple observations. 
\begin{remark}
Varying the parameter $\epsilon$ in formula (\ref{eq:KPP}) is not essential. It is equivalent to scaling the surface $f\to f/\epsilon$.
\end{remark}

\begin{proposition}
The map $f\to -f^{-1}$ is the inversion of $\R^3\cup \infty$  in the unit sphere. This map is conformal and maps isothermic surfaces to isothermic surfaces. Thus if $f:\C/{\cal L}\to {\rm Im}\ {\H}$ is an isothermic torus then $f^{-1}:\C/{\cal L}\to {\rm Im}\ {\H}$ is also an isothermic torus. It and its dual are given by 
\begin{eqnarray}
d(f^{-1})=-f^{-1}f_u f^{-1} du - f^{-1} f_v f^{-1} dv, \nonumber \\ 
d(f^{-1})^*=f(f_u)^{-1}f du - f(f_v)^{-1} f dv. \label{eq:f^{-1}^*}
\end{eqnarray}
\end{proposition}
\begin{proof}
The class of isothermic immersions belongs to M\"obius geometry and is invariant under M\"obius transformations, in particular under $f\to -f^{-1}$.  We see that $f^{-1}(u,v)$ is conformal, and by direct computation one can check that (\ref{eq:f_uv-isothermic}) with $f_{uv} = \alpha f_u + \beta f_v$ for $\alpha,\beta \in R$ implies 
$$
(f^{-1})_{uv}=(\alpha +2\<f^{-1},f_v>)(f^{-1})_u+(\beta +2\<f^{-1},f_u>)(f^{-1})_v.
$$
Thus, $f^{-1}(u,v)$ is isothermic.
\end{proof}

\begin{remark}
Substituting $f^{-1}$ instead of $f$ into (\ref{eq:KPP}) and using (\ref{eq:f^{-1}^*}) we observe that the isothermic surfaces $f$ and $f^{-1}$  generate the same $\epsilon$-set of Bonnet pairs. 
\end{remark}

Formula (\ref{eq:f^{-1}^*}) implies the following periodicity property.
\begin{corollary} \label{cor:A-condition}
Let $f:\C/{\cal L}\to {\rm Im}\ {\H}$ be an isothermic torus such that the dual isothermic surfaces $f^*$ and $(f^{-1})^*$ are also tori. Then the $\Apart{}$-periodicity condition (\ref{eq:A-condition}) for Bonnet pairs is satisfied for any cycle on $\C/{\cal L}$ and for any $\epsilon$. Moreover,
\begin{equation} \label{eq:A-condition_integrated}
\int -f df^* f +\epsilon^2 df^*= (f^{-1})^*+\epsilon^2 f^*.
\end{equation} 
\end{corollary}

To construct Bonnet pair tori we will consider isothermic tori with closed curvature lines ($u$-lines and $v$-lines). They are given by a doubly periodic immersion $f:\C/{\cal L}\to {\rm Im}\ {\H}$, satisfying $f(u+\rm U,v)=f(u, v+{\rm  V})=f(u,v)$. The $\Bpart{}$-periodicity condition \eqref{eq:B-condition} becomes
\begin{eqnarray*}
\int_0^{\rm U} (f(f_u)^{-1}-(f_u)^{-1} f)du =0 \text{ and }
\int_0^{\rm V} ((f_v)^{-1} f - f(f_v)^{-1})dv =0.
\end{eqnarray*}

\section{Bonnet periodicity conditions when f has one generic family of closed planar curvature lines}
\label{sec:periodicty-conditions-planar-u}
 
This section describes our main insight: examples of compact Bonnet pairs that are tori arise from isothermic tori with one family of planar curvature lines.

First, we summarize the classification and global geometry of an isothermic cylinder $f(u,v)$ with one generic family of closed planar curvature lines, as described in our paper~\cite{short-isothermic-planar}.

Second, we build a Bonnet pair $f^+(u,v), f^-(u,v)$ from an isothermic cylinder $f(u,v)$ with one family of closed planar curvature lines.  We see the geometric construction and global properties of $f(u,v)$ vastly simplify the Bonnet periodicity conditions.

\subsection{Geometry of an isothermic surface with one generic family of closed planar curvature lines}
Construct an isothermically parametrized cylinder $f$ as follows. Consider a particular family of periodic curves $\gamma(\cdot, \w)$ satisfying $\gamma(u,\w)= \gamma(u + 2\pi,\w)$ expressed in the $1,\qi$-plane. Now, traverse these curves using a \emph{reparametrization} function $\w: \R \to \R, v \mapsto \w(v)$. Multiply by $\qj$ to place each curve into the $\qj, \qk$-plane and, simultaneously, rotate using a unit-quaternion valued function $\Phi: \R \to \H_1, v \mapsto \Phi(v)$. The curvature line planes therefore intersect in the origin, forming a cone, and generically span $\R^3$. Explicitly,
\begin{equation}
	\begin{aligned}
		\label{eq:fBasicStructure}
		f(u,v) &= \Phi(v)^{-1} \gamma(u,\w(v)) \qj \Phi(v) \text{ where } \\
		\Phi'(v) \Phi^{-1}(v) &= \sqrt{1 - \w'(v)^2} W_1(\w(v)) \qk.
	\end{aligned}	
\end{equation}
The term $\ci W_1(\w(v)) \qj$ lies in the $\qj,\qk$-plane and expresses the tangent line to the cone about which the corresponding plane infinitesimally rotates. The logarithmic derivative $\gamma_z/\gamma$ is a particular elliptic function of $z = u + \ci \w$ and $\overline{\gamma(z)}=-\gamma(\bar z)$.

It turns out that every isothermic cylinder with one generic family of closed planar curvature lines is constructed this way. The possible families of plane curves $\gamma(\cdot,\w)$ are given explicitly in terms of theta functions on an elliptic curve with rhombic period lattice. There is an open interval of such lattices.

We now state the classification and geometric results from our paper on isothermic surfaces with one family of planar curvature lines~\cite{short-isothermic-planar}. For theta functions $\vartheta_i$ we use the convention of Whittaker and Watson ~\cite[Sec. 21.11]{whittaker_watson_1996}, where they are defined on a lattice spanned by $\pi$ and $\tau \pi$ with nome $q = e^{\ci \pi \tau}$.

\begin{theorem}[\protect{\cite[Theorem 5]{short-isothermic-planar}}]
	\label{thm:planarIsothermicCylinderFormulas}
	Every real analytic isothermic cylinder with one generic family ($u$-curves) of closed planar curvature lines is given in isothermic parametrization by the following formulas.
	\begin{align}
		f(u,v) &= \Phi^{-1}(v) \gamma(u,\w(v)) \qj \Phi(v), 
		\label{eq:isothermicCylinder}\\
		\gamma(u,\w) &=  -\ci \frac{2 \vartheta _2(\omega)^2}{\vartheta_1^{\prime}(0) \vartheta _1(2 \omega)} \frac{\vartheta_1\left(\frac{1}{2} (u+\ci \w-3 \omega)\right)}{\vartheta_1\left(\frac{1}{2} (u+\ci\w+\omega)\right)},
		\label{eq:gammaClosedCurves} \\
		\Phi'(v) \Phi^{-1}(v) &= \sqrt{1-\w'(v)^2} W_1(\w(v)) \qk, \label{eq:phiPrimePhiInverse}\\
		W_1(\w) &=\ci \frac{\vartheta _1^{\prime}(0)}{2\vartheta _2(\omega)}\frac{\vartheta _2(\omega -\ci \w)}{\vartheta_1(\ci \w)}.
	\end{align}
	The theta functions $\vartheta_i(z | \tau)$ are defined on an elliptic curve of rhombic type spanned by $\pi$ and $\pi \tau$ with $\tau = \frac12 + \ci \R$ satisfying $0 < \Imc \tau < \Imc \tau_0 \approx 0.3547$ (defined by $\vartheta_2''(0 | \tau_0) = 0$). The parameter $\omega$ is the unique \emph{critical} $\omega \in (0, \pi/4)$ satisfying $\vartheta _2^{\prime}(\omega | \tau) = 0$. The curvature line planes are tangent to a cone with apex at the origin. The function $\w(v)$ is a $\tau$-admissible reparametrization function.
\end{theorem}

For each rhombic lattice, we emphasize the functional freedom given by the reparametrization function $v \mapsto \w(v)$. To ensure the isothermic surface is real analytic we mildly restrict to $\tau$-admissible reparametrization functions.

\begin{definition}[\protect{ \cite[real analytic version of Definition 3]{short-isothermic-planar}}]	\label{def:admissibleReparametrizationFunction} Fix $\tau \in \frac12 + \ci \R$ such that $0 < \Imc \tau < \Imc \tau_0 \approx 0.3547$. A map $\w: \R \to \R$ is called a $\tau$-\emph{admissible reparametrization function} if
	\begin{enumerate}
		\item $\w: \R \to (0, 2 \pi \Imc \tau)$ is a real analytic function and
		\item $\sqrt{1-\w'(\cdot)^2}: \R \to \R$ is also a real analytic function.
	\end{enumerate}
\end{definition}

\begin{remark} 
\label{re: psi}
We make three remarks about this definition.
	\begin{itemize}
		\item The range of $\w$ is bounded because the closed curve $\gamma(\cdot, \w)$ degenerates for $\w = 0$ and $\w = 2\pi \Imc \tau$.
		
		\item We necessarily have $|\w'(v)| \leq 1$ for all $v \in \R$. Note that the associated square root function $\sqrt{1-\w'(\cdot)^2}$ can change sign at $v_0$ where $\w'(v_0) = 1$.
		
		\item To close the isothermic cylinders into tori we will consider $\tau$-admissible reparametrization functions that are periodic. These are essentially all periodic real analytic functions. If $\psi: \S^1 \to \R$ is a
		real analytic function then after appropriate renormalization to $\tilde \psi$ one can define $\w'(v) := \cos(\tilde \psi(v))$ and $\sqrt{1-\w'(v)^2} := \sin(\tilde \psi(v))$ so that $\w: \S^1 \to (0, 2 \pi \Imc \tau)$ is a periodic $\tau$-admissible reparametriztion function.
	\end{itemize}
\end{remark}

We summarize the conformal frame of $f$. The metric $e^{2h}$ and Gauss map $n$ will be used in the next section to better understand the periodicity conditions of the Bonnet pair surfaces $f^\pm$.

\begin{corollary}[\protect{\cite[Corollary 3]{short-isothermic-planar}}]
	Let $f(u,v)$ be an isothermic cylinder with one generic family of planar curvature lines, as in Theorem~\ref{thm:planarIsothermicCylinderFormulas}. Then its frame is 
	\begin{align}
		\label{eq:fuIsothermicPlanarClosed}
		f_u(u,v) &= e^{h(u,\w(v))} \Phi(v)^{-1} e^{\ci \sigma(u,\w(v))} \qj \Phi(v), \\
		\label{eq:fvIsothermicPlanarClosed}
		f_v(u,v) &= e^{h(u,\w(v))} \Phi(v)^{-1} \left( \sqrt{1-\w'(v)^2} \, \qi + \w'(v) \, e^{\ci \sigma(u,\w(v))} \qk \right) \Phi(v), \\
		\label{eq:nIsothermicPlanarClosed}
		n(u,v) &= \Phi(v)^{-1} \left( \w'(v) \, \qi - \sqrt{1-\w'(v)^2} \, e^{\ci \sigma(u,\w(v))} \qk \right) \Phi(v),
	\end{align}
	and the family of closed planar curves $\gamma(u,\w)$ satisfies the following.
	\begin{align}
		\label{eq:gammaUClosedCurves}
		\gamma_u(u,\w) &= -\ci \gamma_{\w}(u,\w) = e^{h(u,\w) + \ci \sigma(u,\w)} = -\ci \left(\frac{\vartheta_2\left(\frac{u + \ci \w - \omega}{2}\right)}{\vartheta_1\left(\frac{u + \ci \w + \omega}{2}\right)}\right)^2, \\
		\label{eq:etohClosedCurves}
		e^{h(u,\w)} &= \frac{\vartheta_2\left(\frac{u + \ci \w - \omega}{2}\right)}{\vartheta_1\left(\frac{u + \ci \w + \omega}{2}\right)}\frac{\vartheta_2\left(\frac{u - \ci \w - \omega}{2}\right)}{\vartheta_1\left(\frac{u - \ci \w + \omega}{2}\right)},\\
		\label{eq:etoisigmaClosedCurves}
		e^{\ci \sigma(u,\w)} &= -\ci \frac{\vartheta_2\left(\frac{u + \ci \w - \omega}{2}\right)}{\vartheta_1\left(\frac{u + \ci \w + \omega}{2}\right)}\frac{\vartheta_1\left(\frac{u - \ci \w + \omega}{2}\right)}{\vartheta_2\left(\frac{u - \ci \w - \omega}{2}\right)}.
	\end{align}
\end{corollary}

Moreover, the global geometry of these isothermic cylinders is as follows.

\begin{theorem}[\protect{\cite[Theorem 6]{short-isothermic-planar}}]
	\label{thm:planarIsothermicCylinderGeometry}
	Every real analytic isothermic cylinder $f(u,v)$ with one generic family of closed planar curvature $u$-lines has the following geometric properties:
	\begin{enumerate}
		\item 	 The $v$-curve defined by critical $u = \omega$ lies on a sphere of radius $| R(\omega)|$, where $R(\omega)$ is the real number given by
		\begin{align} R(\omega) = \frac{2 \vartheta _2(\omega)^2}{\vartheta_1^{\prime}(0) \vartheta _1(2 \omega)} \label{eq:radiusOmega}
		\end{align}
		Moreover, $f(\omega,v)$ and $f_u(\omega,v)$ are parallel and satisfy
		\begin{align}
			\frac{1}{R(\omega)}f(\omega,v) = -\frac{f_u(\omega,v)}{|f_u(\omega,v)|}. \label{eq:fomegafuParallel}
		\end{align}
		\item Define $f^{\textup{inv}} = -R(\omega)^2 f^{-1}$ as the inversion of $f$ in the sphere of radius $| R(\omega) |$. Then,
		\begin{align}
			\label{eq:fInvUShift}
			f^{\textup{inv}}(u, v) = f(2 \omega-u, v).
		\end{align}
		In other words, this inversion maps $f$ onto itself and is an involution. The plane of each $u$-curve is mapped to itself.
		\item A Christoffel dual $f^*$ of $f$, with $(f^*)_u = \frac{f_u}{| f_u |^2}$ and $(f^*)_v = - \frac{f_v}{| f_v |^2}$, is 
		\begin{align}
			\label{eq:fDualUShift}
			f^*(u,v) = - f(\pi - u, v). 
		\end{align}
		In other words, this duality maps $f$ onto (minus) itself and is an involution. The plane of each $u$-curve is mapped to itself.
	\end{enumerate}
\end{theorem}

Figure~\ref{fig:uCurvePlaneDualAndInverse} illustrates that the spherical inversion and dualization operations map each planar curvature line, and therefore the entire surface, onto (minus) itself. 

\begin{figure}[t]
	\centering
	\includegraphics[width=0.7\hsize]{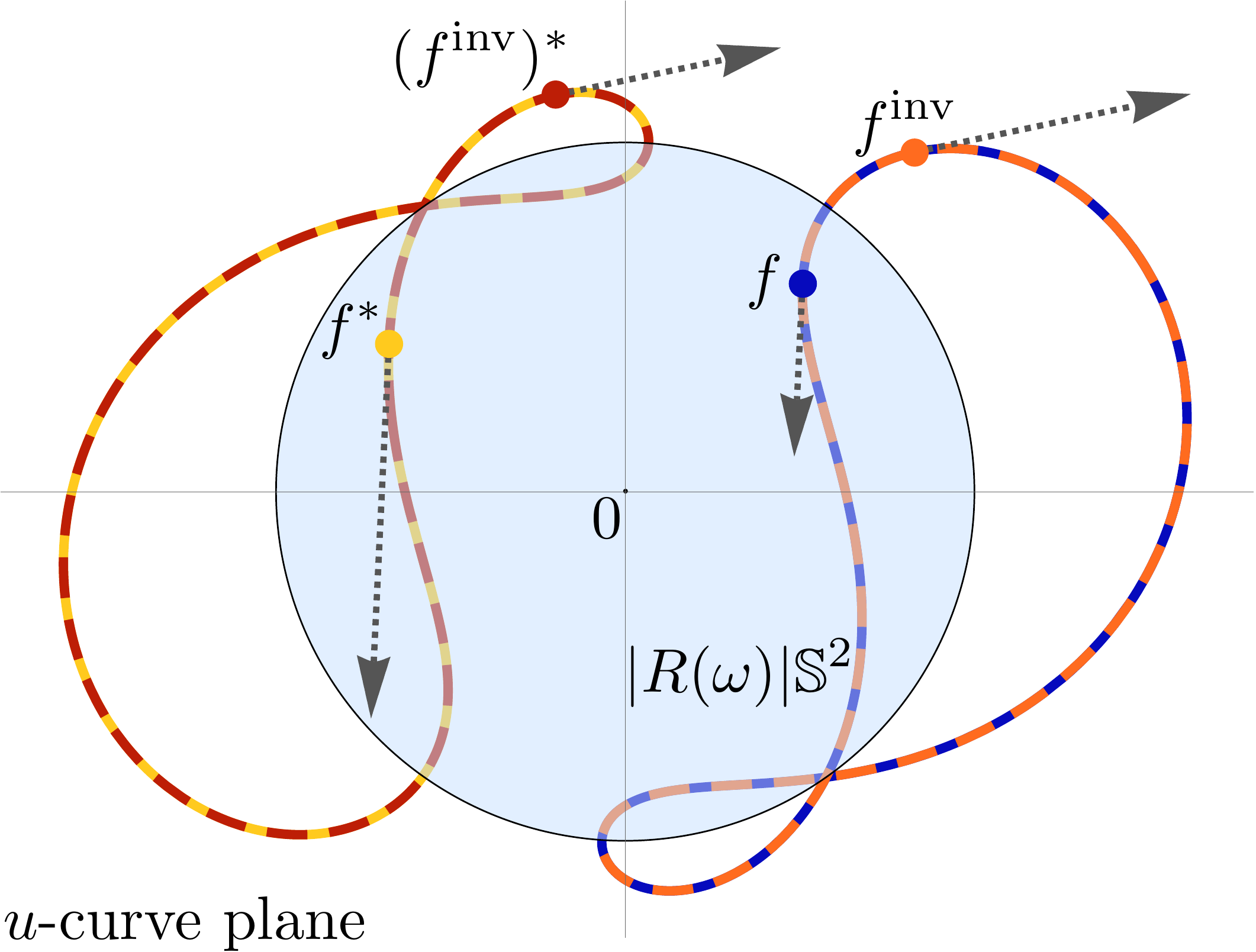}
	\caption{A $u$-curve plane of an isothermic surface $f$ with one generic family of closed planar curvature lines. The transformations $(\cdot)^{\textup{inv}}$, inversion in the sphere of radius $| R(\omega) |$ centered at the origin, and $(\cdot)^*$ a Christoffel dualization, are noncommuting involutions that map every closed planar curvature line onto (possibly minus) itself.}
	\label{fig:uCurvePlaneDualAndInverse}
\end{figure}

These global symmetries have remarkable consequences for the Bonnet periodicity conditions. In particular, when $f$ is a torus the $\Apart{}-$periodicity conditions \eqref{eq:A-condition} are automatically satisfied, as shown in Corollary~\ref{cor:A-condition}. The geometric construction of $f$ also has strong consequences for the $\Bpart{}-$periodicity conditions \eqref{eq:B-condition}. We will describe these in more detail after discussing how to close $f(u,v)$ into a torus. From \eqref{eq:fBasicStructure} we see that for $f(u,v)$ to be a torus, it is necessary that the $\tau$-admissible reparametrization function $\w(v)$ is periodic. We define the fundamental piece as the trace of $f(u,v)$ after one period of $\w(v)$.
\begin{definition}
	\label{def:isothermicCylinderFromAFundamentalPiece}
	We say $f(u,v) = \Phi(v)^{-1} \gamma(u,\w(v)) \qj \Phi(v)$ is an \emph{isothermic cylinder from a fundamental piece} if
	\begin{enumerate}
		\item $f(u,v)$ is an isothermic cylinder with one generic family ($u$-curves) of closed planar curvature lines as in Theorem~\ref{thm:planarIsothermicCylinderFormulas} and 
		\item $\w(v)$ has period $\V$, but $f$ is not closed after only one period, i.e., 
		\begin{align*}
			\w(v + \V) = \w(v) \quad \text{ and } \quad	f(u,v + \V) \neq f(u,v). 
		\end{align*}
		The \emph{fundamental piece $\FP$} is the parametrized cylinderical patch
		\begin{align*}
			\FP = \{ f(u,v) \, \big \vert \, u \in [0,2\pi] \text{ and } v \in [0, \V]\}.
		\end{align*}
	\end{enumerate}
	The \emph{axis} $A$ and \emph{generating rotation angle} $\angleFP \in [0,\pi]$ are defined by the monodromy matrix of the ODE for $\Phi$ \eqref{eq:phiPrimePhiInverse}:
	\begin{equation}
		\label{eq:monodromy}
		\Phi(0)^{-1}\Phi(\V) = \cos (\angleFP/2) + \sin (\angleFP/2) \axisDir.
	\end{equation}
	
\end{definition}

An isothermic cylinder from a fundamental piece $\Pi$ is extended by piecing together congruent copies of $\Pi$ via a fixed rotation.

\begin{lemma}
	\label{lem:rotationalSymmetryIsothermicCylinderFromFundamentalPiece}
	Let $f$ be an isothermic cylinder from a fundamental piece $\FP$ with axis $\axisDir$ and angle $\angleFP$. Define the rotation quaternion $R = \cos (\angleFP/2) + \sin (\angleFP/2) \axisDir$. Then for all $1 < k \in \N, u \in [0,2\pi],$ and $v \in [0, \V]$ we have
	\begin{equation}
		\label{eq:rotation_f}
		f(u,v+k \V) = R^{-k} f(u,v) R ^{k}.
	\end{equation}
\end{lemma}
\begin{proof}
	The traversed family of planar $u$-curves only depends on $\w(v)$, therefore $\gamma(u,v) = \gamma(u, v + \V)$, and
	$
	f(u,v+k \V)
	= \Phi(v + k \V)^{-1} \gamma(u,\w(v)) \qj \Phi(v + k \V).
	$	
	The frame $\Phi(v)$ is integrated from the ODE \eqref{eq:phiPrimePhiInverse}. When $\w(v)$ is periodic this ODE has periodic coefficients 
	with monodromy matrix $\Phi(0)^{-1}\Phi(\V)$. In other words, 
	\begin{align}
		\label{eq:phiMonodromyForAllV}
		\Phi(v + k \V) = \Phi(v) \left( \Phi(0)^{-1}\Phi(\V) \right)^{k}.
	\end{align}
	Continuing the calculation from above, we arrive at \eqref{eq:rotation_f}.
	This formula desribes the $k$-times rotation with the axis $\axisDir$ and generating rotation angle $\theta \in [0,\pi]$.
\end{proof}

Thus, closing the isothermic cylinder into a torus is a rationality condition, as shown in Figure~\ref{fig:fundamental-piece-axis}.

\begin{lemma}
\label{lem:rationalClosingIsothermicCylinderFromFundamentalPiece}
	An isothermic cylinder $f$ from a fundamental piece is a torus if and only if the generating rotation angle $\angleFP \in [0,\pi]$ satisfies $k \angleFP \in 2\pi \N $ for some $k \in \N$ so the $v$-period is $k \V$.
\end{lemma}

\begin{figure}
	\centering
	\includegraphics[width=\hsize]{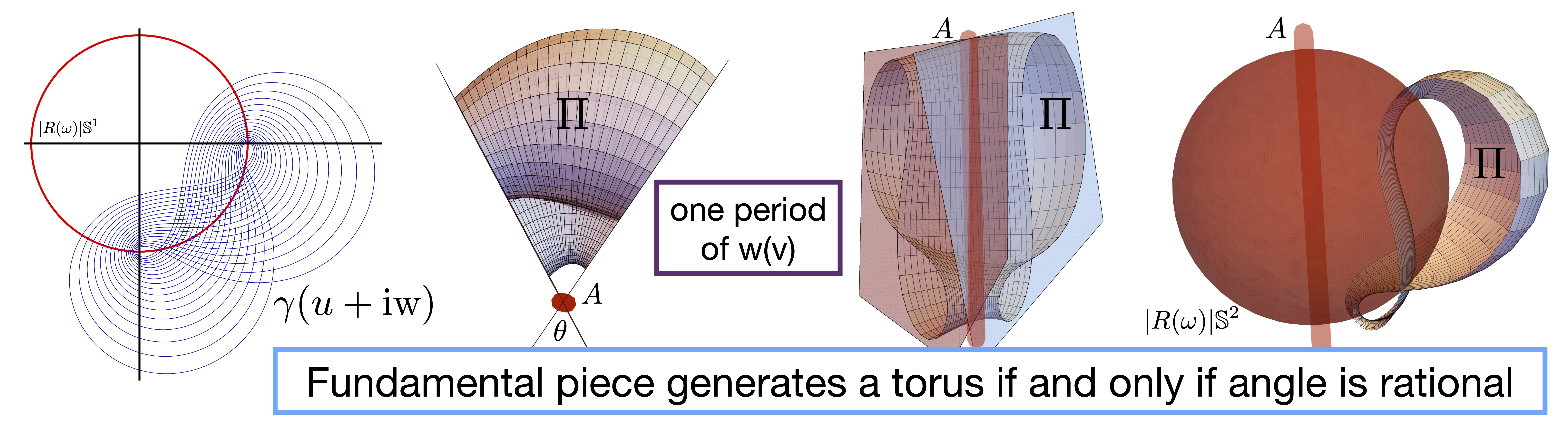}
	\caption{Left: A family of closed curves $\gamma(\cdot, \w)$ with its symmetry circle. Middle: Two views of a fundamental piece $\Pi$ with axis $\axisDir$ and angle $\angleFP$ traced out by one period of a periodic $\tau$-admissible reparametrization function $\w(v)$. Right: A third view of the fundamental piece showing the axis and sphere symmetry.}
	\label{fig:fundamental-piece-axis}
\end{figure}

\subsection{The corresponding Bonnet pair cylinders}

The global symmetries, stated in Theorem~\ref{thm:planarIsothermicCylinderGeometry}, of an isothermic cylinder $f$ with one generic family of closed planar curvature lines have remarkable consequences for the Bonnet periodicity conditions~\eqref{eq:A-condition} and~\eqref{eq:B-condition}.

\begin{theorem}
	\label{thm:BPFromPlanarFormulas}
	Let $f(u,v) = \Phi(v)^{-1} \gamma(u,\w(v)) \qj \Phi(v)$ be a real analytic isothermic cylinder with one generic family of closed planar curvature lines as in Theorem~\ref{thm:planarIsothermicCylinderFormulas}. For each $\epsilon \in \R$, the resulting Bonnet pair surfaces $f^\pm(u,v)$ are real analytic cylinders with translational periods in $v$ that are equal up to sign. Their immersion formulas are:
	\begin{align}
		\begin{aligned}
			\label{eq:fPlusMinusBasicStructure}
	f^{\pm}(u,v) = &\overbrace{R(\omega)^2f(\pi-2\omega + u, v) - \epsilon^2 f(\pi - u, v)}^{\Apart{} = (f^{-1})^* + \epsilon^2 f^*} \pm \underbrace{2\epsilon \left( \Phi^{-1}(v) \hat B(u,\w(v)) \qi \Phi(v)  + \tilde B(v) \right)}_{\Bpart{} = 2 \epsilon \int \Imq ( df^* f )},
		\end{aligned}
	\end{align}
	where 
$		R(\omega) = \frac{2 \vartheta _2(\omega)^2}{\vartheta_1^{\prime}(0) \vartheta _1(2 \omega)}$~\eqref{eq:radiusOmega}, $\hat B(u,\w(v))$ is a real analytic real-valued function that is $2\pi$-periodic in $u$, and $\tilde B(v)$ is a real analytic $\R^3$-valued function that depends only on $v$.
\end{theorem}
\begin{proof}
	The proof of this theorem, together with explicit formulas for $\hat B(u,\w(v))$ and the ODE determining $\tilde B(v)$ are given in Appendix~\ref{sec:explicit-bonnet-formulas}.
\end{proof}

Theorem~\ref{thm:BPFromPlanarFormulas} shows that an isothermic cylinder $f(u,v)$ with one generic family of closed planar curvature lines gives rise to Bonnet pair surfaces $f^\pm(u,v)$ that are cylinders. All three surfaces $f(u,v)$ and $f^\pm(u,v)$ are $2\pi$-periodic in $u$.

Moreover, when $f$ is a torus, both $\Phi$ and $\w$ are periodic. So, if $f$ is a torus, then the closing of both Bonnet pair cylinders $f^\pm$ reduces to a single real-valued integral in terms of $\tilde B$. This integral is computed along the spherical $v$-curve $f(\omega,v)$, recall Theorem~\ref{thm:planarIsothermicCylinderGeometry}, on the isothermic surface.

\begin{lemma}
	\label{lem:reductionToBPartAlongAVCurve}
	Let $f$ be an isothermic torus from a fundamental piece with axis $\axisDir$ and $v$-period $k \V$. The following are equivalent.
	\begin{enumerate}
		\item The Bonnet pair surfaces $f^\pm$ parametrized by \eqref{eq:fPlusMinusBasicStructure} are tori.
		\item The axial component of $\tilde B(v)$ vanishes over a period of $\w(v)$, i.e.,
		\begin{align}
			\label{eq:axialTildeBCondition}
			\left \langle \axisDir,\int_0^{\V}  \tilde B'(v) dv \right \rangle_{\R^3} = 0.
		\end{align}
		\item The Gauss map $n$, metric $e^{2h} = |f_u|^2 = |f_v|^2$, and axis $\axisDir$ of $f$ satisfy the following along the spherical curve $u = \omega$ over a period of $\w(v)$.
		\begin{align}
		\label{eq:axialWeightedGaussMapCondition}
		\left\langle \axisDir,	\int_0^{\V} e^{-h(\omega,\w(v))} n(\omega,v) dv \right \rangle_{\R^3}  = 0.
		\end{align}
	\end{enumerate}
\end{lemma}
\begin{proof}
The Bonnet pair cylinders parametrized by \eqref{eq:fPlusMinusBasicStructure} are tori when the four terms on the right-hand side are $v$-periodic. When $f(u,v)$ is a torus, $\Phi(v)$ and $\w(v)$ are periodic, so one immediately verifies the first three terms are $v$-periodic. We compute the translational period of $\tilde B(v)$. 

By combining \eqref{eq:fPlusMinusBasicStructure} and \eqref{eq:fDualUShift} we see that $\tilde B'(v)$ can be written in the form $\Phi(v)^{-1} X(\w(v), \w'(v)) \Phi(v)$. With \eqref{eq:phiMonodromyForAllV} for $\Phi(v + k \V)$ we find
\begin{align*}
	\int_{0}^{k \V} \tilde B'(v) dv 
	& = \sum_{n=0}^{k-1} \left(\Phi(0)^{-1} \Phi(\V) \right)^{-n} \left( \int_0^{\V} \tilde B'(v) dv \right) \left(\Phi(0)^{-1} \Phi(\V) \right)^{n}.
\end{align*}
This sum decomposes into the parts parallel and orthogonal to the axis. The orthogonal part must vanish. To see this, recall that $\Phi(0)^{-1} \Phi(\V)$ induces a rotation about the axis $\axisDir$ by angle $\angleFP$ with $k \angleFP \in 2\pi \N$. Hence, the orthogonal part is the sum of the planar $k$ roots of unity, which is zero. The parallel components then sum together, giving
\begin{align}
	\int_{0}^{k \V} \tilde B'(v) dv = k \left \langle \axisDir,\int_0^{\V}  \tilde B'(v) dv \right \rangle_{\R^3} \axisDir.
\end{align}
Thus, the condition that $\tilde B(v)$ is periodic with period $k \V$, so that the Bonnet surfaces are tori, is equivalent to \eqref{eq:axialTildeBCondition}.

It remains to prove the equivalent condition on the weighted integral of the Gauss map $n = e^{-2h} f_u f_v$. Since $\Phi$ and $\hat B$ are both $v$-periodic, \eqref{eq:fPlusMinusBasicStructure} implies that
\begin{align*}
\int_0^{k \V} \tilde B'(v) dv = \int_0^{k\V} \Imq \left( (f^*)_v(\omega,v) f(\omega,v) \right) dv.
\end{align*}
Now use $(f^*)_v = -\frac{f_v}{|f_v|^2} = -f_ve^{-2h}$ and $f(\omega,v)$ from \eqref{eq:fomegafuParallel} to compute
\begin{align}
\label{eq:bPartWeightedGaussMap}
(f^*)_v(\omega,v) f(\omega,v) 
= - R(\omega) e^{-h(\omega,\w(v))} n(\omega,v).
\end{align}
Thus,
\begin{align*}
\int_0^{k \V} \tilde B'(v) dv = - R(\omega) \int_0^{k\V} e^{-h(\omega,\w(v))} n(\omega,v) dv.
\end{align*}
The Gauss map $n$ along the $u = \omega$ curve has a similar structure as $\tilde B'(v)$, namely $\Phi(v)^{-1} Y(\w(v), \w'(v)) \Phi(v)$. The same reasoning applies to show that
\begin{align*}
	\int_{0}^{k \V} e^{-h(\omega,\w(v))} n(\omega,v) dv = k \left \langle \axisDir,\int_0^{\V}  e^{-h(\omega,\w(v))} n(\omega,v) dv \right \rangle_{\R^3} \axisDir,
\end{align*}
so condition \eqref{eq:axialWeightedGaussMapCondition} is equivalent to \eqref{eq:axialTildeBCondition}.
\end{proof}

\subsection{Simultaneous periodicity conditions of an isothermic cylinder from a fundamental piece and its corresponding Bonnet pair cylinders.}
We summarize Lemma~\ref{lem:rationalClosingIsothermicCylinderFromFundamentalPiece} and Lemma~\ref{lem:reductionToBPartAlongAVCurve} into the following theorem. These are the conditions on $f$ so that it is a torus and gives rise to tori $f^\pm$.

\begin{theorem}
	\label{thm:isothermicCylinderBonnetPeriodicityConditions}
	Let $f(u,v)$ be an isothermic cylinder with one generic family ($u$-curves) of closed planar curvature lines, and with periodic $\w(v + \V) = \w(v)$ that yields a fundamental piece with axis $\axisDir$ and generating rotation angle $\angleFP \in [0, \pi]$. Denote its Gauss map by $n$ and metric by $e^{2h}$.
	
	Then the resulting Bonnet pair cylinders $f^\pm$ are tori if and only if
	\begin{enumerate}
		\item (Rationality condition) 
		\begin{align}
		\label{eq:rationalityCondition}
			k \angleFP \in 2\pi \N \text{ for some } k \in \N \quad \text{and}
		\end{align}
		\item (Vanishing axial \Bpart{} part)
		\begin{align}
		\label{eq:axialWeightedGaussMapConditionAgain}
		\left\langle \axisDir,	\int_0^{\V} e^{-h(\omega,\w(v))} n(\omega,v) dv \right \rangle_{\R^3}  = 0.
		\end{align}
	\end{enumerate}
\end{theorem}

Theorem~\ref{thm:isothermicCylinderBonnetPeriodicityConditions} reveals a path to prove the existence of compact Bonnet pairs, see Figure~\ref{fig:periodicity-conditions-outline}. Use the functional freedom of the periodic reparametrization function $\w(v)$ to simultaneously satisfy the rationality \eqref{eq:rationalityCondition} and one dimensional integral \eqref{eq:axialWeightedGaussMapConditionAgain} conditions.

The main analytical challenge is that we cannot explicitly compute the frame $\Phi(v)$ that rotates the planar curvature lines. The above conditions depend on this frame as follows: the monodromy $\Phi(0)^{-1}\Phi(\V)$ determines the angle $\angleFP$ for the rationality condition \eqref{eq:rationalityCondition} and the axis $\axisDir$; the Gauss map curve $n(\omega,v)$ in \eqref{eq:axialWeightedGaussMapConditionAgain}, however, depends directly on $\Phi(v)$.

To overcome this challenge, we restrict to the very special geometric setting where the $v$-curves of the isothermic surface $f$ lie on spheres. It is a classical theorem that if a surface has one family of planar and one family of spherical curvature lines then the centers of the spheres lie on a common line. By the rotational symmetry from the fundamental piece, this line must be the axis. Thus, in this case, we can compute the axis from local data at every $v$, which allows us to express both of the above conditions with analytically tractable formulas.

We prove the existence of compact Bonnet pairs by explicitly constructing them from isothermic tori with planar and spherical curvature lines. More general examples are then found by perturbing the isothermic torus so it no longer has spherical curvature lines, but retains its one family of planar curvature lines. The remainder of the article is dedicated to this construction.

\begin{figure}[tbh!]
	\centering
	\includegraphics[width=\linewidth]{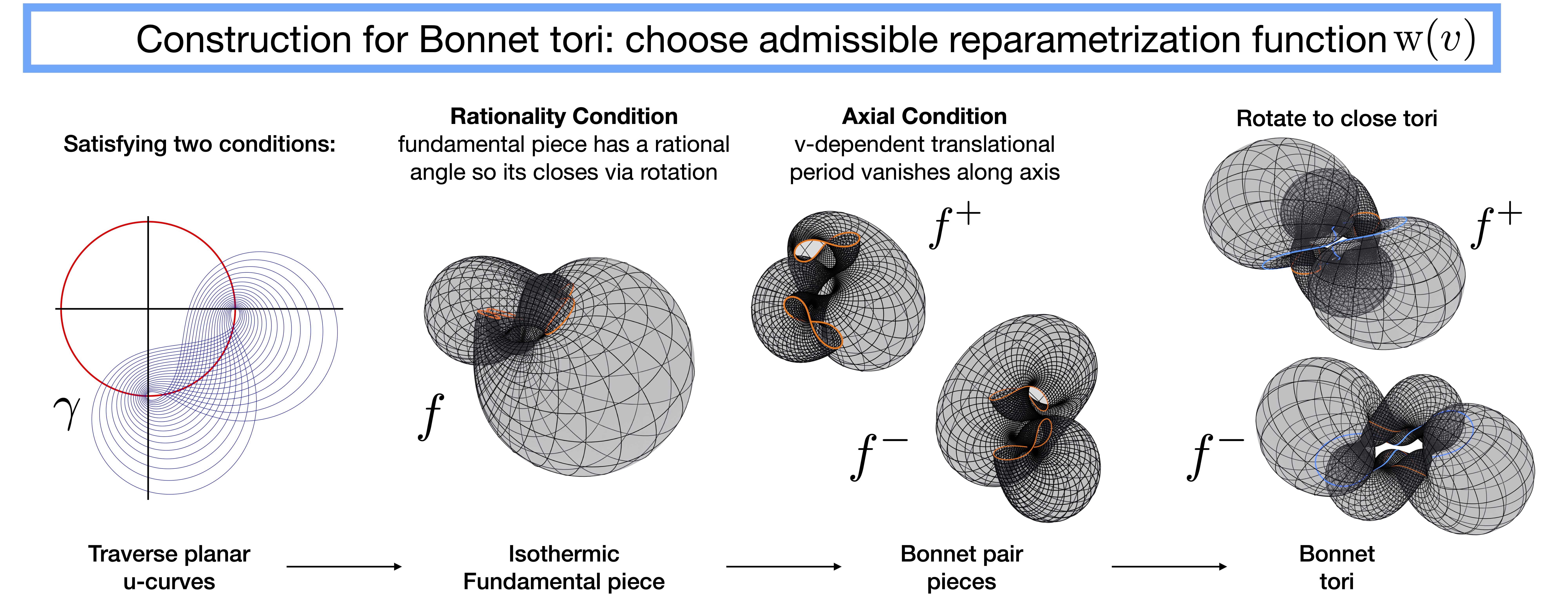}
	\caption{Overview of the construction for real analytic Bonnet pair tori as suggested by Theorem~\ref{thm:isothermicCylinderBonnetPeriodicityConditions}.}
	\label{fig:periodicity-conditions-outline}
\end{figure}


\section{Bonnet periodicity conditions when $f$ has one generic family of closed planar curvature lines and one family of spherical curvature lines}
\label{sec:periodicty-conditions-planar-u-spherical-v}
Our goal is to construct compact Bonnet pairs out of an isothermic cylinder $f(u,v) = \Phi^{-1}(v) \gamma(u,\w(v))\qj \Phi(v)$ with one generic family ($u$-curves) of closed planar curvature lines from a fundamental piece with axis $\axisDir$ and generating roation angle $\theta$. Recall Definition~\ref{def:isothermicCylinderFromAFundamentalPiece} and the discussion surrounding  Theorem~\ref{thm:isothermicCylinderBonnetPeriodicityConditions}. Note that one of the periodicity conditions is computed along the $u=\omega$ curve $f(\omega,v)$, which always lies on a sphere of radius $|R(\omega)|$ centered at the origin. The challenge is computing the axis $\axisDir$.

Fortunately, additionally requiring the isothermic cylinder to have all of its $v$-curvature lines spherical makes the problem analytically tractable. The key geometric insight is that the centers of the curvature line spheres are collinear and lie on the axis.

The analysis, however, is still quite involved. In particular, both periodicity conditions \eqref{eq:rationalityCondition}, \eqref{eq:axialWeightedGaussMapConditionAgain} are written as elliptic integrals on the elliptic curve that governs the spherical $v$-curves. This elliptic curve is stated in Corollary~\ref{cor:cylinderEllipticCurves}.

\subsection{Two elliptic curves and a local formula for the axis}
We state the necessary results from our paper on isothermic surfaces with one family of planar curvature lines~\cite{short-isothermic-planar}.

\begin{corollary}[\protect{\cite[Corollary 6]{short-isothermic-planar}}]
	\label{cor:cylinderEllipticCurves}
	Let $f(u,v)$ be an isothermic cylinder with one generic family ($u$-curves) of closed planar curvature lines, as in Theorem~\ref{thm:planarIsothermicCylinderFormulas}, with rhombic lattice spanned by $\pi, \tau \pi$, critical parameter $\omega$ satisfying $\vartheta_2'(\omega)=0$, and $\tau$-admissible reparametrization function $\w(v)$. Then,
	\begin{enumerate}
		\item The function $s(\w) = e^{-h(\omega, \w)}$ satisfies
		\begin{align}
			s'(\w)^2 &= Q_3(s), \label{eq:uEllipticCurve} \text{ where }\\
	Q_3(s) &= \frac{\vartheta _1^{\prime }(0)^2}{\vartheta_2(\omega)^2} \left(s-\frac{\vartheta _1(\omega)^2}{\vartheta _2(0)^2}\right) \left(s-\frac{\vartheta_3(\omega)^2}{\vartheta _4(0)^2}\right) \left(s-\frac{\vartheta _4(\omega)^2}{\vartheta _3(0)^2}\right) \\
	&= 2U_1'(\omega) s^3 -U_2(\omega) s^2 - 2 U'(\omega) s - U(\omega)^2,
	\label{eq:criticalQ3factorized}
	\end{align}
		and
		\begin{align}
			\label{eq:bigUAtOmegaAndRadius}
			\quad U(\omega) &= -\frac12  \frac{\vartheta_1'(0)}{\vartheta_2(\omega)^2} \vartheta_1(2\omega) = -R(\omega)^{-1} , \quad U_1(\omega) = 0,\\
			U'(\omega) &= -\frac12 \frac{\vartheta_1'(0)}{\vartheta_2(\omega)^2} \vartheta_1'(2\omega) , \quad			U_1'(\omega) = \frac12 \frac{\vartheta_1'(0)^2}{\vartheta_2(\omega)^2}, \\
			U_2(\omega) &= \frac{\vartheta _1^{\prime }(0)^2}{\vartheta _2(\omega)^2} \left(\frac{\vartheta _1(\omega)^2}{\vartheta _2(0)^2}+\frac{\vartheta _4(\omega)^2}{\vartheta_3(0)^2}+\frac{\vartheta _3(\omega)^2}{\vartheta _4(0)^2}\right).
		\end{align}
		Moreover, the elliptic curve \eqref{eq:uEllipticCurve} has rhombic lattice $\pi, \tau \pi$.
		\item The $v$-curvature lines are spherical if and only if the function $\w$ is given by $\w'(v) = \w'(s) s'(v)$ where
		\begin{align}
			\label{eq:sphericalEllipticCurve}
			\w'(s) &= 
			\frac{1}{\sqrt{Q_3(s)}} \quad \text{ and } \quad v'(s) = 
			\frac{\delta}{\sqrt{Q(s)}}, \text{ where }\\
			\label{eq:sphericalEllipticCurveFormula}
			Q(s) &= -(s-s_1)^2(s-s_2)^2 + \delta^{2}Q_3(s),
		\end{align}
		for some $0 \neq \delta \in \R$ and either $s_1, s_2 \in \R$ or $s_2 = \overline{s_1} \in \C$. In this case, $\sqrt{1-\w'(\cdot)^2}$ as a signed real-valued function is
		\begin{align}
			\label{eq:signedSquareRoot}
			\sqrt{1-\w'(v)^2} &= \frac{\delta^{-1} \left(s(\w(v))-s_1\right)\left(s(\w(v))-s_2\right)}{s'(\w(v))},  \text{ where } \\
			s(\w) &= e^{-h(\omega, \w)} = \frac{\vartheta _2(\omega)^2}{\vartheta _2(0)^2} \left(\frac{\vartheta _1(\omega)^2}{\vartheta _2(\omega)^2}-\frac{\vartheta _1\left(\frac{\ci \w}{2}\right)^2}{\vartheta _2\left(\frac{\ci \w}{2}\right)^2}\right).
		\end{align}
		\end{enumerate}
\end{corollary}
\begin{remark}
	The polynomial \eqref{eq:sphericalEllipticCurveFormula} defines a second elliptic curve $y^2 = Q(s)$. Throughout the following, we refer to this elliptic curve and the elliptic curve \eqref{eq:uEllipticCurve} as $Q$ and $Q_3$, respectively.
\end{remark}

Note that $s = e^{-h(\omega, \w)}$ already appears as the Gauss map weight in the vanishing axial periodicity condition \eqref{eq:axialWeightedGaussMapConditionAgain}. When the $v$-curvature lines are spherical, the axis $\axisDir$ from that periodicity condition are also understood in terms of local data depending on $s$, the parameters $\delta, s_1, s_2$ and constants from Corollary~\ref{cor:cylinderEllipticCurves}.

\begin{proposition}[\protect{\cite[Proposition 11]{short-isothermic-planar}}]
	\label{prop:bigZPrimeAtOmegaMovingFrameCoefficients}
	Let $f(u,v)$ be an isothermic cylinder from a fundamental piece with axis $\axisDir$ and whose second family of curvature lines are spherical. Then, the centers $Z(u)$ of the spherical curvature line spheres are collinear and
\begin{align}
	\label{eq:axisDir}
	\pm \axisDir = \frac{Z'(\omega)}{|Z'(\omega)|} &= 	z_1(s(v))\left(e^{-h(\omega,v)}f_u(\omega,v)\right) + z_2(s(v))\left(e^{-h(\omega,v)}f_v(\omega,v)\right) + z_3(s(v)) n(\omega, v),
\end{align}
where
\begin{align}
	z_1(s) &= \frac{s^{-1}}{|Z'(\omega)|} \left(1 + s R(\omega)^2 \left( U'(\omega) + s_1 s_2 U_1'(\omega)\right) \right),\\
	z_2(s) &= \frac{s^{-1}R(\omega)\delta^{-1}}{|Z'(\omega)|} \sqrt{Q(s)},\\
	z_3(s) &= \frac{s^{-1}R(\omega)\delta^{-1}}{|Z'(\omega)|} \left(-\delta^2 U_1'(\omega) s + (s-s_1)(s-s_2) \right),
\end{align}
and
\begin{align}
	\label{eq:bigZPrimeNormSquared}
	\begin{aligned}
		|Z'(\omega)|^2 =  R(\omega)^2 &\left( 2(s_1 + s_2)U_1'(\omega) + \delta^2 U_1'(\omega)^2 - U_2(\omega)\right) + \\ & + R(\omega)^4 \left( U'(\omega) + s_1 s_2 U_1'(\omega)\right)^2.
	\end{aligned}
\end{align}
\end{proposition}
Note the axis is constant, so the right hand side of \eqref{eq:axisDir} is independent of $v$. 
\subsection{Moduli space of isothermic cylinders from a fundamental piece of rhombic type}
To construct explicit examples of isothermic tori that lead to compact Bonnet tori, we work in the following finite dimensional space of isothermic cylinders with planar and spherical curvature lines.

\begin{definition}
	\label{def:isothermicCylinderFromFundamentalPieceOfRhombicType}
	Let $f(u,v)$ be an isothermic cylinder with one generic family of closed planar curvature lines whose second family of curvature lines are spherical as in Corollary~\ref{cor:cylinderEllipticCurves}. The closed planar $u$-curves are governed by a real elliptic curve $Q_3$ of rhombic type. We call $f(u,v)$ an \emph{isothermic cylinder from a fundamental piece of rhombic type} if :
	\begin{enumerate}
		\item the second real elliptic curve $Q$ governing the spherical $v$-curvature lines is also of rhombic type, and arises from choosing real parameters $\delta \neq 0, s_1, s_2 \in \R$;
		\item the real oval of $Q_3$ strictly contains the real oval of $Q$; and
		\item in \eqref{eq:sphericalEllipticCurve} we choose positive signs for both square roots.
	\end{enumerate}
\end{definition}
\begin{remark} \ 
	\begin{itemize}
		\item We denote the two real zeroes of $Q(s)$ by $s_1^-,s_1^+$, so its real oval is $[s_1^-, s_1^+]$. As suggested by the notation, we will often assume that $s_1^- < s_1 < s_1^+$.
		\item The real oval of $Q_3$ is $[s_0,\infty)$, where $s_0 = \frac{\vartheta_1(\omega ; Q_3)^2}{\vartheta_2(0; Q_3)^2} > 0.$ Note that the leading coefficient is $2U'_1(\omega) = \frac{\vartheta_1'(0 ; Q_3)^2}{\vartheta_2(\omega; Q_3)^2} > 0$.
		\item The next lemma shows $\w(v)$ is periodic. So, this definition is a special case of, and thus compatible with, Definition~\ref{def:isothermicCylinderFromAFundamentalPiece} for an \emph{isothermic cylinder from a fundamental piece}.
		\item The terminology \emph{of rhombic type} refers to the curve $Q$, since $Q_3$ must be of rhombic type in order to have closed $u$-curves. For simplicity we assume $s_1, s_2 \in \R$.
	\end{itemize}
\end{remark}

\begin{lemma}
	\label{lem:globalSphericalWOfv}
	An isothermic cylinder from a fundamental piece of rhombic type has immersion formula $f(u,v)$ given as in Theorem~\ref{thm:planarIsothermicCylinderFormulas} with periodic $\tau$-admissible reparametrization function $\w$ globally defined by:
	\begin{align}
		\label{eq:wOfVDoublyRhombic}
		\w(v) &= \wp^{-1} \left(\a_0 + \frac{\a_1 \, \wp(v;\delta^{-2}Q) + \a_2}{\a_3 \, \wp(v; \delta^{-2}Q) + \a_4} ; Q_3\right), \text{ where }\\
		\a_0 &= \frac{1}{24}Q_3''(s_0), \quad
		\a_1 = \frac{1}{4}Q_3'(s_0), \quad
		\a_2 = - \frac{1}{96}Q_3'(s_0)\delta^{-2}Q''(s_1^-), \\
		\a_3 &= (s_1^- - s_0), \quad
		\a_4 = \frac{1}{4}\delta^{-2}Q'(s_1^-) - (s_1^- - s_0)\frac{1}{24}\delta^{-2}Q''(s_1^{-}).
	\end{align}
	The period $\mathcal{V}$ of $\w(v)$ is the real period of $\wp(v;\delta^{-2}Q)$. 
	We use a nonstandard notation for the Weierstrass $\wp$ function as explained in Remark~\ref{rem:wpNotation}.
\end{lemma}
\begin{proof}
The real oval $[s_0,\infty)$ of $Q_3$ strictly contains the real oval $[s_1^-, s_1^+]$ of $Q$. Both elliptic integrals
\begin{equation}
\label{eq:elliptic_w,v}
\w(s)=\int_{s_0}^s \frac{dt}{Q_3(t)}, \quad v(s)=\int_{s_1^-}^s \frac{\delta dt}{Q(t)}
\end{equation}
are monotonic on the interval $[s_1^-, s_1^+]$, and thus define there a real analytic periodic function $\w(v)$. The period of this function 
$\V= 2 \int_{s_1^-}^{s_1^+} \frac{\delta dt}{Q(t)}$ is given by integrating around the real oval $[s_1^-, s_1^+]$ for $Q$. We need to verify that $\w(v)$ is $\tau$-admissible to ensure the immersed surface $f$ is real analytic.

The period lattice of $Q_3$ is of rhombic type and is spanned by the parallelogram $\pi, \tau \pi$. The elliptic integral $s\mapsto \ci\w$ maps the elliptic curve $Q_3$ to the fundamental parallelogram spanned by $\pi, \tau \pi$, and the real oval of $Q_3$ is mapped to the imaginary period $2\ci\pi \Imc \tau$ of the lattice of $Q_3$. Thus we obtain $0 < \w(v) < 2\pi \Imc \tau$. Further, we verify that $\sqrt{1-\w'(\cdot)^2}$ is also real analytic across its zero set. Using the right hand side of \eqref{eq:signedSquareRoot} we see that $v \mapsto \sqrt{1-\w'(v)^2}$ is real analytic with sign changes at $v_0 \in [0, \V/2]$ and $\V - v_0 \in [\V/2, \V]$, where $s(\w(v_0)) = s_1 \in [s_1^-, s_1^+]$.

Finally, the variable $s$ equates the inverse elliptic integrals of (\ref{eq:elliptic_w,v}). By the formula for the inverse of such integrals 
\cite[Sec. 20.6]{whittaker_watson_1996}, we have
		\begin{align}
		s_0 + \frac{\frac{1}{4}Q_3'(s_0)}{\wp(\w;Q_3)-\frac{1}{24}Q_3''(s_0)} = s_1^- + \frac{\frac{1}{4}\delta^{-2}Q'(s_1^-)}{\wp(v;\delta^{-2}Q)-\frac{1}{24}\delta^{-2}Q''(s_1^-)}.
	\end{align}
Now we can invert $\wp(\w;Q_3)$ for $s\in [s_1^-, s_1^+]$ to find the globally defined real analytic function $\w$ given by \eqref{eq:wOfVDoublyRhombic}.

\end{proof}
\begin{remark}
	\label{rem:wpNotation}
	We use a nonstandard notation for the Weierstrass $\wp$ function, motivated by the following result in \cite[Sec. 20.6]{whittaker_watson_1996}. For 
	\begin{align}
		\label{eq:generalQuartic}
		\mathcal{Q}(x) = c_4 + 4 c_3 x + 6 c_2 x^2 + 4 c_1 x^3 + c_0 x^4
	\end{align}
	the inverse of the elliptic integral $z = \int_{x_0}^x \mathcal{Q}(t)^{-\frac12}\,dt$ with $\mathcal{Q}(x_0) = 0$ is a rational function of $\wp(z; g_2(\mathcal{Q}), g_3(\mathcal{Q}))$ where the invariants are
	\begin{align}
		\label{eq:g2}
		g_2( \mathcal{Q}) &= c_0 c_4 - 4 c_1 c_3 + 3 c_2^2, \\
		\label{eq:g3}
		g_3( \mathcal{Q}) &= c_0 c_2 c_4 + 2 c_1c_2c_3 - c_2^3 - c_0c_3^2-c_1^2c_4.
	\end{align}
	We use the notation $\wp(z;  \mathcal{Q})$ to emphasize the dependence on $\mathcal{Q}$.
\end{remark}
\begin{remark}
	We assume the period $\V$ of $\w(v)$ does not close $f$ into a torus, i.e., $f(u,v + \V) \neq f(u, v)$. 
\end{remark}

\begin{lemma}
	\label{lem:moduliSpaceIsothermicCylindersFundamentalPieceRhombicType}
	The moduli space of isothermic cylinders from a fundamental piece of rhombic type has real dimension four. The parameters are $\Imc \tau \in \R$ with $0 < \Imc \tau < \Imc \tau_0$  giving critical $\omega \in (0, \pi/4)$ and $\delta \neq 0, s_1, s_2 \in \R$.
\end{lemma}

\subsection{Computing the periodicty conditions from the axis}
\begin{lemma}
	\label{lem:axialBPartAbelianIntegralGenericA3}
	Let $f(u,v)$ be an isothermic cylinder from a fundamental piece of rhombic type, with axis written as
	\begin{equation}
	\label{eq:axis_in_basis}
		\axisDir = a_1(s) \left(e^{-h(\omega,v)}f_u(\omega,v)\right) + a_2(s) \left(e^{-h(\omega,v)}f_v(\omega,v)\right) + a_3(s) n(\omega, v),
	\end{equation}
	where $s = e^{-h(\omega,\w(v))}$ lives on the elliptic curve $Q$ of rhombic type with real oval $[s_1^-, s_1^+]$.
	Then
	\begin{align*}
		\left\langle \axisDir,	\int_0^{\V} e^{-h(\omega,\w(v))} n(\omega,v) dv \right \rangle_{\R^3} = 2 \delta \int_{s_1^-}^{s_1^+} s \frac{a_3(s) }{\sqrt{Q(s)}} ds
	\end{align*}
\end{lemma}
\begin{proof}
	Since $A$ is constant we pass it under the integral
	and use
	\begin{align}
		\label{eq:axialDirectionWeightedGaussMapComponent}
		\left\langle \axisDir, e^{-h(\omega,\w(v))} n(\omega,v) \right\rangle_{\R^3} = s a_3(s).
	\end{align}
	In integrated form, using $v'(s) = \frac{\delta}{\sqrt{Q(s)}}$, we have 
	\begin{align*}
	\int_0^{\V} e^{-h(\omega,\w(v))} \left\langle \axisDir, n(\omega,v) \right\rangle_{\R^3} dv = \int_{v=0}^{v=\V} \delta s \frac{a_3(s) }{\sqrt{Q(s)}} ds.
	\end{align*}
	We conclude by noting that the integral from $v = 0$ to $v = \V$ is the contour integral in $s$ around the real oval $[s_1^-, s_1^+]$. Each value of $s$ is traversed exactly twice, so the contour integral is twice the real integral from $s = s_1^-$ to $s = s_1^+$.
\end{proof}

To compute the angle $\angleFP$ for the rationality condition \eqref{eq:rationalityCondition}, we write the spherical curve $f(\omega,v)$ in terms of spherical coordinates that are aligned with the axis direction at the north pole. Equating the axial $\Bpart{}$ part found by this method with the one found by the previous method leads to an equation for $\angleFP$ as a second elliptic integral.

\begin{lemma}
	\label{lem:angleFPAbelianIntegralGenericA1A3}
	Let $f(u,v)$ be an isothermic cylinder from a fundamental piece of rhombic type, with generating rotation angle $\angleFP$ and axis \eqref{eq:axis_in_basis}. Then
	\begin{align*}
		\angleFP = -\frac{2 \delta}{R(\omega)} \int_{s_1^-}^{s_1^+} \frac{1}{s}\frac{1}{(1 - a_1(s)^2)} \frac{a_3(s)}{\sqrt{Q(s)}} ds.
	\end{align*}
\end{lemma}
\begin{proof}
	By Theorem~\ref{thm:planarIsothermicCylinderGeometry}, the $u=\omega$ curve $f(\omega,v)$ lies on a sphere centered at the origin and $f(\omega,v)= - R(\omega) e^{-h(\omega, \w(v))} f_u(\omega,v)$. Let $e_1, e_2$ span the plane perpendicular to the axis $\axisDir$, so that $\left\{e_1, e_2, \axisDir\right\}$ is a fixed orthonormal basis for $\R^3$. Spherical coordinates in this basis define functions $r(v),z(v), \tilde \theta(v)$ satisfying
	\begin{align*}
		f(\omega,v) &= r(v)\cos\tilde\theta(v) e_1 + r(v) \sin\tilde\theta(v) e_2 + z(v) \axisDir \text{ with } \\
		R(\omega)^2 &= r(v)^2 + z(v)^2.
	\end{align*}
	Moreover, by the expansion of $A$ in terms of the moving frame, we know that, since $z(v) = \langle A, f(\omega,v) \rangle = -R(\omega) a_1(s(v))$,
	\begin{align*}
		r(v)^2 = R(\omega)^2 (1 - a_1(s(v))^2).
	\end{align*}
	
	Now, the angle $\tilde \theta(v)$ measures the rotation about the axis, so at the end of the fundamental piece it must agree with the generating rotation angle, i.e., $\tilde \theta(\V) = \angleFP$. So we seek a differential equation for $\tilde \theta(v)$. This arises by considering the axial component of $\left( f^* \right)_v(\omega,v) f(\omega,v)$, which we saw in \eqref{eq:bPartWeightedGaussMap} satisfies
	\begin{align}
		\left( f^* \right)_v(\omega,v) f(\omega,v) = - R(\omega) e^{-h(\omega,\w(v))}n(\omega,v).
	\end{align}
	Alternatively, using the above spherical coordinates and $(f^*)_v = -e^{-2h} f_v$,
	\begin{align*}
		\left\langle \axisDir, \left( f^* \right)_v(\omega,v) f(\omega,v) \right\rangle_{\R^3} = e^{-2h(\omega,\w(v)))} r(v)^2 \tilde\theta'(v).
	\end{align*}
	Combining this with the previous equation and expression for $r(v)^2$ gives, in terms of $s = e^{-h(\omega,\w(v))}$,
	\begin{align*}
		\left\langle \axisDir, e^{-h(\omega,\w(v))} n(\omega,v) \right\rangle_{\R^3} = - R(\omega) s^2 (1-a_1(s(v))) \tilde \theta'(v).
	\end{align*}
	On the other hand, \eqref{eq:axialDirectionWeightedGaussMapComponent} is
$
		\left\langle \axisDir, e^{-h(\omega,\w(v))} n(\omega,v) \right\rangle_{\R^3} = s a_3(s).
$
	Equating the two expressions and solving for $\tilde \theta'(v)$ yields
	\begin{align*}
		\tilde \theta'(v) = -\frac{1}{R(\omega)}\frac{s^{-1} a_3(s)}{1-a_1(s)^2}.
	\end{align*}
	Converting the integral from $v = 0$ to $v = \V$ into twice the real integral from $s = s_1^-$ to $s = s_1^+$ yields the result.
\end{proof}

To find explicit formulas for the periodicity conditions, we use the local axis formula in Proposition~\ref{prop:bigZPrimeAtOmegaMovingFrameCoefficients}.

\subsection{Periodicity conditions as elliptic integrals}
Now the periodicity conditions of Theorem~\ref{thm:isothermicCylinderBonnetPeriodicityConditions} can be given as elliptic integrals.

\begin{theorem}
	\label{thm:bonnetPeriodicityAsAbelianIntegrals}
	Let $f(u,v)$ be an isothermic cylinder, with one generic family ($u$-curves) of closed planar curvature lines and one family ($v$-curves) of spherical curvature lines, from a fundamental piece of rhombic type, determined by parameters $\omega, \delta, s_1, s_2$. 
	
	Then the arising Bonnet pair cylinders $f^\pm$ are tori if and only if
	\begin{enumerate}
		\item (Rationality condition)
			\begin{align}  \label{eq:rationality_integral}
				\frac{\angleFP}{2} = \int_{s_1^-}^{s_1^+} \frac{ Z_0}{\tilde{Q_2}(s)}\frac{Q_2(s)}{\sqrt{Q(s)}} ds \text{ satisfies } 
				k \angleFP \in 2\pi \N \text{ for some } k \in \N
			\end{align}
		and
		\item (Vanishing axial \Bpart{}- part)
		\begin{align}  \label{eq:B_part_integral}
			\int_{s_1^-}^{s_1^+} \frac{Q_2(s)}{\sqrt{Q(s)}} ds = 0.
		\end{align}
	\end{enumerate}
	Here, $s_1^-, s_1^+$ are the two real zeroes of $Q(s)$, and 
	\begin{align}
		Q(s) &= -(s-s_1)^2(s-s_2)^2 + \delta^{2}Q_3(s), \\
		Q_3(s) &= 2U_1'(\omega) s^3 -U_2(\omega) s^2 - 2 U'(\omega) s - U(\omega)^2, \\
		Q_2(s) &= -(s-s_1)(s-s_2) + \delta^2 U_1'(\omega) s,  \label{eq:Q_2}\\
		\tilde{Q}_2(s) &= Z_0(\omega, \delta,s_1,s_2)^2 s^2 -\nonumber \\ &\qquad - \left(1 + s U(\omega)^{-2} \left( U'(\omega) + s_1 s_2 U_1'(\omega)\right)\right)^2.
		\label{eq:tilde_Q_2}
	\end{align}
	The constant $Z_0(\omega, \delta, s_1, s_2) = |Z'(\omega)|$ depends on the parameters via
	\begin{align}
		\begin{aligned}
		Z_0^2=|Z'(\omega)|^2 = U(\omega)^{-2} &\left( 2(s_1 + s_2)U_1'(\omega) + \delta^2 U_1'(\omega)^2 - U_2(\omega)\right) + \\ & + U(\omega)^{-4} \left( U'(\omega) + s_1 s_2 U_1'(\omega)\right)^2.
		\end{aligned}
	\end{align}
\end{theorem}
\begin{proof}
	This follows by combining the results of Lemma~\ref{lem:axialBPartAbelianIntegralGenericA3} and Lemma~\ref{lem:angleFPAbelianIntegralGenericA1A3}, for the periodicity conditions in terms of the coefficients of the moving frame for the axis $\axisDir$ with their explicit formulas (up to a common sign) from Proposition~\ref{prop:bigZPrimeAtOmegaMovingFrameCoefficients}. Note we also used that $R(\omega)=-U(\omega)^{-1}$ \eqref{eq:bigUAtOmegaAndRadius}.
	
	To ease computation in the rationality condition, rewrite the relevant part of the integrand of $\angleFP$ as follows.
	\begin{align*}
	-\frac{2 \delta}{R(\omega)} \frac{1}{s}\frac{1}{(1 - a_1(s)^2)} a_3(s) = 2\frac{\frac{-\delta}{R(\omega)}s a_3(s)}{(s^2 - (s a_1(s))^2)}.
	\end{align*}
	Now, replace the $a_i$ with the relevant expressions for $z_i$ in Proposition~\ref{prop:bigZPrimeAtOmegaMovingFrameCoefficients}.
	\begin{align*}
	2  \frac{\frac{Q_2(s)}{|Z'(\omega)|}}{s^2 - \left(\frac{1 + s R(\omega)^2 \left( U'(\omega) + s_1 s_2 U_1'(\omega)\right)}{|Z'(\omega)|}\right)^2} = 2 |Z'(\omega)| \frac{Q_2(s)}{\tilde{Q}_2(s)}.
\end{align*}
	This proves \eqref{eq:rationality_integral}. Technically, we only know $\angleFP$ up to sign, but this does not impact the result. 
\end{proof}

We prove that these conditions can be simultaneously satisfied by studying the limit as the parameter $\delta$ goes to zero.

\subsection{Asymptotics for the periodicity conditions as $\delta \to 0$}
Let us investigate the asymptotic behavior of the integrals (\ref{eq:rationality_integral}) and (\ref{eq:B_part_integral}) in the limit $\delta \to 0$. In the degenerate case $\delta=0$ the polynomial $Q(s)$ has two double zeros at $s_1$ and $s_2$. For small $\delta$ they split into pairs $s_1^{\pm}$ and $s_2^{\pm}$ converging  to $s_1$ and $s_2$ respectively for $\delta\to 0$. A direct computation gives the following asymptotics for $s_1^{\pm}$:
\begin{equation}
\label{eq:alpha-beta}
\begin{aligned}
s_1^{\pm}=s_1\pm \alpha\delta + \beta \delta^2 + O(\delta^3), \quad \delta\to 0, \text{ with }\\
\alpha=\frac{\sqrt{Q_3(s_1)}}{s_1-s_2}, \quad \beta=\frac{1}{2}\frac{Q'_3(s_1)}{(s_1-s_2)^2}-\frac{Q_3(s_1)}{(s_1-s_2)^3},
\end{aligned}
\end{equation}
and similar asymptotics for $s_2^{\pm}$.

Since $s_1^\pm$ are real, and the elliptic curve $Q$ is of rhombic  type, the other pair $s_2^{\pm}$ must be a pair of complex conjugated points $s_2^-=\overline{s_2^+}$. Due to (\ref{eq:alpha-beta}) these conditions are equivalent to
\begin{equation}
\label{eq:Q_3><0}
Q_3(s_1)>0 \quad \text{ and } \quad Q_3(s_2)<0.
\end{equation}

By the change of variables $ s \to t$ defined by
$$
s=\frac{s_1^+ + s_1^-}{2}+\frac{s_1^+ - s_1^-}{2} t
$$
the integration interval $s\in[s_1^-, s_1^+]$ in the periodicity conditions (\ref{eq:rationality_integral}) and (\ref{eq:B_part_integral}) transforms to $t \in [-1,1]$. Note that
\begin{equation}
\label{eq:s(t)}
s=s_1+\alpha \delta t + \beta\delta^2 + O(\delta^3).
\end{equation}
In terms of the new variable the differential $ds/\sqrt{Q(s)}$ becomes
$$
\frac{2 dt}{\sqrt{1-t^2} \sqrt{(s_1^+ + s_1^- + (s_1^+ - s_1^-)t -2 s_2^+)(s_1^+ + s_1^- + (s_1^+ - s_1^-)t -2 s_2^-)}}
$$
with the asymptotics
$$
\frac{ds}{\sqrt{Q(s)}}=\frac{dt}{\sqrt{1-t^2}(s_1-s_2)}\left( 1- \frac{\delta\alpha t}{s_1-s_2}+ O(\delta^3)\right).
$$
Substituting (\ref{eq:s(t)}) into (\ref{eq:Q_2}) we obtain
$$
Q_2(s)=-\delta\alpha t (s_1-s_2) + \delta^2\left(-\beta (s_1-s_2)-\alpha^2 t^2+U_1'(\omega) s_1\right) +O(\delta^3).
$$
For the integral (\ref{eq:B_part_integral}) this implies
\begin{eqnarray}
\int_{s_1^-}^{s_1^+}\frac{Q_2(s)ds}{\sqrt{Q(s)}}&=
\int_{-1}^{1}\frac{\delta}{\sqrt{1-t^2}}\left(-\alpha t +\delta\left( -\beta + \frac{U_1'(\omega) s_1}{s_1-s_2}\right) +O(\delta^2)  \right)dt  \nonumber \\ 
&=
 \frac{\pi \delta^2}{(s_1-s_2)^3} A +O(\delta^3), \label{eq:asymptotics_periodicity_vanishing} 
\end{eqnarray}
where
\begin{equation}
A=(s_1-s_2)^2 U_1'(\omega) s_1 -\frac{1}{2} Q'_3(s_1)(s_1-s_2) +Q_3(s_1).
\label{eq:A}
\end{equation}
The coefficients in (\ref{eq:rationality_integral}) can be computed similarly. For brevity we omit $\omega$ in $U(\omega), U_2(\omega), U'(\omega), U_1'(\omega)$ below. 
\begin{align}
Z_0&=Z+O(\delta^2) \text{ with } \\ \label{eq:Z} Z&=\frac{1}{U^2} \sqrt{(U'+s_1 s_2 U'_1)^2 +U^2(2(s_1+s_2)U'_1-U_2)} \text{ and }\\
\frac{1}{\tilde{Q}_2(s)}&=\frac{1}{\tilde{Q}_2(s_1)}\left(1-  \delta\alpha t \frac{\tilde{Q'}_2(s_1)}{\tilde{Q}_2(s_1)} + O(\delta^2)\right).
\end{align}
For the integral (\ref{eq:rationality_integral}) we find
$$
\frac{Z\delta}{\tilde{Q}_2(s_1)}\int_{-1}^{1}\left(-\alpha t +\delta\left( -\beta + \frac{U_1' s_1}{s_1-s_2} + \alpha^2 t^2 \frac{\tilde{Q'}_2(s_1)}{\tilde{Q}_2(s_1)}\right) +O(\delta^2)  \right)\frac{dt}{\sqrt{1-t^2}}
$$
so that
\begin{align}
\int_{s_1^-}^{s_1^+}\frac{Z_0}{\tilde{Q}_2(s)}\frac{Q_2(s)ds}{\sqrt{Q(s)}}=
\frac{Z\delta^2}{\tilde{Q}_2(s_1)}\left( \frac{\pi A}{(s_1-s_2)^3} + \alpha^2 \frac{\pi}{2}\frac{\tilde{Q'}_2(s_1)}{\tilde{Q}_2(s_1)}\right)+O(\delta^3).
\label{eq:rationality_asymptotic}
\end{align}

\begin{lemma}
\label{lem:periodicity_spherical_implicit}
The periodicity conditions (\ref{eq:rationality_integral}) and (\ref{eq:B_part_integral}) have solutions for sufficiently small $\delta$ if there exist 
$s_1$ and $s_2$ satisfying the following conditions:
\begin{itemize}
	\item[(i)] elliptic curve of rhombic type: inequalities in \eqref{eq:Q_3><0},
	\item[(ii)] vanishing axial \Bpart{}-part:  $A(s_1, s_2)=0$ with $A(s_1, s_2)$ given by (\ref{eq:A}),
	\item[(iii)] real square root in (\ref{eq:Z}): $(U'+s_1 s_2 U'_1)^2 +U^2(2(s_1+s_2)U'_1-U_2)>0$,
	\item[(iv)] non-degeneracy of the rationality condition:  $(2s_1+s_2)s_1 U'_1-s_1U_2-U'\neq 0$,
	\item[(v)] non-degeneracy of the vanishing condition: $(s_1+2s_2)s_1 U'_1-s_1U_2-U'\neq 0$.
\end{itemize}
\end{lemma}
\begin{proof}
The periodicity condition (\ref{eq:B_part_integral}) is the vanishing condition for the integral 
$$
b(s_1,s_2,\delta):=\int_{s_1^-}^{s_1^+}\frac{Q_2(s)ds}{\sqrt{Q(s)}}.
$$
Its behavior for small non-vanishing $\delta$ is given by the asymptotic (\ref{eq:asymptotics_periodicity_vanishing}). The function
$$
\tilde{b}(s_1,s_2,\delta):=\frac{1}{\delta^2}b(s_1,s_2,\delta)
$$
is an analytic function for small non-vanishing $\delta$.

Let $s_1^0, s_2^0$ satisfy (i-v). We fix $s_1^0$ and apply the implicit function theorem to $\tilde{b}(s_1, s_2, \delta)=0$ at $(s_1, s_2, \delta)=(s_1^0, s_2^0, 0)$.  At this point $\tilde{b}(s_1^0, s_2^0, 0)=0$ and $\partial \tilde{b}(s_1,s_2,\delta)/\partial s_2 \neq 0$ is equivalent to $\partial A/\partial s_2 \neq 0$, which is the condition (v). The implicit function theorem guarantees the existence of a solution of $\tilde{b}(s_1, s_2(\delta), \delta)=0$  for small $\delta$  with $s_1=s_1^0$ and $s_2(0)=s_2^0$.

In the case when (ii) is satisfied the asymptotic (\ref{eq:rationality_asymptotic}) of the rationality condition (\ref{eq:rationality_integral}) becomes
$$
\theta (s_1,s_2,\delta)=\delta^2  
Z \alpha^2 \pi\frac{\tilde{Q'}_2(s_1)}{\tilde{Q}^2_2(s_1)}+O(\delta^3).
$$
The leading term does not vanish since $\tilde{Q'}_2(s_1)\neq 0$ is equivalent to (iv). Thus, the rationality of $\theta$  can be achieved by a variation of $\delta$. 
\end{proof}

Note that $s_1^0$ and $\omega$ (or, equivalently, the modulus $\tau$ of the elliptic curve) stay fixed in this consideration and can be treated as additional parameters. 

\begin{figure}[h]
\begin{center}
\includegraphics[width=0.5\linewidth]{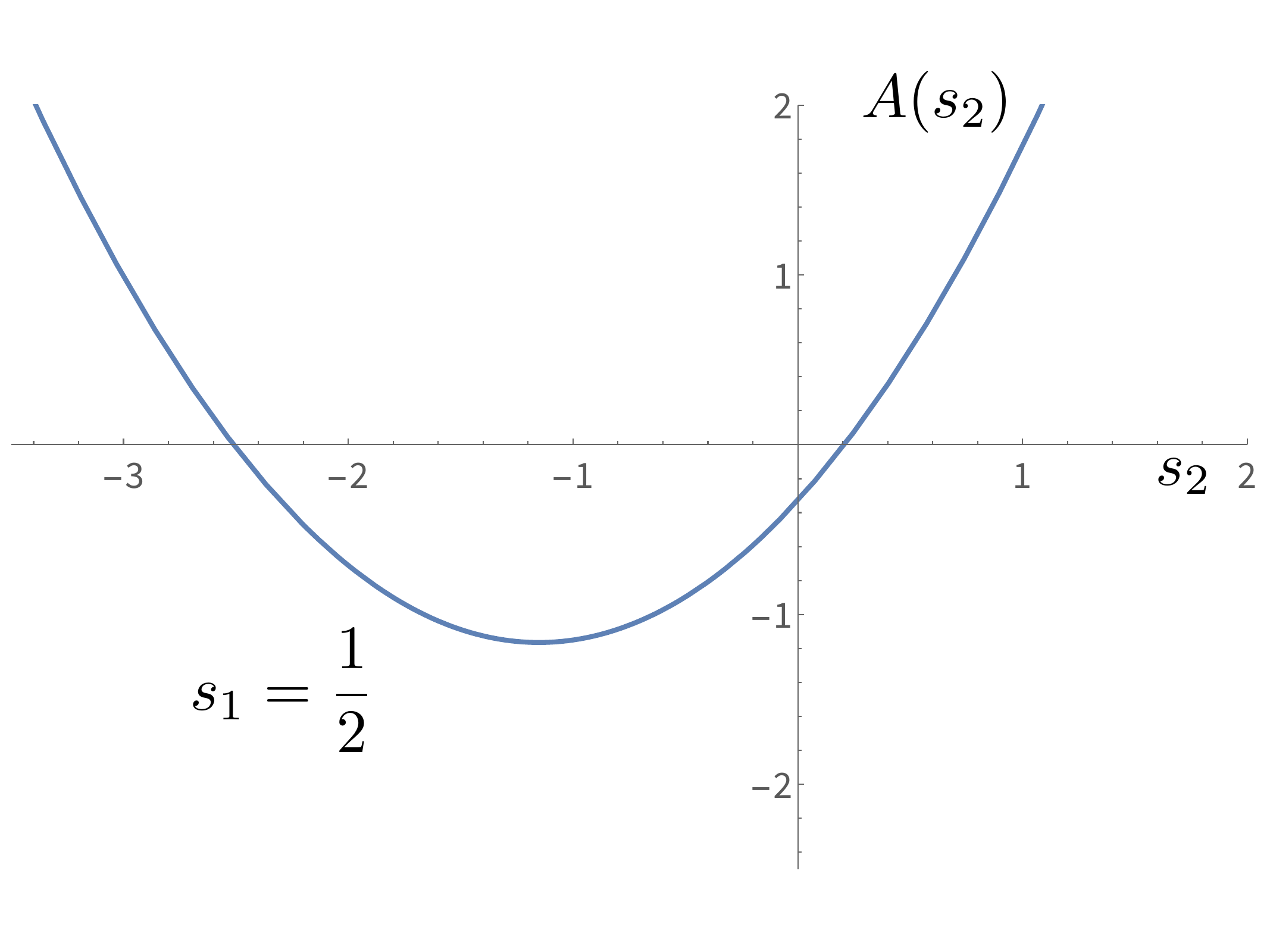}
\end{center}
\caption{Plot of the quadratic $A(s_1=\frac12, s_2)$ \eqref{eq:A} for the vanishing axial $\Bpart{}$ part condition, for $\Imc \tau \approx 0.3205128205$, critical $\omega \approx 0.3890180475$, and $\delta = 0.001$.}
\label{fig:zero_A}
\end{figure}

\begin{lemma}
\label{lem:exist_s_12}
There exist $s_1,s_2\in{\mathbb R}$ satisfying  conditions (i-v).
\end{lemma}
\begin{proof}
We analyze these conditions for large $s_1$. Let $s_1$ be either large positive or large negative with $U'_1(\omega)s_1\to +\infty$, such that $Q_3(s_1)>0$. The quadratic equation $A=0$ for $s_2$ has two solutions with the following asymptotic behavior: 
\begin{equation}
\label{eq:asymptotics_s12}
s_2^{(1)}=\frac{U'}{2s_1 U'_1}+ O(\frac{1}{s_1^2}), \quad s_2^{(2)}=-s_1+ \frac{U_2}{U'_1}+O(\frac{1}{s_1}), \quad s_1\to \infty.
\end{equation}
Both of them satisfy $Q_3(s_2)<0$. Indeed, we have in the limit $Q(s_2^{(1)})\to Q_3(0)=-U^2<0$, and $Q(s_2^{(2)})\to Q_3(-s_1+\frac{U_2}{U'_1})$. Since $Q_3(s)$ is degree 3,  $Q_3(s_1)$ and $Q_3(-s_1)$ have different signs for large $s_1$. 
Using the leading terms of the asymptotics (\ref{eq:asymptotics_s12}) we see that inequality (iii) is valid for $s_2^{(1)}, s_2^{(2)} $ for large $s_1$. Finally, conditions (iv) and (v) can be satisfied by a small perturbation of $s_1$.
\end{proof}

\section{Compact Bonnet pairs}
\label{sec:compact-bonnet-pairs}

We explicitly construct isothermic tori with one family of planar curvature lines that lead to compact Bonnet pairs of genus one.

We start with an observation about the rotational symmetry in the construction. Recalling  Definition~\ref{def:isothermicCylinderFromAFundamentalPiece}, we consider an isothermic torus $f(u,v)$ with one generic family ($u$-curves) of planar curvature lines and $\V$-periodic reparametrization function $\w(v)$ generated by a fundamental piece. Lemma~\ref{lem:rotationalSymmetryIsothermicCylinderFromFundamentalPiece} shows that the periodicity of $\w(v)$ corresponds to a rotational symmetry in $f(u,v)$. From the formulas for the differentials of the resulting Bonnet pairs \eqref{eq:KPP}, we see that the rotational symmetry carries over to the corresponding Bonnet pair tori. In particular, if the rotational symmetry of $f$ is
$
	f(u,v + \V) = \overline{R} f(u,v) R
$
for some fixed unit rotation quaternion $R$ then
\begin{align*}
	df^+(u,v + \V) 	= \overline{R} df^+(u,v) R \quad \text{ and  }	\quad df^-(u,v + \V) = \overline{R} df^-(u,v) R.
\end{align*}
Therefore, each Bonnet pair torus $f^+$ or $f^-$ is generated by rotating its respective transformed fundamental piece of $f$.

\begin{proposition}
	Let $f^+$ and $f^-$ be Bonnet tori arising from an isothermic torus $f$ with one family of planar curvature lines. Then $f^+, f^-$, and $f$ have the same rotational symmetry.
\end{proposition}

\subsection{Existence from isothermic tori with planar and sphe\-rical curvature lines}
\label{sec:Bonnet-one_surface}
\begin{theorem}
	\label{thm:compactBonnetPairsSphericalExistence}
	There exists an analytic isothermic torus parametrized by $f(u,v)$ with one generic family of planar and one family of spherical curvature lines such that the analytic Bonnet pair surfaces $f^\pm(u,v)$, given in Theorem~\ref{thm:BPFromPlanarFormulas}, are tori.
	
	The compact analytic immersions $f^+$ and $f^-$ correspond via a mean curvature preserving isometry that is not a congruence.
\end{theorem}
\begin{proof}
	The result is proven by combining Theorem~\ref{thm:bonnetPeriodicityAsAbelianIntegrals} and Lemma~\ref{lem:exist_s_12}, as follows.
	
	Consider the subset of isothermic surfaces with one family of planar and one family of spherical curvature lines given by Definition~\ref{def:isothermicCylinderFromFundamentalPieceOfRhombicType}. By Lemma~\ref{lem:moduliSpaceIsothermicCylindersFundamentalPieceRhombicType}, every such \emph{isothermic cylinder from a fundamental piece of rhombic type} $f(u,v)$ is determined by four real parameters $\Imc \tau, \delta, s_1, s_2$.
	
	By Lemma~\ref{lem:exist_s_12}, for each $\Imc 
	\tau$, with $0 < \Imc \tau < \Imc \tau_0 $, and sufficiently small $\delta > 0$, there exists $s_1$ and $s_2$ such that the periodicity conditions of Theorem~\ref{thm:bonnetPeriodicityAsAbelianIntegrals} are satisfied. Specifically, the rationality condition \eqref{eq:rationality_integral} is satisfied, closing $f(u,v)$ into a torus, and the vanishing axial \Bpart{}-part condition \eqref{eq:B_part_integral} is satisfied, closing both Bonnet pair surfaces $f^+(u,v)$ and $f^-(u,v)$ into tori.
\end{proof}

\begin{figure}[tbh!]
	\begin{center}
		\includegraphics[width=.4\textwidth]{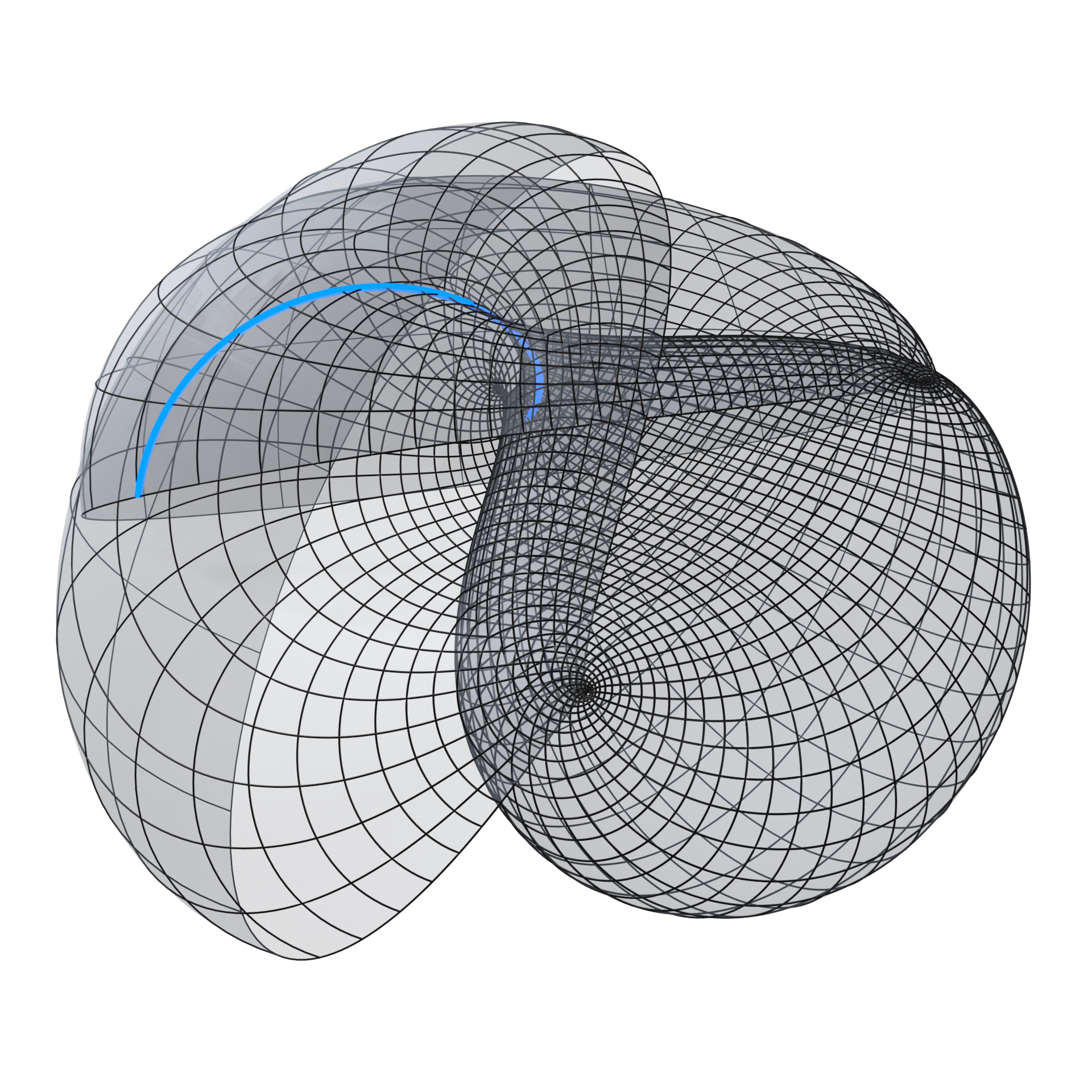}%
		\includegraphics[width=.4\textwidth]{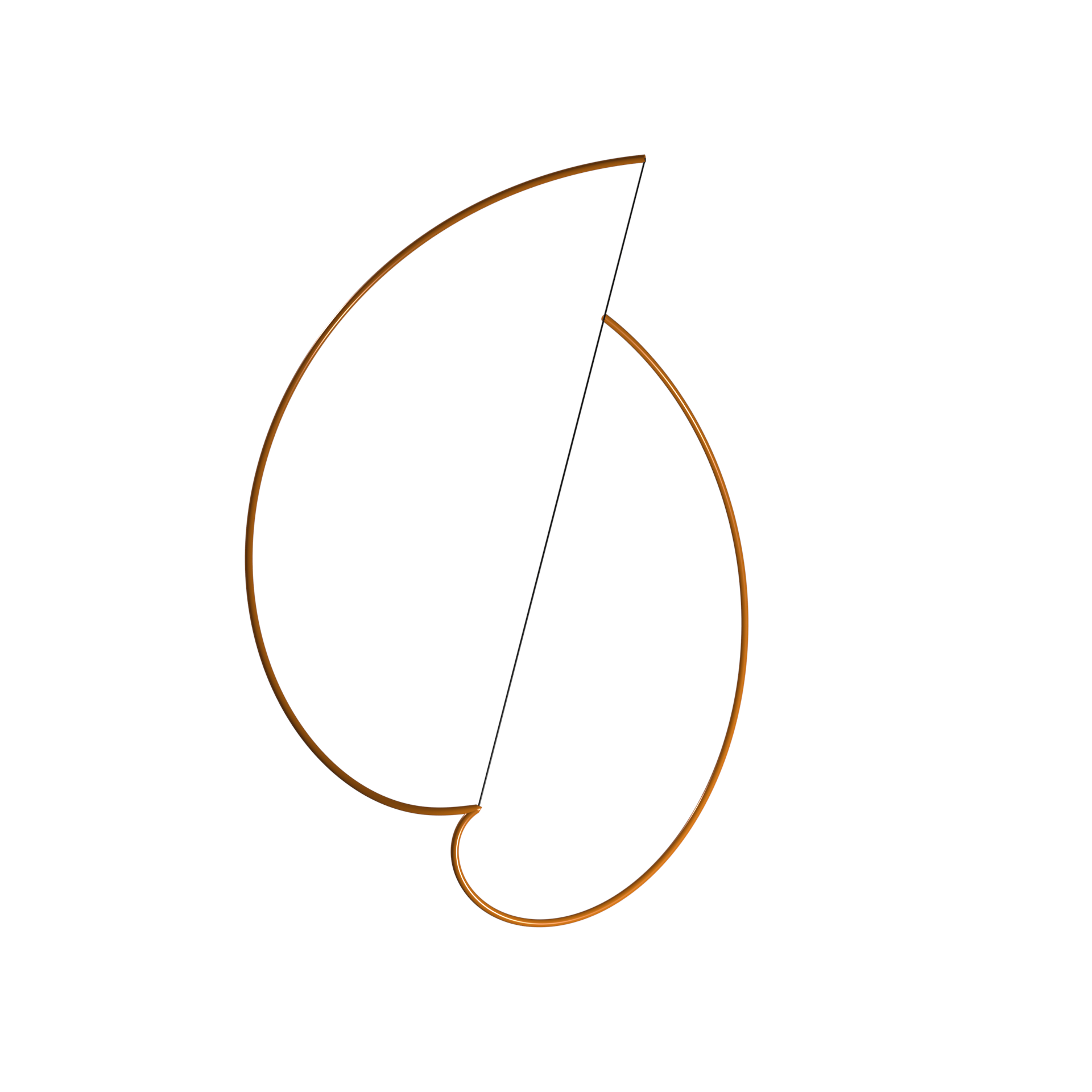}\\
		\includegraphics[width=.4\textwidth]{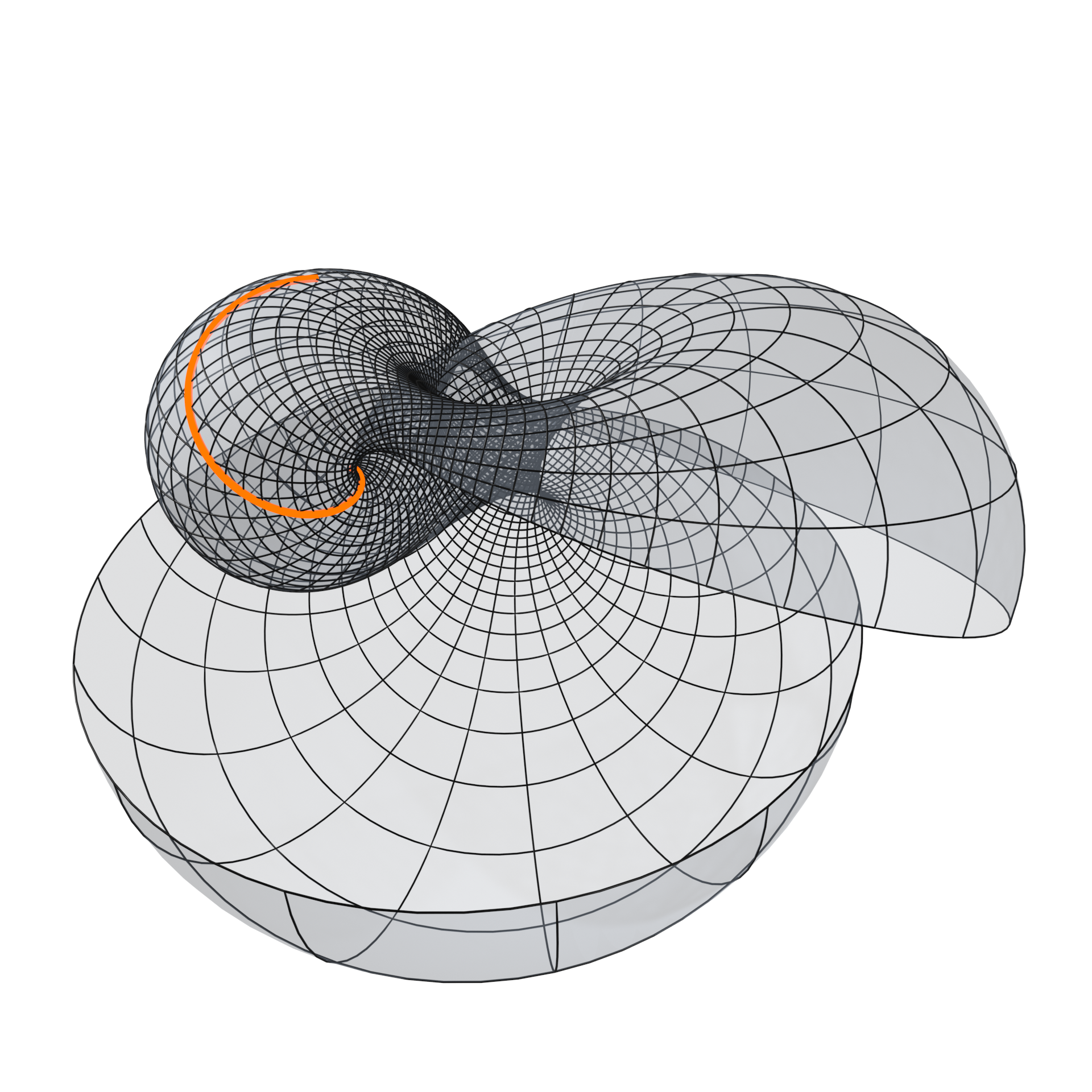}%
		\includegraphics[width=.4\textwidth]{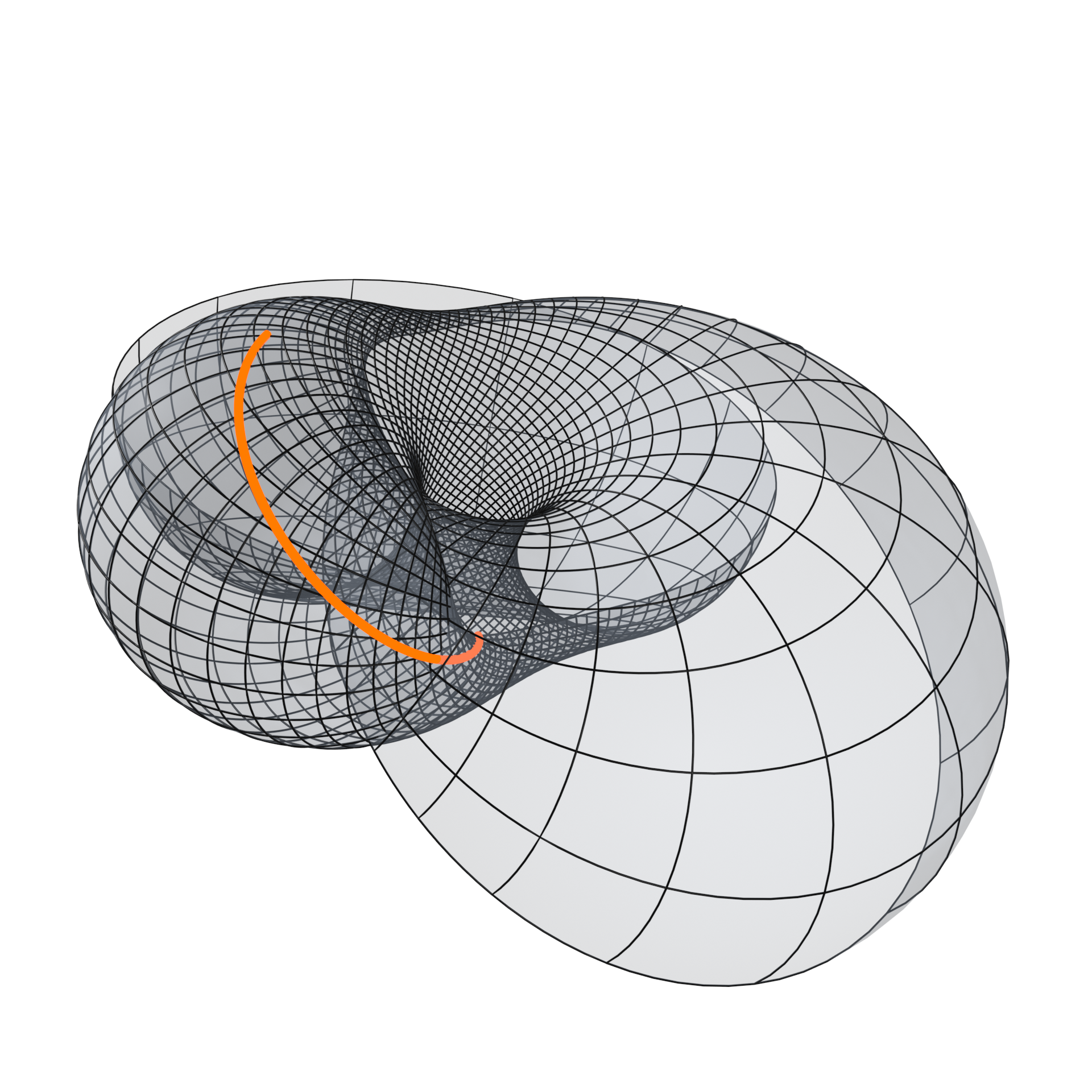}%
	\end{center}
	\caption{Fundamental pieces of an example with 3-fold rotational symmetry from parameters \eqref{eq:3FoldSphericalParameters}. The top left image is the fundamental piece of an isothermic torus with planar and spherical curvature lines. The corresponding portions of the Bonnet tori are shown on the bottom left and right. All three tori are closed by two applications of the same $120^\circ$ rotation. The Bonnet tori have an orange highlighted curve segment that shows the mean curvature preserving isometry (the curve segment on the isothermic torus is shown in blue). The top right image shows the pair of corresponding orange curves after being aligned under Euclidean motions. The Bonnet tori are not congruent as the final endpoints of corresponding curves do not agree.}
	\label{fig:3FoldFundamentalPieces}
\end{figure}

\subsection{Constructing examples}
An isothermic cylinder with planar $u$-curves and the resulting Bonnet pair cylinders are determined by a real parameter $\Imc \tau$, with $0 < \Imc \tau < \Imc \tau_0$, for the $u$-curves and a $\tau$-admissible reparametrization function $\w(v)$ for the $v$-curves. Here is a summary of the construction.

\begin{enumerate}
	\item[1.] Theorem~\ref{thm:planarIsothermicCylinderFormulas} gives the isothermic cylinder $f(u,v)$.The choice of $0 < \Imc \tau < \Imc \tau_0$ determines critical $\omega > 0$. The explicit formulas only require numerical integration of the ODE for $\Phi(v)$.
	\item[2.] Theorem~\ref{thm:BPFromPlanarFormulas} gives the Bonnet pair cylinders $f^\pm(u,v)$. The explicit formulas only require numerical integration of the ODE for $\tilde B(v)$, see Appendix~\ref{sec:explicit-bonnet-formulas}.
\end{enumerate}
For all $0 < \Imc \tau < \Imc \tau_0$ and $\w(v)$ these formulas produce cylinders. To construct examples of compact Bonnet pairs from an isothermic torus with planar and spherical curvature lines we choose $\w(v)$ as follows.

\begin{enumerate}
	\item[3.] Consider $\w(v)$ as in \eqref{eq:wOfVDoublyRhombic}, seen as a reduction of the $\w(v)$ in Corollary~\ref{cor:cylinderEllipticCurves} that characterize spherical $v$-curves. Three real parameters $\delta\neq0, s_1, s_2$ remain to be determined.
	\item[4.] Fix one of the parameters, say $s_1$. Choose a target rational angle $\theta_0 \in \pi \Q$ for the fundamental piece, and then numerically solve for the other two parameters, say $\delta$ and $s_2$, so that the rationality condition \eqref{eq:rationality_integral} and the vanishing axial \Bpart{}-part condition \eqref{eq:B_part_integral} are satisfied. Note that both conditions are expressed as elliptic integrals, which can be desingularized through a change of variables to allow for stable numerics during root finding.
\end{enumerate}

\begin{figure}[tbh!]
	\begin{center}
		\includegraphics[width=.4\textwidth]{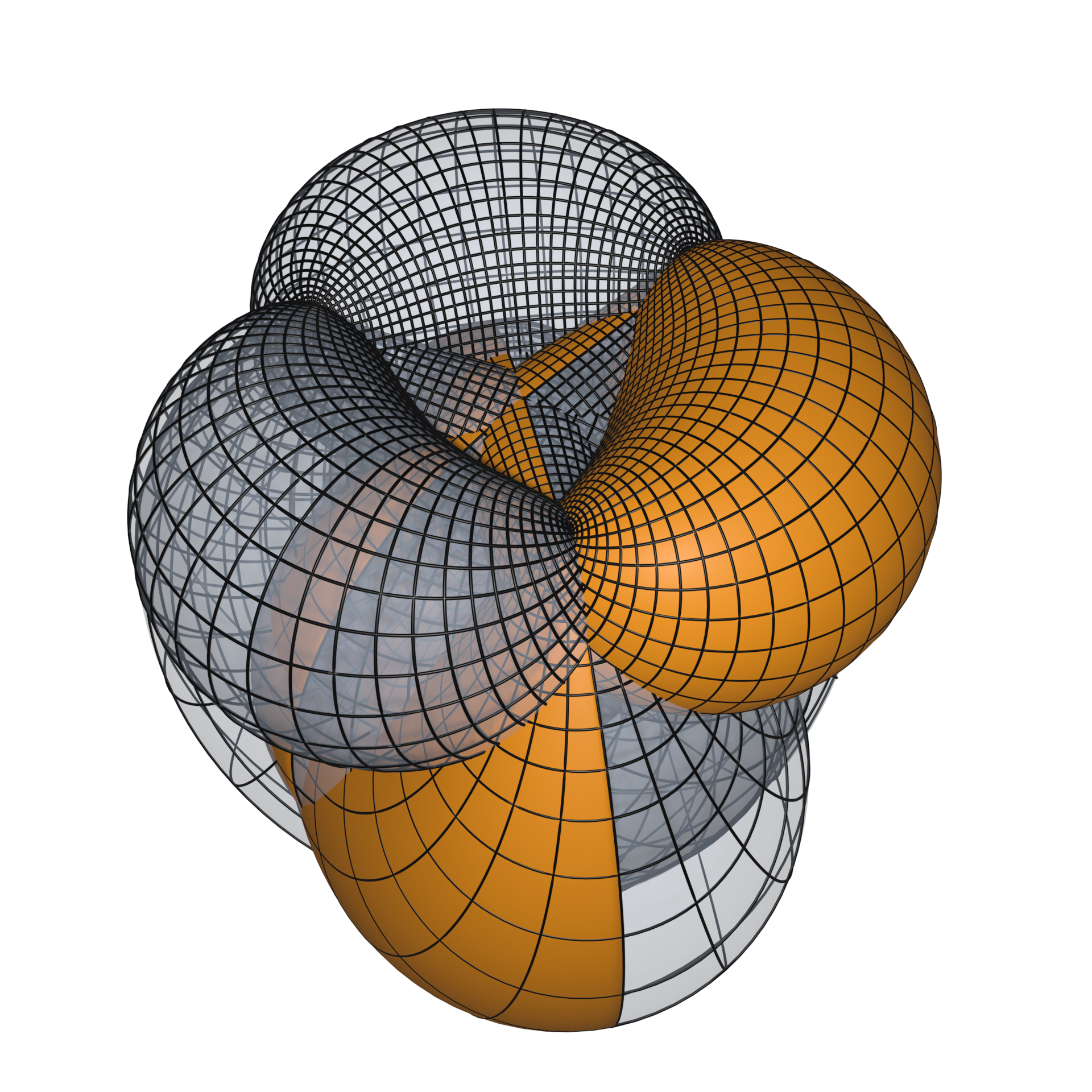}\\
		\includegraphics[width=.4\textwidth]{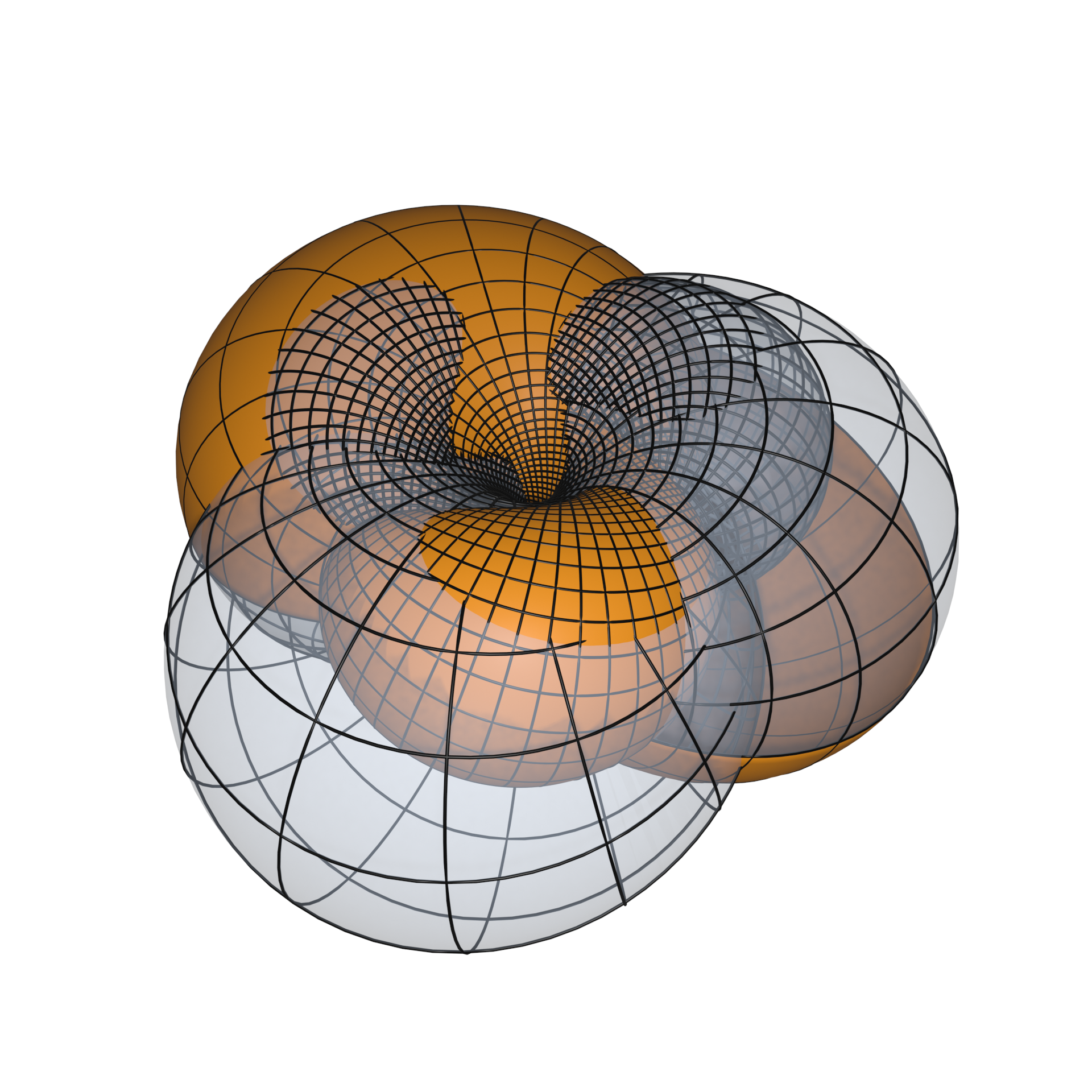}%
		\includegraphics[width=.4\textwidth]{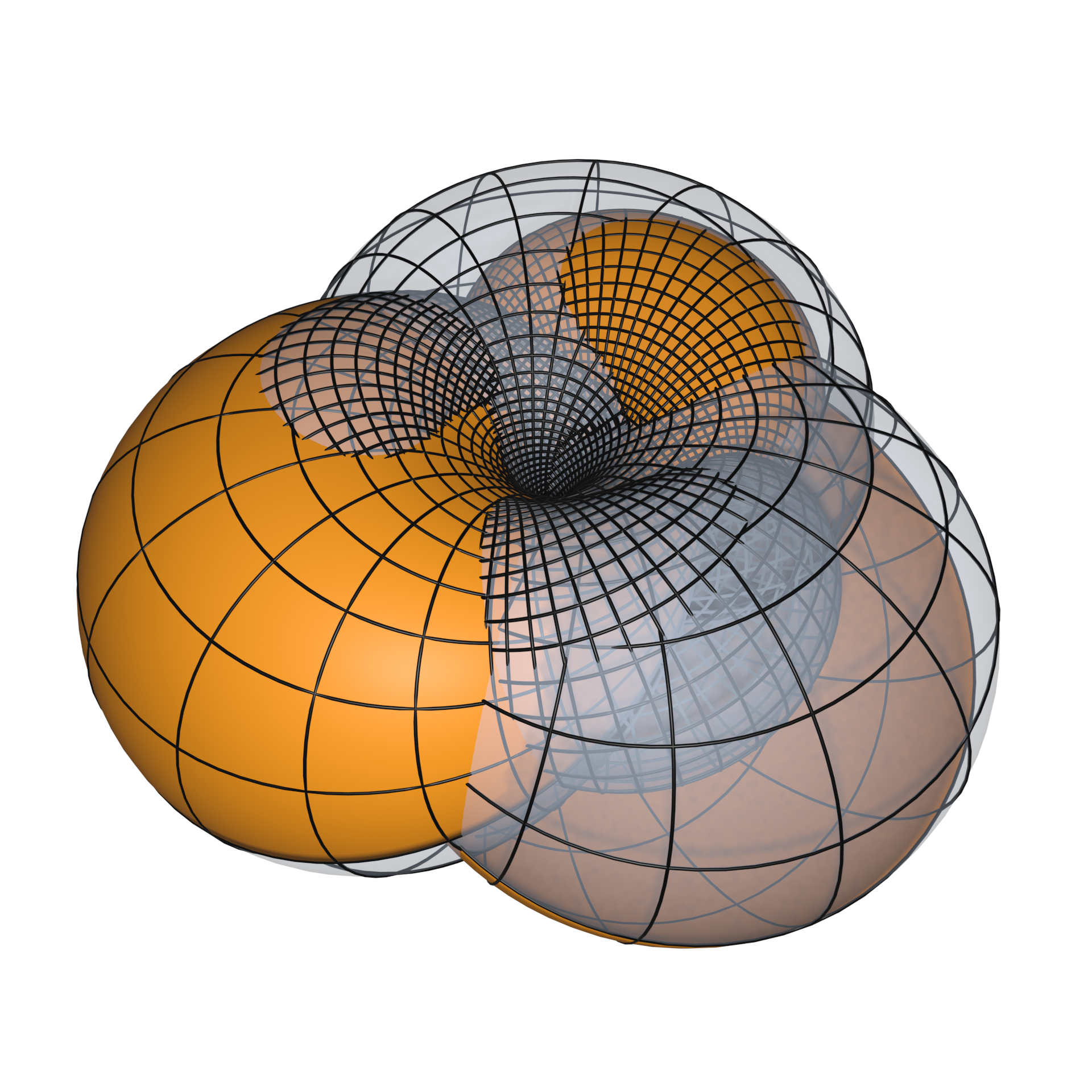}%
	\end{center}
	\caption{The example with 3-fold rotational symmetry from parameters \eqref{eq:3FoldSphericalParameters}, whose construction from a fundamental piece is shown in Figure~\ref{fig:3FoldFundamentalPieces} and highlighted here in orange. The top image is the isothermic torus with planar and spherical curvature lines, while the bottom left and right images are the corresponding Bonnet tori. All three tori share the same $120^\circ$ rotational symmetry.}
	\label{fig:3FoldFullTori}
\end{figure}

We implemented the above procedure in \emph{Mathematica}~\cite{Mathematica} to construct examples with 3-fold and 4-fold rotational symmetry. The following parameters lead to an example where the isothermic torus has a fundamental piece of angle $2\pi/3$.
\begin{equation}
	\label{eq:3FoldSphericalParameters}
	\begin{aligned}
		&\Imc \tau 	\approx 0.3205128205 \text{ with }  
		\omega 		\approx 0.3890180475,	\\
		&\delta 		\approx 1.897366596, \quad
		s_1 		\approx -3.601381552, \quad
		s_2 		\approx 0.5965202011.
	\end{aligned}
\end{equation}
The isothermic fundamental piece and corresponding portions of the Bonnet tori are shown in Figure~\ref{fig:3FoldFundamentalPieces}. The full tori are shown in Figure~\ref{fig:3FoldFullTori}. The following parameters lead to an example where the isothermic torus has a fundamental piece of angle $\pi/2$.
\begin{equation}
	\label{eq:4FoldSphericalParameters}
	\begin{aligned}
		&\Imc \tau 	\approx 0.3205128205 \text{ with }
		\omega 		\approx 0.3890180475,	\\
		&\delta 		\approx 1.61245155, \quad
		s_1 		\approx -3.13060628, \quad
		s_2 		\approx 0.5655771591.
	\end{aligned}
\end{equation}
The fundamental pieces and tori are shown in Figure~\ref{fig:4FoldFullTori}.

\begin{figure}[tbh!]
		\begin{center}
		\includegraphics[width=.35\textwidth]{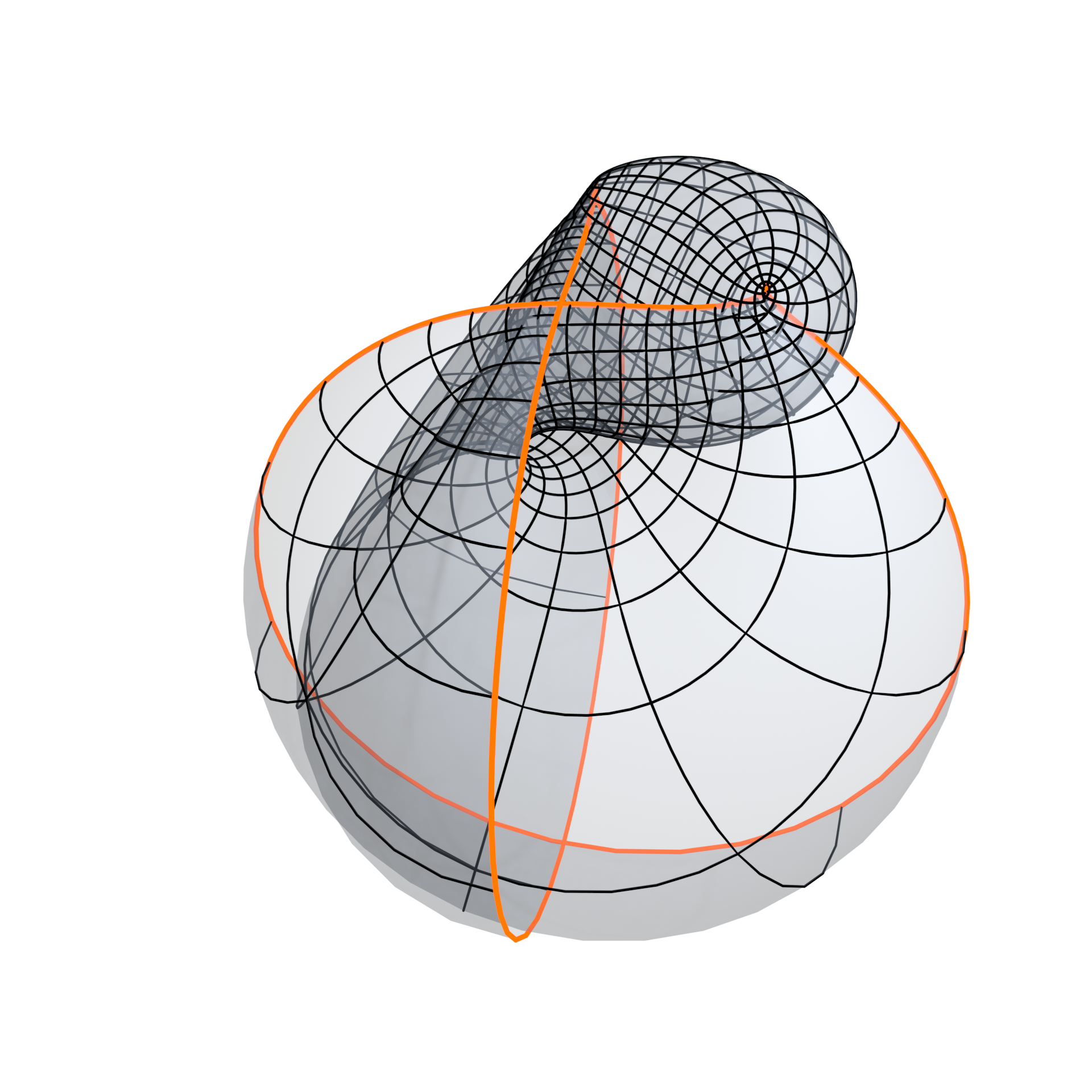}%
		\includegraphics[width=.35\textwidth]{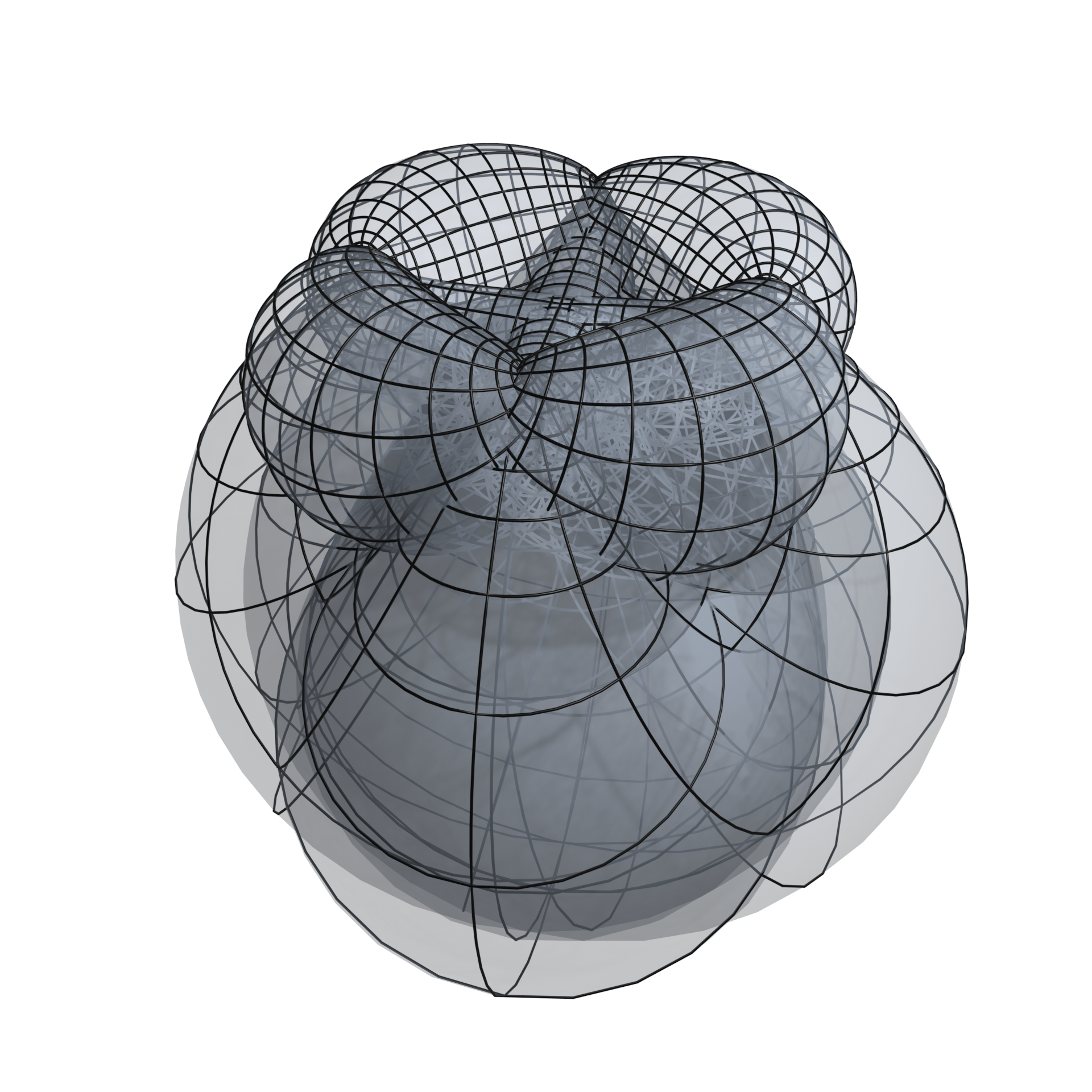}\\
		\includegraphics[width=.35\textwidth]{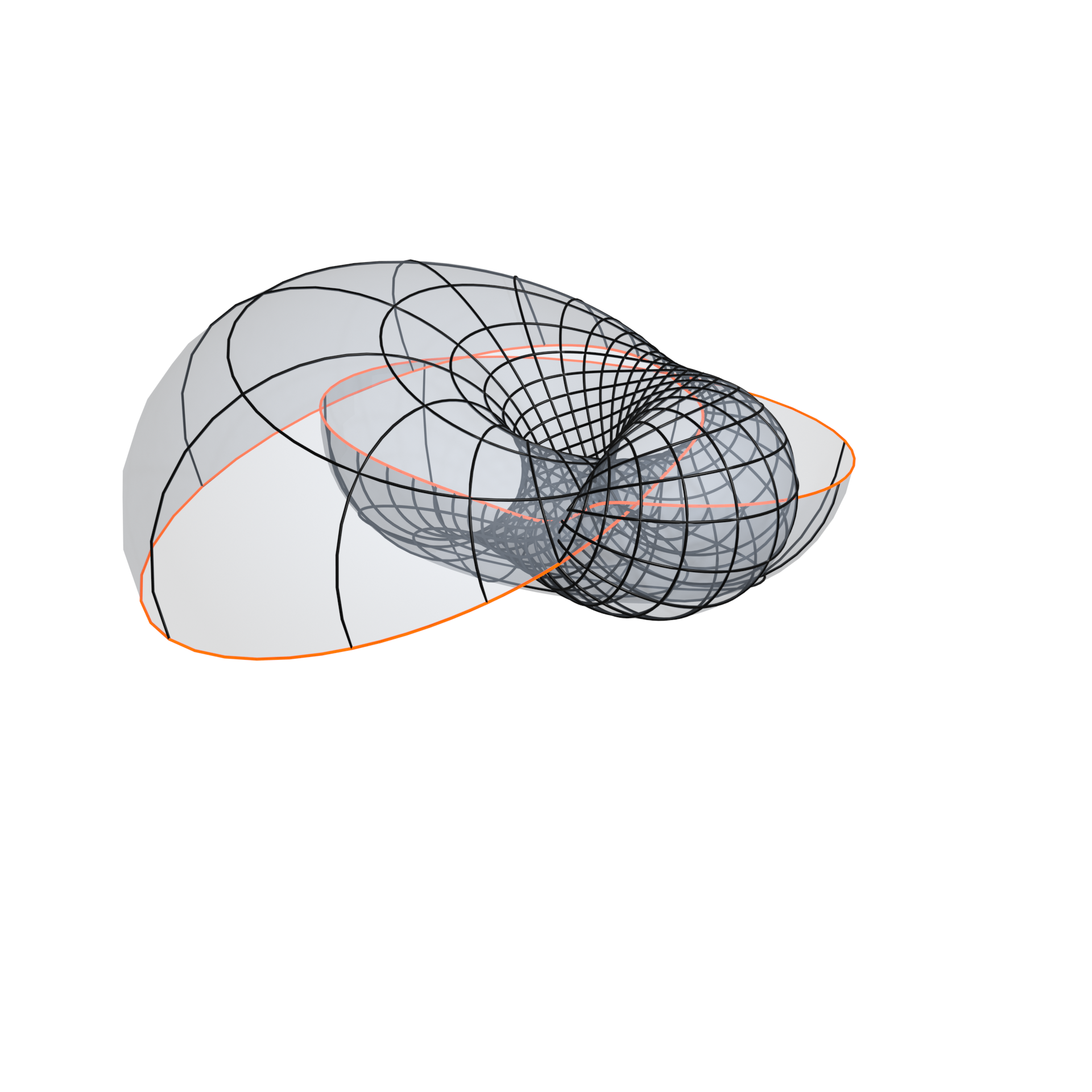}%
		\includegraphics[width=.35\textwidth]{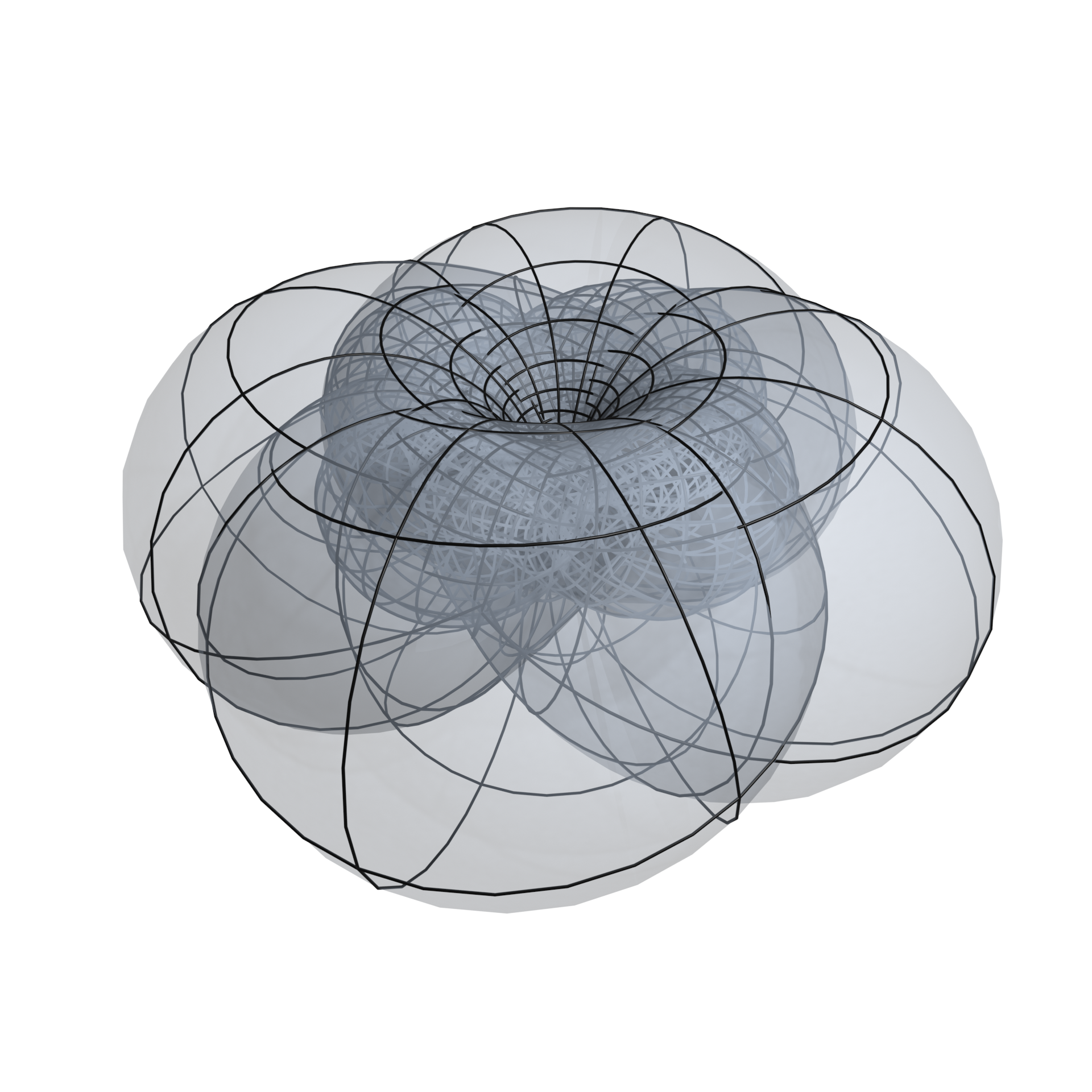} \\
		\includegraphics[width=.35\textwidth]{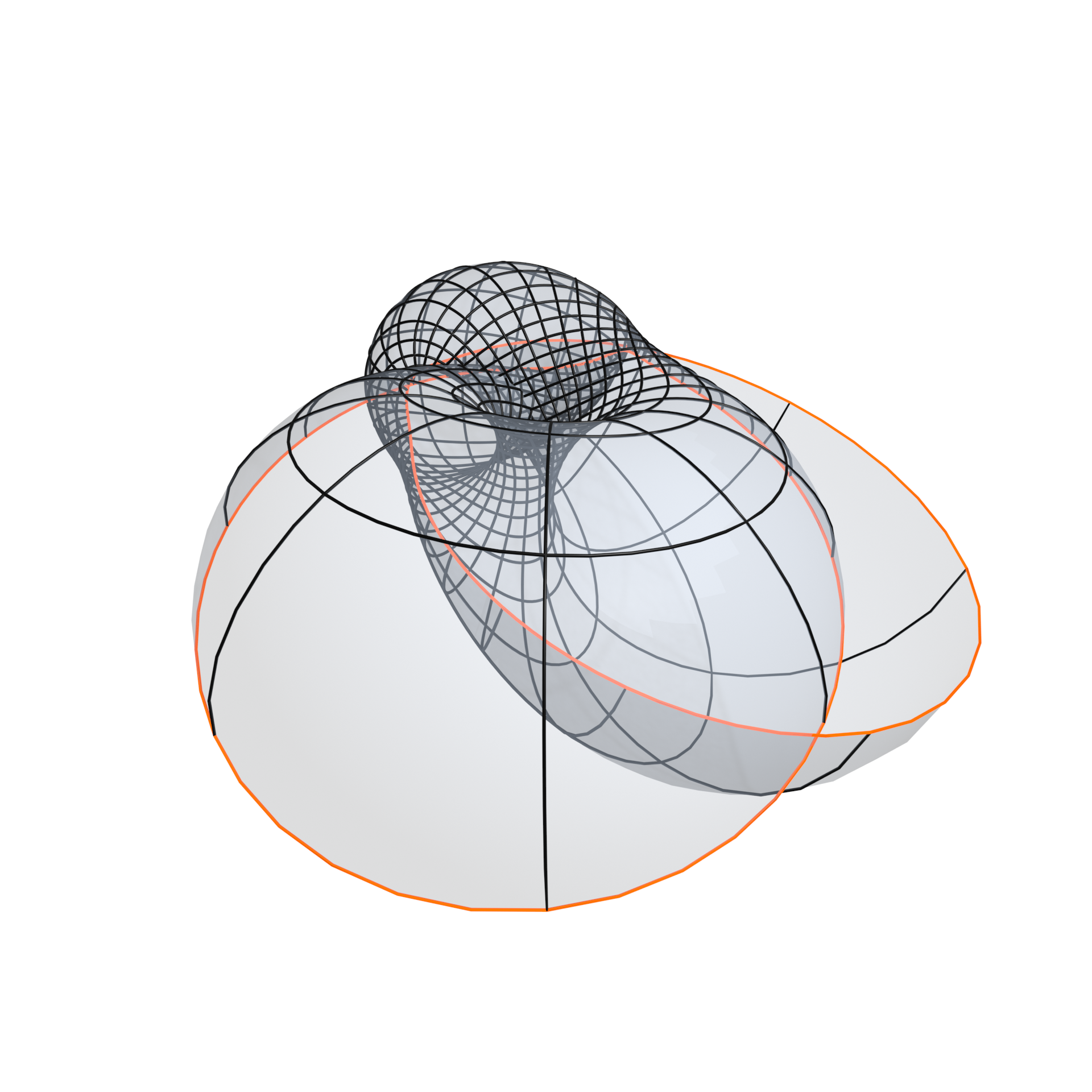}%
		\includegraphics[width=.35\textwidth]{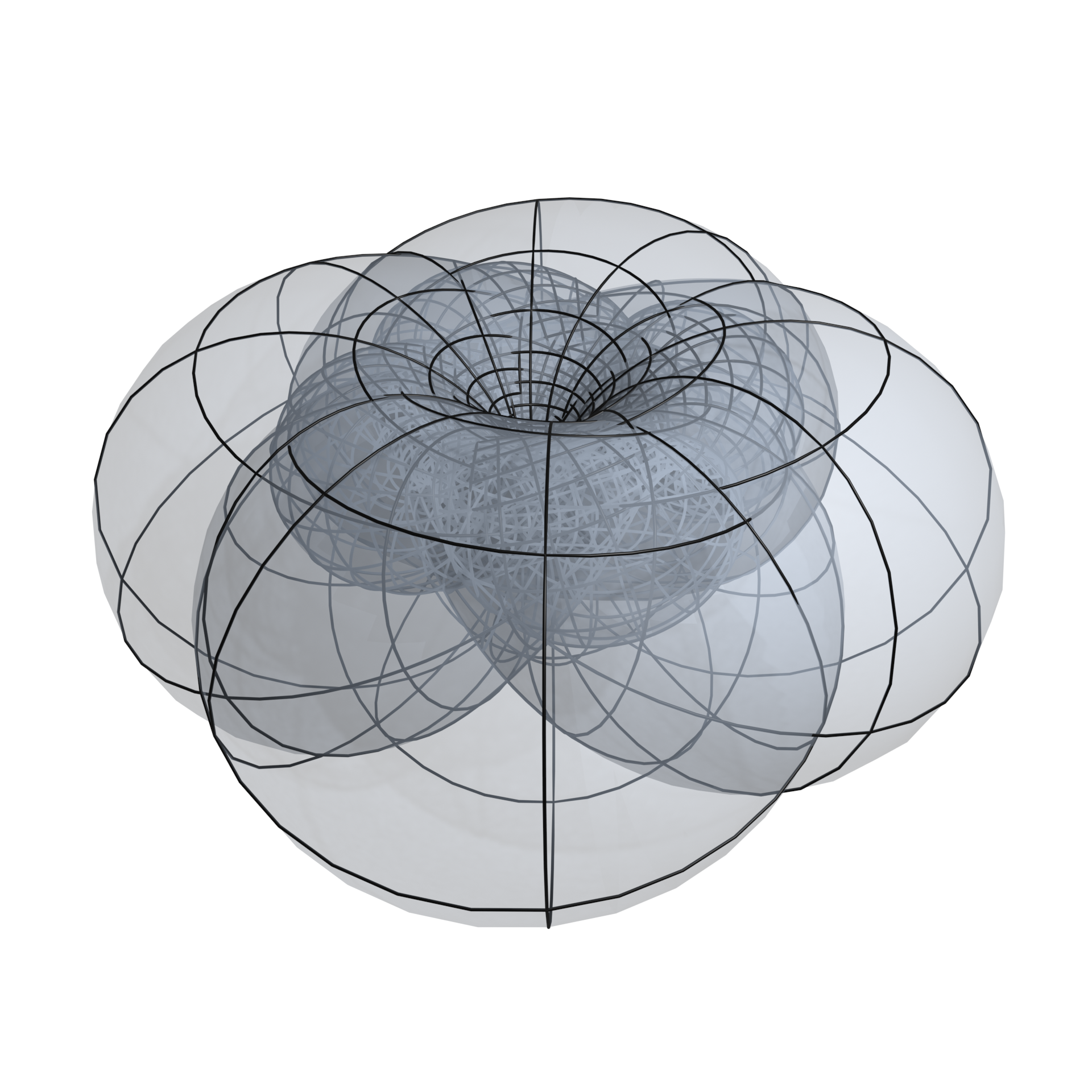}%
	\end{center}
	\caption{An example with 4-fold rotational symmetry from parameters \eqref{eq:4FoldSphericalParameters}. Top shows the fundamental piece and full isothermic torus with planar and spherical curvature lines. The corresponding Bonnet cylinders and tori are shown to the middle and bottom. All three tori share the same $90^\circ$ rotational symmetry. The congruent planar boundary curves of the fundamental piece, and their corresponding curves on the Bonnet cylinders, are shown in orange.}
	\label{fig:4FoldFullTori}
\end{figure}

\begin{remark} (One surface and its reflection)
	\label{rem:oneSurface}
	The plots of the examples reveal that the immersed Bonnet pair tori $f^+$ and $f^-$ globally resemble each other, even though they correspond via a mean curvature preserving isometry that is not a congruence. The explanation comes from the reparametrization functions $\w(v)$ that lead to spherical $v$-curves. They have the symmetry $\w(\V/2 + v) = \w(\V/2-v)$. This leads to a reflectional symmetry in the fundamental piece of the isothermic surface $f$, which implies that the corresponding Bonnet tori $f^+$ and $f^-$ are mirror images of each other. Note that the mirror symmetry mapping $f^+$ to $f^-$ is not the mean curvature preserving isometry.
\end{remark}

\subsection{Existence of examples with less symmetry. Compact Bonnet pairs with 2 different surfaces}
\label{sec:Bonnet_two_surfaces}
The Bonnet pairs constructed in Section~\ref{sec:Bonnet-one_surface} are represented by one surface and its reflection, see Remark~\ref{rem:oneSurface}. These surfaces possess an intrinsic isometry preserving the mean curvature, which is not a congruence. In this section, by a small perturbation, we construct Bonnet pairs represented by two surfaces not related by reflection. The construction has a functional freedom.

Let $\w_0(v)$ be a $\tau$-admissible reparametrization function from one of the examples constructed in Section ~\ref{sec:Bonnet-one_surface}, see Figure~\ref{fig:wAndWPrimeAndPsiPlot-3fold-example}. It is a periodic function with period ${\V}$. Both periodicity conditions of Theorem~\ref{thm:isothermicCylinderBonnetPeriodicityConditions} are satisfied by $\w_0(v)$:
\begin{eqnarray*}
& & \theta_0\in \pi {\mathbb Q},\\
& & b_0:=\langle \axisDir_0, \int_0^{\V} e^{-h(\omega, \w_0(v))} \Phi_0^{-1}(v)\hat{n}(\omega, \w_0(v))\Phi_0 (v)dv \rangle =0.
\end{eqnarray*}
Here the axis $\axisDir_0$ and the dihedral angle $\theta_0$ are determined by the monodromy (\ref{eq:monodromy})
$$
M_0=\Phi_0^{-1}(0)\Phi_0({\cal V})=\cos \frac{\theta_0}{2}+\sin \frac{\theta_0}{2} \axisDir_0,
$$
and $\hat{n}$ is determined by \eqref{eq:nIsothermicPlanarClosed}.

The frame $\Phi_0$ satisfies 
$$
\frac{d}{dv}\Phi_0=\xi (\w_0(v)) \Phi_0,
$$
with $\xi(\w_0(v)))$ given by the explicit formula (\ref{eq:phiPrimePhiInverse}).

Note that $\xi$ and $\hat{n}$ depend on $\w$ and $\w'$.
\begin{figure}
	\centering
	\includegraphics[width=\linewidth]{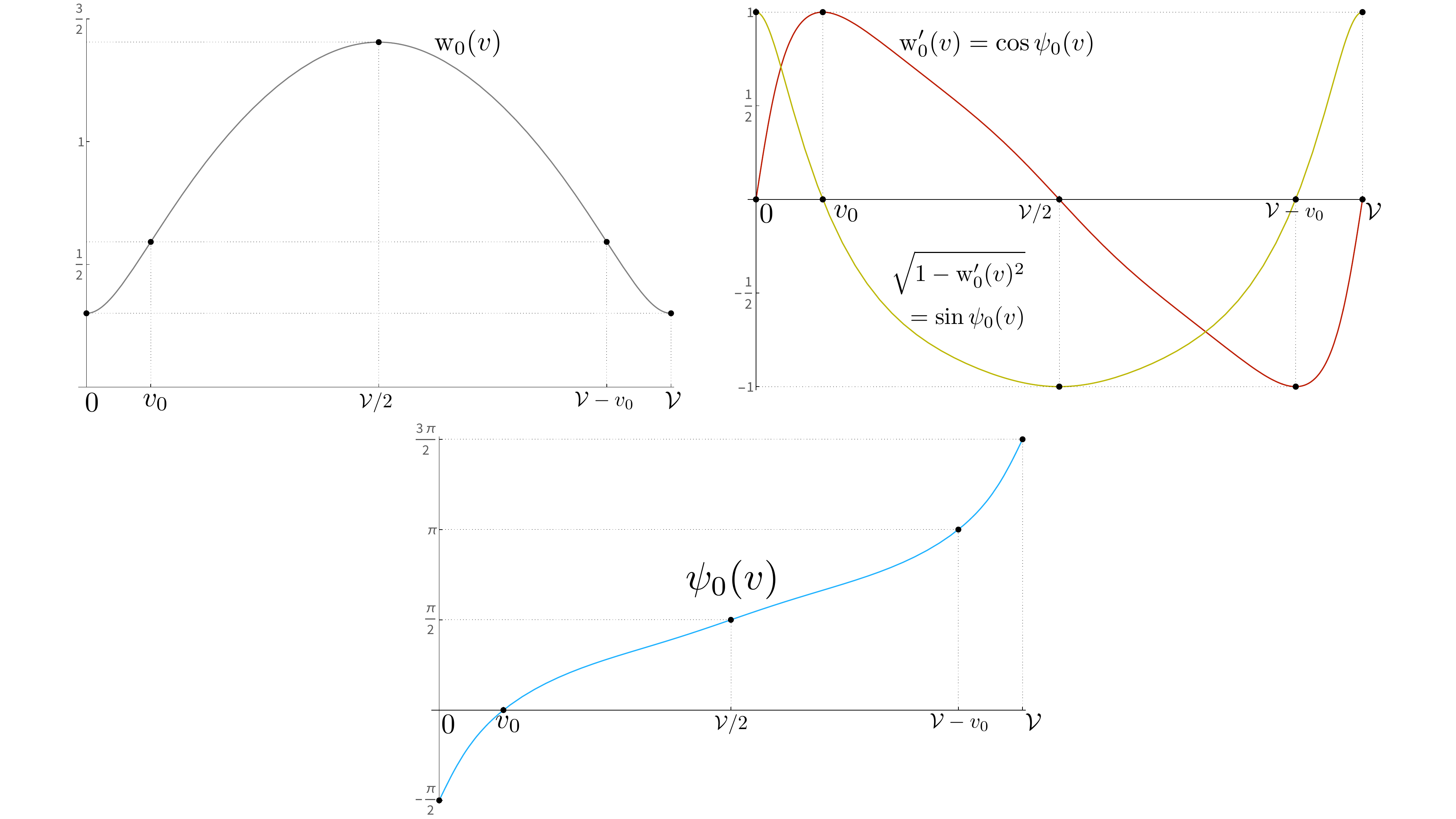}
	\caption{Reparametrization function from 3-fold spherical example in Figures~\ref{fig:3FoldFundamentalPieces} and \ref{fig:3FoldFullTori}. Top left: $\w_0(v)$, Top right: $\w_0'(v) = \cos \psi_0(v), \sqrt{1-\w_0'(v)^2} = \sin \psi_0(v)$, Bottom: $\psi_0(v).$}
	\label{fig:wAndWPrimeAndPsiPlot-3fold-example}
\end{figure}

We consider a $\tau$-admissible small pertubation $\w(v)$ of $\w_0(v)$ 
that is periodic with the same period. As suggested in Remark~\ref{re: psi}, we define $\w(v)$ in terms of a real analytic function $\psi: \S^1\to \R$ by
\begin{equation}
\label{eq:psi-reparametrization}
\w'(v)=\cos\psi(v),\quad \sqrt{1-\w'(v)^2}=\sin\psi(v),
\end{equation}

Thus in the following $\psi_0(v)$ is the function corresponding to $\w_0(v)$, see Figure~\ref{fig:wAndWPrimeAndPsiPlot-3fold-example}, and 
\begin{equation}
\label{eq:phi-perturbation}
\psi(v):=\psi_0(v)+\epsilon\phi(v)
\end{equation}
is its small perturbation. The properties of $\w_0(v)$ in the proof of Lemma~\ref{lem:globalSphericalWOfv} show that $\psi_0(v)$ is a monotonic real analytic function on the interval  $[0, \mathcal{V}]$ with $\psi_0(v+ \mathcal{V})=\psi_0(v)+2\pi$. Let $v_0$ and $\mathcal{V}-v_0$ be the two zeros of $\sqrt{1-\w_0'(v)^2}$ on $[0, \mathcal{V}]$, i.e. $\psi_0(v_0)=0, \psi_0(\mathcal{V}-v_0)=\pi$. We keep these zeros for $\sqrt{1-\w'(v)^2}$ and introduce the perturbation space
\begin{equation}
\label{eq:perturbation-space}
\phi\in C_0^\omega [0,\mathcal{V}]:=
\{ \phi\  \text{real analytic},\  \phi(v+\mathcal{V})=\phi(v) \mid \phi(v_0)=\phi(\mathcal{V}-v_0)=0 \}.
\end{equation}
For $\epsilon$ small enough, $\psi(v)$ defined by (\ref{eq:phi-perturbation}) determines via (\ref{eq:psi-reparametrization}) a $\tau$-admissible reparametrization function
$$
\w(v)=\int_0^v \cos\psi(t) dt + \w_0(0),
$$
provided 
$
\int_0^\mathcal{V} \cos\psi(v) dv=0.
$
The last condition implies the periodicity of $\w(v)$. Its derivative at $\w_0(v)$ is given by 
$$
\frac{\partial \w}{\partial \epsilon}_{| \epsilon =0}= - \int_0^v \sin \psi_0(t) \phi(t) dt.
$$

Thus one obtains  Bonnet tori if the following three conditions are satisfied:
\begin{eqnarray}
\theta & \in & \pi \Q, \label{eq: per1}\\
b & :=&\langle \axisDir, \int_0^{\V} e^{-h(\omega, \w(v))} \Phi^{-1}(v)\hat{n}(\omega, \w(v))\Phi (v)dv \rangle =0,
\label{eq: per2}\\
c & := & \int_0^\mathcal{V} \cos\psi(v) dv=0.  \label{eq: per3}
\end{eqnarray}
We compute the derivatives of these functions at $\epsilon=0$ and apply the implicit function theorem. For condition \ref{eq: per3} one obviously has
$$
\frac{\partial c}{\partial \epsilon}_{| \epsilon =0}= - \int_0^{\mathcal{V}} \sin \psi_0(t) \phi(t) dt.
$$

Formulas for the frame and for the normal read as follows:
\begin{eqnarray*}
\xi(\w(v))& =& \sin \psi(v)W_1(\w(v))\qk,\\
\hat{n}(\omega, \w(v)) & = & \qi \cos\psi(v)- \sin\psi(v) e^{\ci \sigma(\omega,\w(v))}\qk .
\end{eqnarray*}

For  $\epsilon\to 0$ the frame satisfies
\begin{equation}
\label{eq:frame_delta}
\begin{aligned}
& \frac{d}{dv}\Phi=(\xi(\w_0)+\epsilon g(v)+o(\epsilon))\Phi,\  \text{with}\ \\
& g(v)=\frac{\partial}{\partial \epsilon} \xi (\w(v))_{|\epsilon=0}=
\cos\psi_0 (v) W_1(\w_0)\qk \phi (v) -
\sin \psi_0(v) \frac{\partial W_1}{\partial \w}_{| \w=\w_0} \qk \left( \int_0^v \sin\psi_0(t) \phi(t) dt    \right).
\end{aligned}
\end{equation}
Substituting $\Phi=\Phi_0 (I+ \epsilon R + o(\epsilon))$ we obtain $\frac{d}{dv} R=\Phi_0^{-1} g \Phi_0$. The normalization $R(0)=0$ yields
$$
R(v)=\int_0^v \Phi_0^{-1}(t) g(t) \Phi_0(t) dt.
$$
For the monodromy of $\Phi$ this implies $M=M_0(I+\epsilon G+ o(\epsilon))$ with
\begin{equation}
\label{eq: G}
 G=\int_0^{\cal V} \Phi_0^{-1}(v) g(v) \Phi_0(v) dv. 
\end{equation}

Differentiating the monodromy by $\epsilon$ at $\epsilon=0$ we obtain
$$
-\frac{1}{2}\sin \frac{\theta_0}{2} \frac{\partial \theta_0}{\partial \epsilon}+
\frac{1}{2}\cos \frac{\theta_0}{2}  \frac{\partial \theta_0}{\partial \epsilon} \axisDir_0+\sin \frac{\theta_0}{2} \frac{\partial \axisDir}{\partial \epsilon}=
(\cos \frac{\theta_0}{2}+\sin \frac{\theta_0}{2} \axisDir_0)G,
$$
which is equivalent to the following expressions for the derivatives of the dihedral angle and the axis:
\begin{eqnarray}
\frac{\partial \theta}{\partial \epsilon}_{|\epsilon=0}&=& 2\langle \axisDir_0, G\rangle ,
\label{eq: theta_der}\\
\frac{\partial \axisDir}{\partial \epsilon}_{|\epsilon=0}&=& \cot \frac{\theta_0}{2} \left(G-\langle \axisDir_0, G\rangle \axisDir_0  \right )+\axisDir_0\times G.
\label{eq: axis_der}
\end{eqnarray}

Integration of (\ref{eq: G}) by parts gives
\begin{eqnarray}
G=& \int_0^{\mathcal{V}} \phi \cos \psi_0 \Phi_0^{-1} W_1 (\w_0) \qk \Phi_0 dv+
\int_0^{\mathcal{V}} \phi \sin \psi_0 \left( \int_0^v \Phi_0^{-1}  \frac{\partial W_1}{\partial \w}_{| \w=\w_0} \qk \Phi_0 \sin \psi_0 (t) dt\right) dv -
\nonumber\\
& \int_0^{\mathcal{V}} \phi \sin \psi_0 dv 
\int_0^{\mathcal{V}} \Phi_0^{-1}  \frac{\partial W_1}{\partial \w}_{| \w=\w_0} \qk \Phi_0 \sin \psi_0 (v)  dv.
\label{eq: GIntegratedByParts}
\end{eqnarray}
This function is of the form $\int_0^{\mathcal V} \hat{G}(v)\phi(v)dv$, and so is
$$
\frac{\partial \theta}{\partial \epsilon}_{| \epsilon =0}= \int_0^{\mathcal{V}} \hat{a}(v)\phi(v) dv.
$$

By a similar but  more involved computation using (\ref{eq: axis_der}) one proves the same fact for the periodicity condition (\ref{eq: per2}):
$$
\frac{\partial b}{\partial \epsilon}_{| \epsilon =0}= \int_0^{\mathcal{V}} \hat{b}(v)\phi(v) dv.
$$

Indeed, 
$$
\frac{\partial b}{\partial \epsilon}_{|\epsilon=0}=
\langle \frac{\partial \axisDir}{\partial \epsilon}, \int_0^{\cal V} e^{-h_0}\Phi_0^{-1}\hat{n}_0\Phi_0 dv\rangle +
\langle \axisDir_0, \int_0^{\cal V}  \frac{\partial}{\partial \epsilon}(e^{-h}\Phi^{-1}\hat{n}\Phi) dv\rangle.
$$
Due to (\ref{eq: axis_der}) the first term is of required form. For the second term we have
\begin{align}
\label{eq: technical1}
 \frac{\partial}{\partial \epsilon}(e^{-h}\Phi^{-1}\hat{n}\Phi)=e^{-h_0}\Phi_0^{-1} \frac{\partial \hat{n}_0}{\partial \epsilon}\Phi_0+
\left( - \frac{\partial h_0}{\partial \w}  e^{-h_0}\Phi_0^{-1}\hat{n}_0\Phi_0+
e^{-h_0}\Phi_0^{-1}\frac{\partial \hat{n}_0}{\partial \w}\Phi_0
\right) \frac{\partial \w}{\partial \epsilon} +
e^{-h_0}[\Phi_0^{-1}\hat{n}_0 \Phi_0, R],
\end{align}
where
\begin{eqnarray*}
\frac{\partial \hat{n}_0}{\partial \epsilon}=(-\qi\sin\psi_0-\cos\psi_0 e^{i\sigma_0}\qk)\phi, \quad
\frac{\partial \hat{n}_0}{\partial \w}=-i\sin\psi_0 \frac{\partial \sigma_0}{\partial \w}e^{i\sigma_0}\qk,
\end{eqnarray*}
and $[,]$ is the commutator. The first term in this formula is of required form. The second one can be brought to this form by integration by parts similar to (\ref{eq: G}). Integrating by parts the last term one obtains
\begin{eqnarray*}
&\int_0^{\V}e^{-h_0(v)}
\left[
\Phi_0^{-1}(v)\hat{n}_0 (v) \Phi_0(v), \int_0^v \Phi_0^{-1}(t)g(t) \Phi_0(t)
\right] dv=\\
&-\int_0^{\V}
\left[
\int_0^v e{-h_0}(t)\Phi_0^{-1}(t)\hat{n}_0(t)\Phi_0(t) dt, \Phi_0^{-1}(v) g (v) \Phi_0(v)
\right] dv+\\
&\int_0^{\V}e^{-h_0}\Phi_0^{-1}(v)\hat{n}_0 (v) \Phi_0 (v)dv \int_0^{\V} \Phi_0^{-1} (v) g (v) \Phi_0(v) dv.
\end{eqnarray*}
One more integration by parts brings this term to the required form.

Now we apply the implicit function theorem to find perturbations preserving the periodicity conditions. The derivatives of all three periodicity conditions are of the form $\int_0^{\V} \phi (v) \hat{c}(v)dv =0 $ with explicit functions $\hat{c}(v)$.

Since the functions $\hat{a}(v), \hat{b}(v)$ and $\sin\psi_0(v)$ are linearly independent, there exist functions $\phi_1, \phi_2, \phi_3$ with 
\begin{eqnarray*}
\frac{\partial}{\partial \epsilon}(\theta, b, c)_{|\epsilon=0, \phi=\phi_1} = (1,0,0),\quad
\frac{\partial}{\partial \epsilon}(\theta, b, c)_{|\epsilon=0, \phi=\phi_2}  = (0,1,0),\quad
\frac{\partial}{\partial \epsilon}(\theta, b, c)_{|\epsilon=0, \phi=\phi_3}  = (0,0,1).
\end{eqnarray*} 

Now choose 
$
\epsilon \phi=\alpha_1 \phi_1+\alpha_2 \phi_2 + \alpha_3 \phi_3 + \epsilon \chi
$
with an arbitrary $\chi\in C^\omega_0[0, {\cal V}]$. For the map $F(\alpha_1,\alpha_2,\alpha_3, \epsilon)=(\theta, b, c)$ and its Jacobian with respect to $\alpha = (\alpha_1, \alpha_2, \alpha_3)$ we have  $$
F(0,0,0)=(\theta_0,0,0), \quad \frac{\partial F}{\partial \alpha}_{|\epsilon=0}=\rm{Id}.
$$
 By the implicit function theorem, we obtain that for small $\epsilon$ there exist analytic functions $\alpha_1(\epsilon), \alpha_2(\epsilon), \alpha_3(\epsilon)$ with $F(\alpha_1(\epsilon),\alpha_2(\epsilon),\alpha_3(\epsilon),\epsilon)=(\theta_0,0)$. The reparametrization functions $\w=\w_0+\alpha_1(\epsilon)\phi_1+\alpha_2(\epsilon)\phi_2+ \alpha_3(\epsilon)\phi_3+\epsilon\chi$ give us a family of Bonnet pairs depending on a functional parameter. For generic functions $\w(v)$ (that do not possess a reflection symmetry as noted in Remark~\ref{rem:oneSurface}) we obtain two Bonnet tori not related by a reflection. We have therefore proven the following theorem.

\begin{theorem}
\label{thm:Bonnet_pair_two}
There exist Bonnet pairs with analytic tori $f^+$ and $f^-$ not related by an isometry of the ambient space $\R^3$.
\end{theorem}

\begin{corollary}
	The statements of Main Theorems ~\ref{thm:main1} and ~\ref{thm:main2}.
\end{corollary}

\begin{figure}[tbh!]
	\begin{center}
		\includegraphics[width=0.5\textwidth]{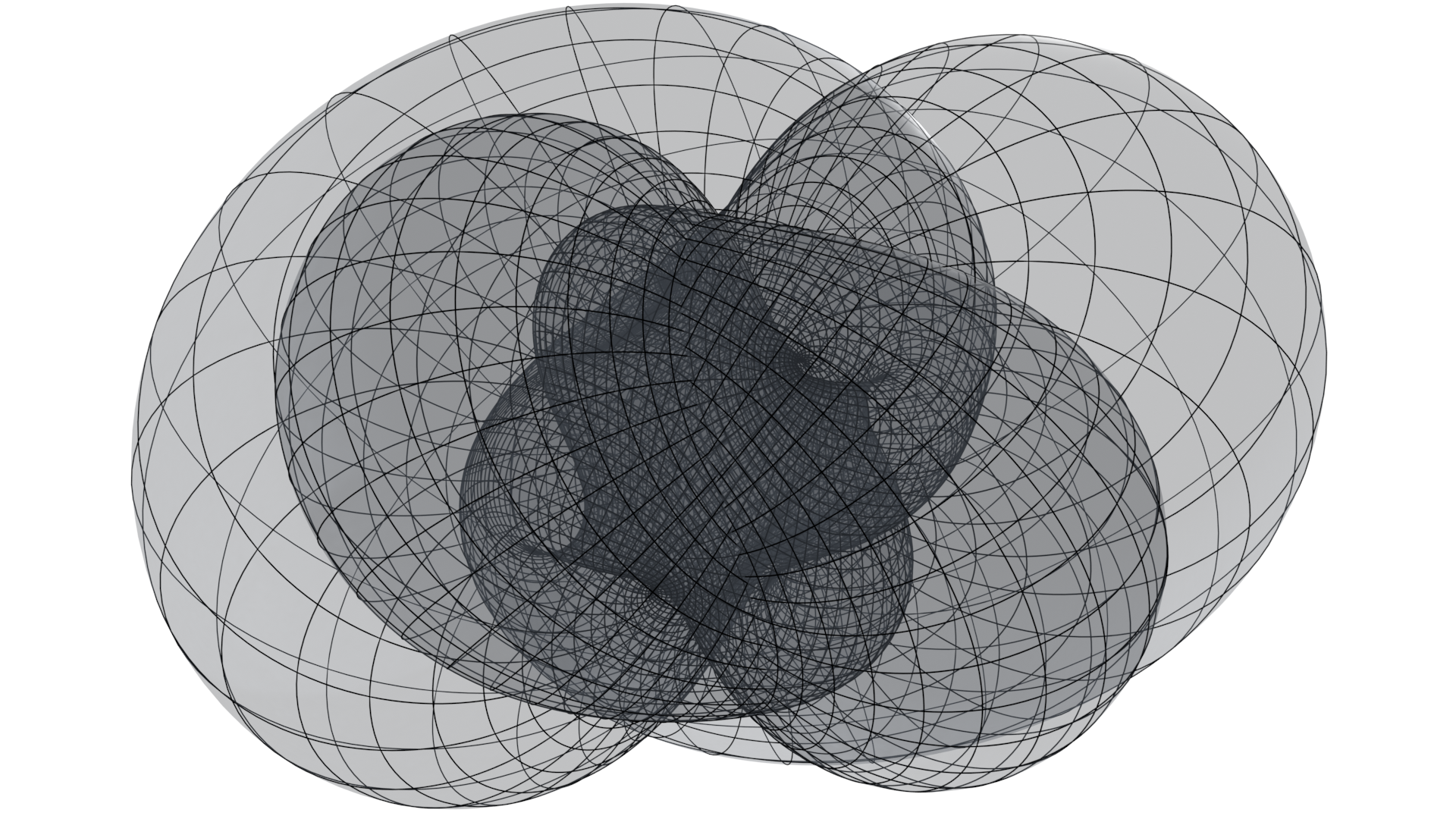}%
	\end{center}
	\caption{The isothermic torus with $180^\circ$ rotational symmetry that generates the Bonnet tori from Figure~\ref{fig:teaserNonsphericalPlusMinus}. The fundamental piece and Bonnet cylinder portions are shown in Figure~\ref{fig:periodicity-conditions-outline}. The resulting Bonnet pair tori are not related by an ambient isometry of $\R^3$.}
	\label{fig:Bonnet_pair_two}
\end{figure}

Figure~\ref{fig:Bonnet_pair_two} shows a numerical example of an isothermic torus that generates a Bonnet pair with two different (i.e. non-congruent and not related by a reflection) tori. The corresponding Bonnet pair are in Figure~\ref{fig:teaserNonsphericalPlusMinus}. This example was constructed as follows. First we fixed the elliptic curve parameter $\Imc \tau$ so that 
\begin{align} \label{eq:imtauApprox}
	\Imc \tau 	\approx 0.3205128205 \text{ with }
	\omega 		\approx 0.3890180475.
\end{align}
Then we considered the following three parameter $A, B, C \in \R$ set of reparame\-trization functions:
\begin{align}
	\w(v) = \rm C + \frac{\rm A \sin (v)}{\pi }-\frac{\rm A \cos (v)}{\pi ^2}-\frac{\rm B \sin (2 v)}{2 \pi }+\frac{\rm B \cos (2 v)}{4 \pi ^2}.
\end{align}
Within this set we numerically solved for an isothermic cylinder satisfying the rationality and vanishing axial $\Bpart{}$ part conditions from Theorem~\ref{thm:isothermicCylinderBonnetPeriodicityConditions}, with 2-fold symmetry $\angleFP_0 = \pi$. The parameters are
\begin{align}
	\rm A = 1.44531765156, \quad \rm B = 1.33527652772, \quad \rm C = 1.05005399924.
\end{align}


\section{Discrete Bonnet pairs}
\label{sec:discrete-theory}
Discrete Differential Geometry studies analogs of the smooth theory that preserve some underlying structure, like the integrability of a compatibility condition or the geometric invariance under a certain transformation group. These ideas have broad application from surface theory and integrable systems to architectural geometry and computer graphics ~\cite{ddgAMS2008,POTTMANN2015145,DDGGlimpse}.

Importantly, discrete properties are preserved at every finite resolution, as opposed to only in a continuum limit. This viewpoint is exemplified by the computational discovery of discrete compact Bonnet pairs shown in Figure~\ref{fig:discretePlanarExample}, which initiated our work on the present article. To describe the setup and experiments, we briefly introduce the necessary ideas.

\begin{figure}
	\begin{center}
		\includegraphics[width=0.48\textwidth]{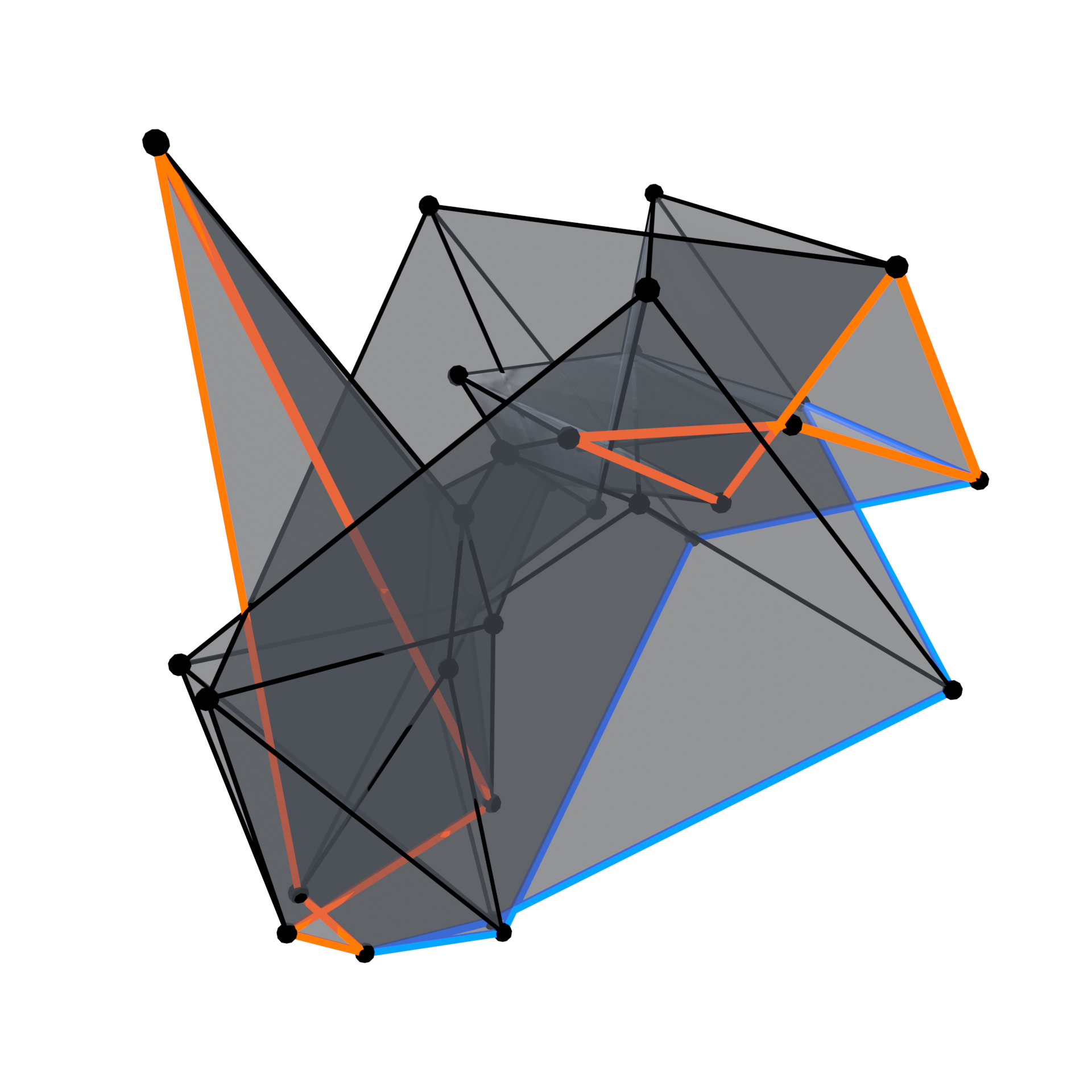}
		\includegraphics[width=0.48\textwidth]{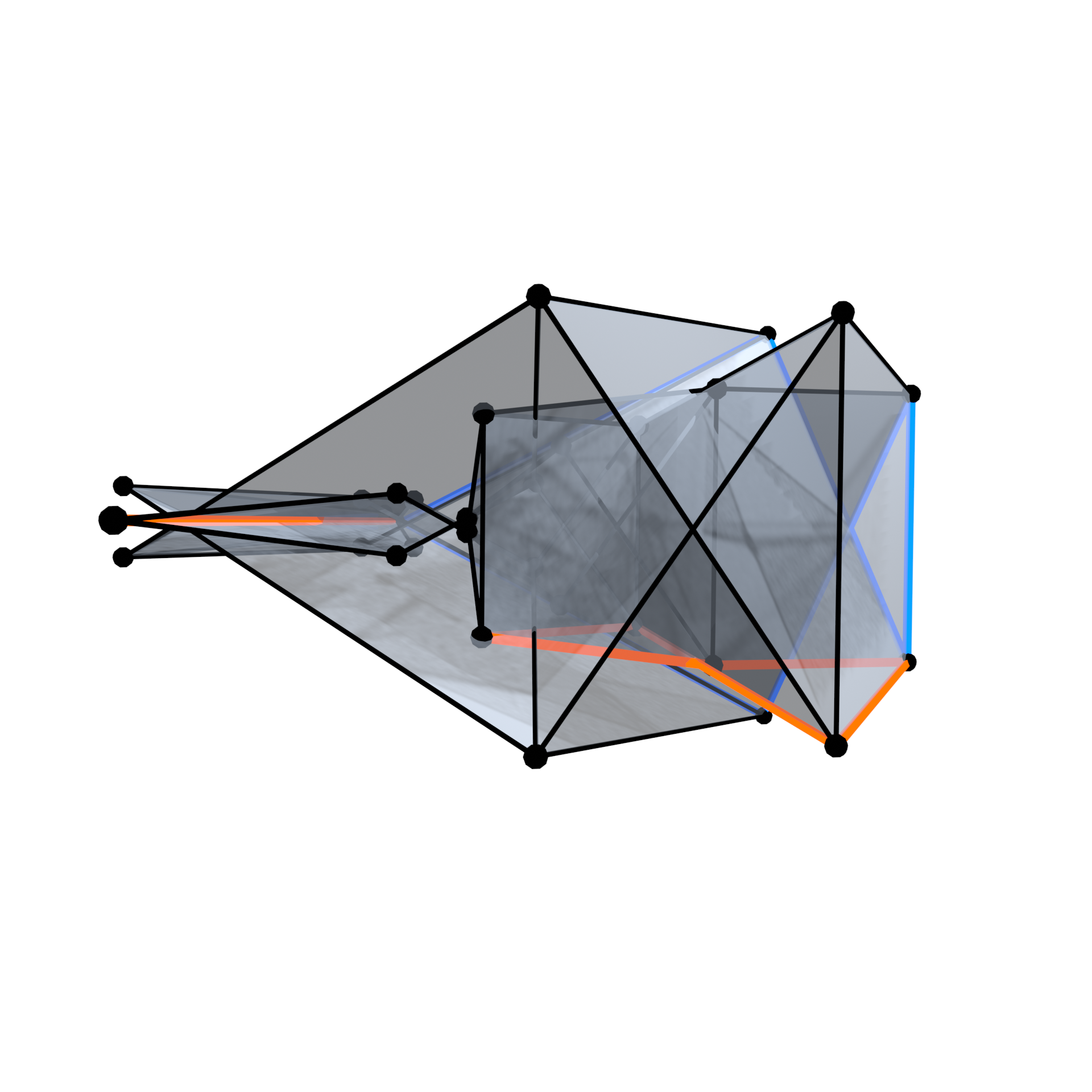}\\
		\includegraphics[width=0.42\textwidth]{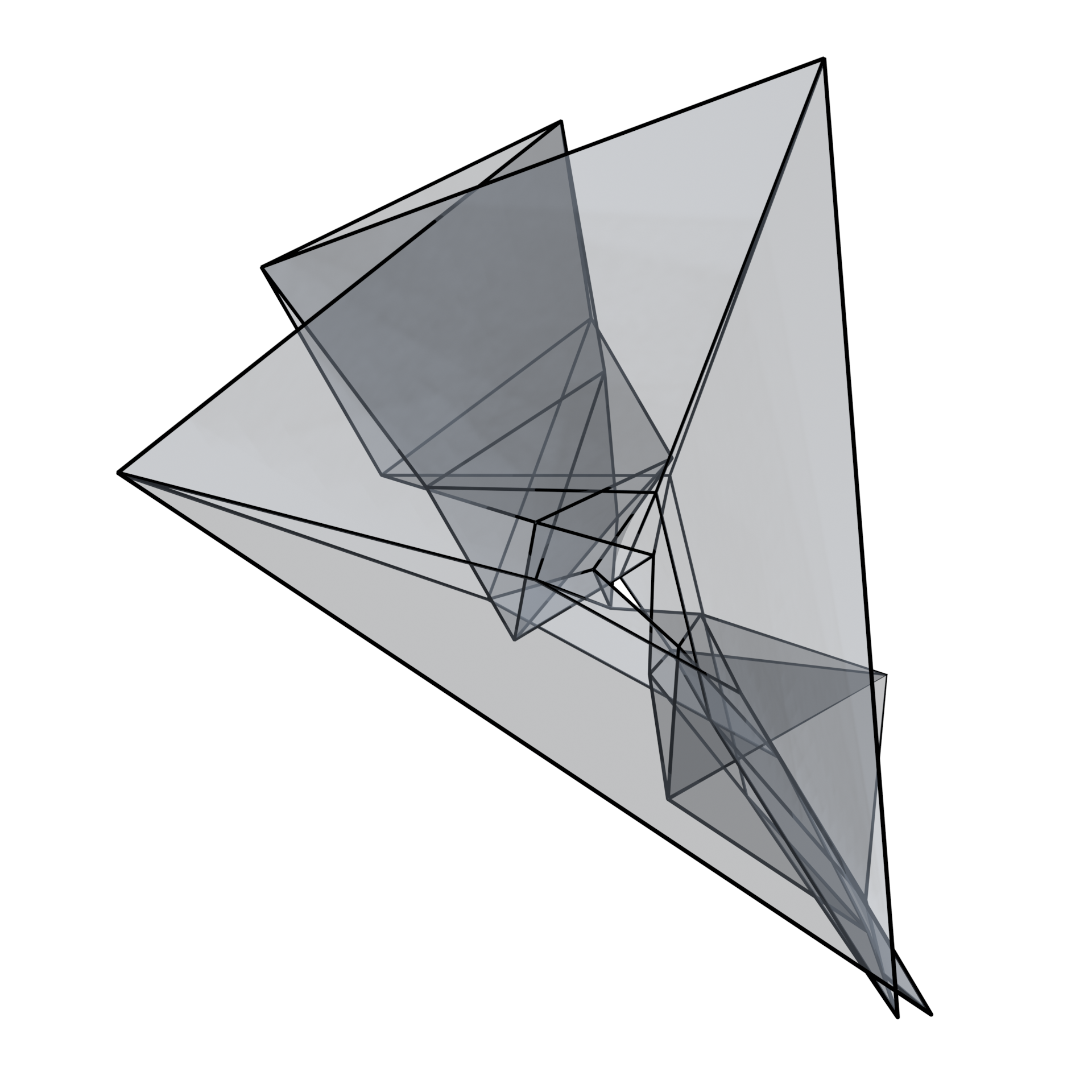}%
		\includegraphics[width=0.42\textwidth]{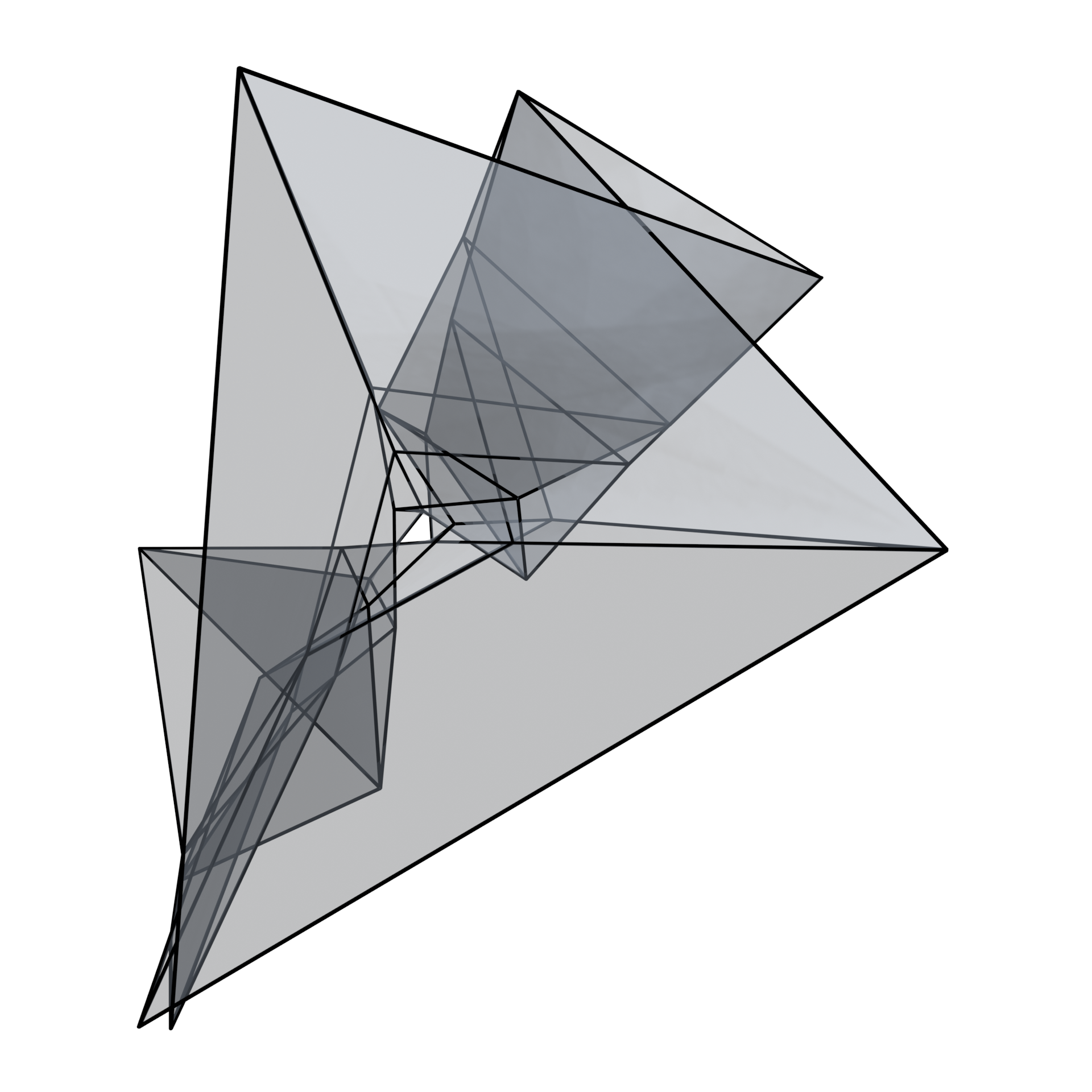}%
	\end{center}
	\caption{Two views (top) of the first numerical example of a discrete isothermic torus that gives rise to a discrete compact Bonnet pair of tori (bottom). This extremely coarse numerical example on a $5\times7$ lattice exhibits the remarkable properties that led to the discovery of smooth compact Bonnet pairs. On the discrete isothermic torus, for example, the curvature lines with 5 vertices are planar (two planar generators are highlighted in orange) and the curvature lines with 7 vertices are spherical (generator highlighted in blue). Moreover, interleaving dualization and inversion operations leads to other examples of discrete isothermic tori that give rise to discrete Bonnet pairs. All these examples of isothermic tori resemble each other, suggesting the symmetries of the smooth isothermic torus in Theorem~\ref{thm:planarIsothermicCylinderGeometry} that are essential to the construction, and derived using the theta function formulas for the family of planar curves.}
	\label{fig:discretePlanarExample}
\end{figure}

An immersed discrete parametrized surface or discrete net is a map from a subset of the standard lattice $f: \Z^2 \to \R^3$ with nonvanishing straight edges in $\R^3$. We use subscripts to denote shifts in a particular lattice direction, i.e., with $f(n,m) = f$ we have $f(n+1, m) = f_1, f(n,m+1) = f_2$ and $f(n+1,m+1) = f_{12}$.

Isothermic surfaces are characterized by having coordinates that are both conformal and curvature line coordinates. The well-studied discrete analog from integrable systems has the following geometric definition~\cite[Definition 6]{DiscreteIsothermic96}.

\begin{definition}
	A map $f : \Z^2 \to \rm{Im} \H = \R^3$ is called a \emph{discrete isothermic net} if for each quad
	\begin{align}
		\label{eq:discreteIsothermicCrossRatioMinusOne}
		(f_1 - f) (f_{12}-f_1)^{-1}(f_{12} - f_2)(f_2-f)^{-1} = -1
	\end{align}
\end{definition}
\begin{remark}
	This definition is equivalent to every quad having coplanar vertices and complex cross-ratio $-1$ in its respective plane, so it can be conformally mapped by a fractional linear transformation to a square. Vertices of each quad lie on a circle. Discrete nets with concircular quads are a discrete analog of curvature line coordinates, see for example \cite{ddgAMS2008}. Hence, these nets can be understood as being in conformal, curvature line coordinates.
\end{remark}

Discrete isothermic nets have dual nets~\cite[Theorem 6]{DiscreteIsothermic96} given by integrating an analogous expression to the smooth one-form \eqref{eq:isothermic_dual}.

\begin{proposition}
	Every discrete isothermic net $f : \Z^2 \to \R^3$ has a dual net $f^* : \Z^2 \to \R^3$ that is defined up to global translation by
	\begin{align}
		f_1^* - f^* = \frac{f_1 - f}{|f_1 - f|^2}, \quad  f_2^* - f^* = \frac{f_2 - f}{|f_2 - f|^2}.
	\end{align}
\end{proposition}

Recall the quaternionic function theory characterization of smooth Bonnet pairs from smooth isothermic surfaces in \cite{KPP}
\begin{align}
	df^\pm = (\pm \epsilon - f) df^* (\pm \epsilon + f).
\end{align}
Analogously, discrete Bonnet pairs have recently been defined from discrete isothermic nets.

\begin{proposition}
	Let $f : \Z^2 \to \R^3$ be a discrete isothermic net with dual net $f^* : \Z^2 \to \R^3$. Then for all $\epsilon \in \R$ the transformations defined by
	\begin{equation}
		\begin{aligned}
		f^\pm_1 - f^\pm &= \Imq \bigg( (\pm \epsilon - f) (f_1^* - f^*) (\pm\epsilon + f_1) \bigg),\\	
		f^\pm_2 - f^\pm &= \Imq \bigg( (\pm \epsilon - f) (f_2^* - f^*) (\pm \epsilon + f_2) \bigg)	
		\end{aligned}
	\end{equation}
	integrate to two discrete nets $f^\pm: \Z^2 \to \R^3$, i.e., 
	\begin{equation*}
		\begin{aligned}
			\Imq \bigg( (\pm \epsilon - f) (f_1^* - f^*) (\pm\epsilon + f_1) + (\pm \epsilon - f_1) (f_{12}^* - f_1^*) (\pm \epsilon + f_{12}) \\- (\pm \epsilon - f_2) (f_{12}^* - f_2^*) (\pm \epsilon + f_{12}) - (\pm \epsilon - f) (f_{2}^* - f^*) (\pm \epsilon + f_{2}) \bigg) = 0.
		\end{aligned}
	\end{equation*}
\end{proposition}
\begin{definition}
	The nets $f^\pm: \Z^2 \to \R^3$ form a \emph{discrete Bonnet pair}.
\end{definition}
This construction was introduced as a remark in \cite{HSFW17}, alongside a theory of first and second fundamental forms for discrete nets. The local theory of discrete Bonnet pairs is thoroughly investigated in \cite{HSFW24} as the main application of a conformal theory for discrete nets immersed in $\R^3$ using quaternions. In particular, a characterization of discrete Bonnet pairs is given in terms of special coordinates (equivalent to the normalized form \eqref{eq:normalizationQ} of the Hopf differential), and a geometric understanding of the mean curvature preserving isometry between the pair of nets is described. 

To discover numerical examples of compact discrete Bonnet pairs, we worked on an extremely coarse $\Z/5\Z \times \Z / 7\Z$ torus. We implemented an optimization algorithm in \emph{Mathematica}~\cite{Mathematica} that searched for a map $f: \Z/5\Z \times \Z / 7\Z \to \R^3$ such that
\begin{enumerate}
	\item $f$ is a discrete isothermic torus, i.e., each quad satisfies \eqref{eq:discreteIsothermicCrossRatioMinusOne},
	\item the integrated net $f^\pm$ are tori, i.e., $f^\pm (5,m) = f^\pm (0,m), m \in \Z / 7\Z$ and $f^\pm (n,7) = f^\pm (n,0), n \in \Z / 5\Z$.
\end{enumerate}

A remarkable feature of the structure preserving discretization is that the discrete theory also appears to well-approximate the smooth theory in a continuum limit. Sampling the fundamental piece of a smooth isothermic torus that gives rise to smooth Bonnet tori, and then optimizing for it to be the fundamental piece of a discrete isothermic torus leading to discrete Bonnet tori, barely moves the vertices. The surfaces are visually indistinguishable.

\section{Concluding remarks}
\label{sec:conclusion}
We establish that the metric and mean curvature do not determine a unique compact surface, by providing the first set of examples of compact Bonnet pairs. Moreover, we prove that a real analytic metric does not determine a unique compact immersion. It was unexpected that the construction we found has a functional freedom in it. There appears to be a lot left to understand and explore. We state some open questions to stimulate further research and the development of new techniques. 

\begin{question}[Embedded examples]
	Do there exist compact Bonnet tori where both surfaces are embedded?
\end{question}

\begin{question}[Tori classification: No planarity assumption]
	Do there exist compact Bonnet pairs from isothermic tori without one family of planar (or spherical) curvature lines? More generally, can one classify all compact Bonnet pairs that are tori?
\end{question}

\begin{question}[Higher genus examples] 
	Do there exist compact Bonnet pairs of higher genus?
\end{question}

\begin{question}[Constant mean curvature examples]
	Do there exist compact Bonnet pairs with constant mean curvature? In other words, does the associated family of a compact constant mean curvature surface contain at least one other compact surface that is not congruent to the original? (Note that the restriction to at most two compact immersions does not apply in the constant mean curvature setting.)
\end{question}
\appendix

\section{Explicit formulas for Bonnet pair cylinders from an isothermic surface with one generic family of closed planar curvature lines}
\label{sec:explicit-bonnet-formulas}

\paragraph{Explicit formula version of Theorem~\ref{thm:BPFromPlanarFormulas}}:
\begin{theorem}
	Let $f(u,v) = \Phi(v)^{-1} \gamma(u,\w(v)) \qj \Phi(v)$ be an isothermic cylinder with one generic family of closed planar curvature lines as in Theorem~\ref{thm:planarIsothermicCylinderFormulas}. For each $\epsilon \in \R$, the resulting Bonnet pair surfaces $f^\pm(u,v)$ are real analytic cylinders with translational periods in $v$ that are equal up to sign. Their immersion formulas are:
	\begin{align}
		\begin{aligned}
			\label{eq:fPlusMinusBasicStructure}
			f^{\pm}(u,v) = &\overbrace{R(\omega)^2f(\pi-2\omega + u, v) - \epsilon^2 f(\pi - u, v)}^{\Apart{} = (f^{-1})^* + \epsilon^2 f^*} \pm \underbrace{2\epsilon \left( \Phi^{-1}(v) \hat B(u,\w(v)) \qi \Phi(v)  + \tilde B(v) \right)}_{\Bpart{} = 2 \epsilon \int \Imq ( df^* f )},
		\end{aligned}
	\end{align}
	where 
$		R(\omega) = \frac{2 \vartheta _2(\omega)^2}{\vartheta_1^{\prime}(0) \vartheta _1(2 \omega)}$~\eqref{eq:radiusOmega}, $\hat B(u,\w(v))$ is a real analytic real-valued function that is $2\pi$-periodic in $u$, and $\tilde B(v)$ is a real analytic $\R^3$-valued function that depends only on $v$.
	
	Explicitly,
	\begin{align}
		\hat B(u,\w) &= R(\omega)^2 \frac{\vartheta_1(2\omega)}{\vartheta_1'(0)} \left( \frac{\vartheta_2''(\omega)}{\vartheta_2(\omega)}\frac{\w}{2} - \Imc \frac{\vartheta_2'(\frac{u+\ci \w - \omega}{2})}{\vartheta_2(\frac{u+\ci \w - \omega}{2})} \right) \label{eq:BpartUIntegral}
	\end{align}
	and $\tilde B(v)$  is determined by
	\begin{align}
		\label{eq:tildeBPrimeExplicit}
		\tilde B'(v) &=  \Phi^{-1}(v) \left( R(\omega) \sqrt{1 - \w'(v)^2} \, \tilde b(\w(v)) \qk \right) \Phi(v),
	\end{align}
	where the complex-valued function $\tilde b(\w)$ is
	\begin{align}
		\tilde b(\w) &= \frac{2\vartheta_2(\omega)}{\vartheta_1'(0)}\frac{\vartheta_2(\ci \w - \omega)}{\vartheta_1(\ci \w)} \left( \frac{\vartheta_2''(\omega)}{\vartheta_2(\omega)}\frac{\w}{2} - \Imc \frac{\vartheta_2'(\frac{\ci \w}{2})}{\vartheta_2(\frac{\ci \w}{2})} \right) - \ci \frac{\vartheta_1(\frac{\ci \w}{2} -\omega)^2}{\vartheta_2(\frac{\ci \w}{2})^2}. \label{eq:littleBTilde}
	\end{align}
\end{theorem}

\begin{proof}
		We first derive the structure of the immersion formulas~\eqref{eq:fPlusMinusBasicStructure} and then compute $\hat B(u,\w(v))$ and $\tilde B(v)$.
		\begin{itemize}
			\item
			{\bf Structure of immersion formulas}.
			For the $\Apart{}$ part we write $\Apart{} = \int -f df^* f +\epsilon^2 df^*= (f^{-1})^*+\epsilon^2 f^*$ and then use the inversion and dual formulas, \eqref{eq:fInvUShift} and \eqref{eq:fDualUShift}, respectively.
			\begin{align*}
				(f(u,v)^{-1})^* = -R(\omega)^2 f(2\omega - u, v)^* = R(\omega)^2 f(\pi - 2\omega + u, v).
			\end{align*}
			
			To derive the structure in the $\Bpart{}$ part we use the isothermic parametrization \eqref{eq:fBasicStructure} for $f(u,v)$. We set $\int \Imq ( df^* f ) = \Phi^{-1} B \Phi$ for some $\R^3$-valued function $B(u,v)$. Its partial derivatives in terms of $u$ and $v$ satisfy
			\begin{align}
				(\Phi^{-1}(v) B(u,v) \Phi(v))_u &= \Imq((f^{*})_u f), \label{eq:BpartUDerivative}\\
				(\Phi^{-1}(v) B(u,v) \Phi(v))_v &= \Imq((f^{*})_v f). \label{eq:BpartVDerivative}
			\end{align}
			As the frame $\Phi(v)$ is independent of $u$, the left side of \eqref{eq:BpartUDerivative} is $\Phi^{-1} B_u \Phi$. The right side is computed with \eqref{eq:fBasicStructure} and $(f^*)_u=\frac{f_u}{|f_u|^2}$, so that  \eqref{eq:BpartUDerivative} is equivalent to
			\begin{align}
				\label{eq:PhiInvBPhiuDerivative}
				\Phi^{-1} B_u \Phi = \Phi^{-1} \Imc\left(\gamma_u^{-1}\gamma \right) \qi \Phi.
			\end{align}
			Since $\gamma$ is a function of $u$ and $\w(v)$ we define the real-valued function
			\begin{align}
				\hat B(u,\w(v)) &= \int_0^u \Imc\left(\gamma_u(\tilde u,\w(v))^{-1}\gamma(\tilde u,\w(v)) \right) d\tilde u.
				\label{eq:BpartUIntegralInGammas}
			\end{align}
			
			Therefore, integrating \eqref{eq:PhiInvBPhiuDerivative} with respect to $u$ gives
			\begin{align}
				\label{eq:PhiInvBPhiIsHatBPlusTildeB}
				\Phi^{-1}(v) B(u,v) \Phi(v) &= \Phi^{-1} \hat B (u,\w(v)) \qi \Phi + \tilde B(v),
			\end{align}
			where $\tilde B(v)$ is an $\R^3$-valued integration function that depends only on $v$. Note that $\tilde B(v)$ is not arbitrary. Since the right hand side of \eqref{eq:BpartVDerivative} and $\hat B(u,\w(v))$ are known, we differentiate \eqref{eq:PhiInvBPhiIsHatBPlusTildeB} with respect to $v$, evaluate at any $u = u_0$, and find the following ODE for $\tilde B(v)$.
			\begin{align}
				\label{eq:tildeBPrimeBasicStructure}
				\tilde B'(v) = \Imq \left( (f^*)_v(u_0,v) f(u_0,v) \right) - \left( \Phi^{-1}(v) \hat B(u_0,\w(v)) \qi \Phi(v) \right)_v.
			\end{align}
			
			\item {\bf Deriving the formula for $\hat B(u,\w(v))$~\eqref{eq:BpartUIntegral}}. From~\eqref{eq:BpartUIntegralInGammas} we have that $$\hat B(u,\w(v)) = \int_0^u \Imc\left(\gamma_u(\tilde u,\w(v))^{-1}\gamma(\tilde u,\w(v)) \right) d\tilde u.$$ Now, using \eqref{eq:gammaClosedCurves} and \eqref{eq:gammaUClosedCurves} we find
			\begin{align*}
				\gamma_u^{-1}\gamma &= \left(\frac{\vartheta _1\left(\frac{1}{2} (u+\ci\w+\omega)\right)}{\vartheta _2\left(\frac{1}{2} (u+\ci\w -\omega)\right)}\right)^2 \frac{2 \vartheta _2(\omega)^2}{\vartheta_1^{\prime}(0) \vartheta _1(2 \omega)} \frac{\vartheta_1\left(\frac{1}{2} (u+\ci \w-3 \omega)\right)}{\vartheta_1\left(\frac{1}{2} (u+\ci\w+\omega)\right)} \\
				& = \frac{2 \vartheta _2(\omega)^2}{\vartheta_1^{\prime}(0) \vartheta _1(2 \omega)} \frac{\vartheta _1\left(\frac{1}{2} (u+\ci\w+\omega)\right) \vartheta_1\left(\frac{1}{2} (u+\ci \w-3 \omega)\right)}{\vartheta _2\left(\frac{1}{2} (u+\ci\w -\omega)\right)^2}.
			\end{align*}
			This is an elliptic function in $z = \frac{u + \ci \w}{2}$ with a fundamental set of zeroes $\alpha_1 = \frac{-\omega}{2}, \alpha_2 = \frac{3\omega}{2}$ and poles $\beta_1 = \frac{\omega + \pi}{2}, \beta_2 = \frac{\omega-\pi}{2}$ satisfying $\alpha_1 - \beta_1 + \alpha_2 - \beta_2 = 0$. Using the method outlined in Section 21.5 of Whittaker and Watson~\cite{whittaker_watson_1996} to rewrite an elliptic function as a sum of Jacobi zeta functions, their derivatives, and constants, we find with $z = \frac{u + \ci \w}{2}$ that
			\begin{align}
				\gamma_u^{-1}\gamma = \frac{\partial}{\partial z} \left( \frac{2 \vartheta _2(\omega)^4}{\vartheta_1^{\prime}(0)^3 \vartheta _1(2 \omega)}
				\left ( z \frac{\vartheta_2'(\omega)^2 + \vartheta_2(\omega) \vartheta_2''(\omega)}{\vartheta_2(\omega)^2} - \frac{\vartheta_2'(z-\frac{\omega}{2})}{\vartheta_2(z-\frac{\omega}{2})} \right ) \right).
				\label{eq:BpartUExplicitZDerivative}
			\end{align}
			Now, note $R(\omega) = \frac{2 \vartheta _2(\omega)^2}{\vartheta_1^{\prime}(0) \vartheta _1(2 \omega)}$~\eqref{eq:radiusOmega} and from Theorem~\ref{thm:planarIsothermicCylinderFormulas} that $\omega$ is \emph{critical}, so it satisfies $\vartheta_2'(\omega) = 0$. Putting these facts together with \eqref{eq:BpartUExplicitZDerivative} gives
			\begin{align*}
				\int (\gamma_u^{-1}\gamma) du = R(\omega)^2 \frac{\vartheta_1(2\omega)}{\vartheta_1'(0)} \left( \frac{\vartheta_2''(\omega)}{\vartheta_2(\omega)}\frac{u + \ci \w(v)}{2} - \frac{\vartheta_2'(\frac{u+\ci \w(v) - \omega}{2})}{\vartheta_2(\frac{u+\ci \w(v) - \omega}{2})} \right) + C,
			\end{align*}
			for a constant $C \in \C$. For a rhombic lattice, $\overline{\vartheta_i(z)} = e^{-\ci \pi / 4} \vartheta_i(\bar z)$, so the constants given in terms of ratios of theta functions are real-valued. Thus, we have proven that $\hat B(u,\w(v))$, defined by \eqref{eq:BpartUIntegralInGammas} as the complex imaginary part of the above expression, is real-valued, $2\pi$-periodic in $u$ and given as in \eqref{eq:BpartUIntegral}. Without loss of generality we put the integration constant $C$ into the definition of $\tilde B(v)$.
			
			\item	{\bf Deriving the formulas for $\tilde B'(v)$~\eqref{eq:tildeBPrimeExplicit} and $\tilde b(\w)$~\eqref{eq:littleBTilde}}.
			We compute $(f^*)_v f$ and $(\Phi^{-1} \hat B \qi \Phi)_v$ and use \eqref{eq:tildeBPrimeBasicStructure}. Note that \eqref{eq:fvIsothermicPlanarClosed} implies
			\begin{align*}
				f_v &= \Phi^{-1} \left( \sqrt{1-(\w')^2} \left | \gamma_u \right | \qi + \w' \gamma_u \qk \right) \Phi.
			\end{align*}
			Combine this with $(f^*)_v = - \frac{f_v}{|\gamma_u|^2}$ and $f = \Phi^{-1} \gamma \qj \Phi$ to find 
			\begin{align*}
				(f^*)_v f
				& = - \Phi^{-1} \left( \sqrt{1 -  (\w')^2} \frac{\gamma}{| \gamma_u|} \qk + \w' \left(- \overline{(\gamma_u)^{-1} \gamma} \right) \qi \right) \Phi.
			\end{align*}
			Now, because $\Imq \left( \left(- \overline{(\gamma_u)^{-1} \gamma} \right) \qi \right) = - \Imc (\ci (\gamma_u)^{-1} \gamma) \qi$ we have
			\begin{align*}
				\Imq \left( (f^*)_v f \right) = - \Phi^{-1} \left( \sqrt{1 -  (\w')^2} \frac{\gamma}{| \gamma_u|} \qk - \w' \Imc (\ci (\gamma_u)^{-1} \gamma) \qi \right) \Phi.
			\end{align*}
			Since $\hat B(u,\w(v))$ is the complex imaginary part of the integral with respect to $u$ of the function $(\gamma_u)^{-1}\gamma$ that is meromorphic in $u + \ci \w$, we see that
			\begin{align*}
				(\hat B)_v = \w' (\hat B)_{\w} 
				= \w' \Imc \left( \ci (\gamma_u)^{-1} \gamma \right).
			\end{align*}
			Thus,
			\begin{align*}
				\Imq \left( (f^*)_v f \right) = - \Phi^{-1} (\sqrt{1 -  (\w')^2}|\gamma_u|^{-1}\gamma) \qk \Phi + \Phi^{-1} (\hat B )_v \qi \Phi.
			\end{align*}
			On the other hand, $\hat B$ is real-valued and from \eqref{eq:phiPrimePhiInverse} we know that $\Phi' \Phi^{-1}$ lies in the $\qj,\qk$-plane, so
			\begin{align*}
				(\Phi^{-1}(v) \hat B(u, \w(v)) \qi \Phi(v))_v 
				& = \Phi^{-1} (2 \hat B \qi \Phi' \Phi^{-1}) \Phi + \Phi^{-1} (\hat B )_v \qi \Phi.
			\end{align*}
			Substitution into \eqref{eq:tildeBPrimeBasicStructure} implies
			\begin{align*}
				\tilde B ' (v) &= - \Phi^{-1} (\sqrt{1 -  (\w')^2}|\gamma_u|^{-1}\gamma \qk + 2 \hat B \qi \Phi' \Phi^{-1}) \Phi.
			\end{align*}
			Now plug in the appropriate expressions for $\gamma, \gamma_u, \Phi' \Phi^{-1},\hat B$ and simplify to find
			\begin{align}
				\tilde B'(v) &= R(\omega) \sqrt{1 - \w'(v)^2} \Phi^{-1}(v) \tilde b(u, \w(v)) \qk \Phi(v), \\
					\tilde b(u,\w) &= \frac{2 \vartheta_2(\omega)}{\vartheta_1'(0)} \frac{\vartheta_2(\ci \w - \omega)}{\vartheta_1(\ci \w)}  \left( \frac{\vartheta_2''(\omega)}{\vartheta_2(\omega)}\frac{\w}{2} - \Imc \frac{\vartheta_2'(\frac{u+\ci \w - \omega}{2})}{\vartheta_2(\frac{u+\ci \w - \omega}{2})} \right) + \ci \frac{\vartheta_1(\frac{u+\ci \w - 3\omega}{2})\vartheta_1(\frac{u-\ci \w + \omega}{2})}{\vartheta_2(\frac{u+\ci \w - \omega}{2})\vartheta_2(\frac{u-\ci \w - \omega}{2})}.
			\end{align}
			As $\tilde B(v)$ only depends on $v$, $\tilde b(u,\w(v))$ must be independent of $u$. Set $u = \omega$ to get $\tilde b(\w)$ as in \eqref{eq:littleBTilde}.
		\end{itemize}
	\end{proof}

\footnotesize
\bibliographystyle{abbrv}
\bibliography{short-compact-bonnet}

\end{document}